\documentclass[reqno,11pt]{amsart}
\usepackage{geometry}
\usepackage{mathtools,amssymb,amsthm,mathrsfs,color,lineno,paralist,graphicx,float}
\usepackage[colorlinks,
linkcolor=red,
anchorcolor=green,
citecolor=blue,
]{hyperref}
\usepackage{comment}
\usepackage{enumitem}

\usepackage[T1]{fontenc}
\usepackage[utf8]{inputenc}
\usepackage[english]{babel}
\usepackage {graphicx,overpic}
\usepackage{subfig}
\usepackage{float}
\usepackage[export]{adjustbox}
\usepackage{tikz}
\usepackage{relsize}

\usepackage{calc}

\newcommand{\gentau}{\mathord{\scalebox{0.8}{$\mathfrak{T}$}}}

\newcommand{\R}{{\mathbb R}}

\newcommand{\T}{{\mathbb T}}

\newtheorem{theorem}{Theorem}[section]
\newtheorem{lemma}[theorem]{Lemma}%[section]
\newtheorem{prop}[theorem]{Proposition}%[section]
\newtheorem{corollary}[theorem]{Corollary}%[section]
\newtheorem{proposition}[theorem]{Proposition}
\newtheorem{theoalph}{Theorem}

\theoremstyle{definition}
\newtheorem{remark}[theorem]{Remark}%[section]
\newtheorem{definition}[theorem]{Definition}

\def\uns{\mathrm{u}}
\def\sta{\mathrm{s}}

\def\BB{B}

\newcommand{\difknu}[2]{\mathcal{C}^{#1}(\mathbb{A}_{#2})}
\newcommand{\der}[1]{\dot{{#1}}}
\newcommand{\dertwo}[1]{\ddot{{#1}}}
\newcommand{\norm}[2]{\mathcal{C}^{{#1}}({#2})}

\renewcommand{\ss}{\mathbf{s}}
\newcommand{\vp}{\varphi}

\newcommand{\gna}{\gamma}
\newcommand{\dergna}{\der{\gamma}}
\newcommand{\dertwogna}{\dertwo{\gamma}}

\newcommand{\lna}{\lambda}
\newcommand{\derlna}{\der{\lambda}}
\newcommand{\dertwolna}{\dertwo{\lambda}}

\newcommand{\dgnalna}{\gamma_\lambda}

\newcommand{\nna}{\mathbf{n}}

\newcommand{\ga}{{\beta}}
\newcommand{\derga}{\der{\beta}}
\newcommand{\dertwoga}{\dertwo{\beta}}

\newcommand{\defga}{\beta_\varepsilon}
\newcommand{\derdefga}{\der{\beta}_\varepsilon}
\newcommand{\dertwodefga}{\dertwo{\beta}_\varepsilon}

\newcommand{\dgalac}{\beta_\eta}
\newcommand{\derdgalac}{\der{\beta}_\eta}

\newcommand{\lac}{\eta}
\newcommand{\derlac}{\der{\eta}}
\newcommand{\dertwolac}{\dertwo{\eta}}

\newcommand{\narc}{n}

\newcommand{\etana}{\xi}

\newcommand{\etarc}{\zeta}
\newcommand{\deretarc}{\der{\zeta}}
\newcommand{\dertwoetarc}{\dertwo{\zeta}}

\newcommand{\co}{C_0}

\renewcommand{\diamond}{\ast}

\newcommand{\ttt}{t}

\newcommand{\g}{\psi}

\newcommand{\Bl}[1]{\BB_{#1}}
\newcommand{\Bf}[2]{\mathbf{B}^{\nu}_{#1}(#2)}
  \RequirePackage{textcmds}\relax

\ProvideTextCommandDefault{\cprime}{\tprime} 
\begin{document}

\title[]{Chaoticity of generic analytic convex billiards}

\author{Inmaculada Baldom\'a}
\address{Departament de Matem\`atiques, Universitat Polit\`ecnica de Catalunya (UPC), IMTECH (UPC), Centre de Recerca Matem\`atica (CRM), Barcelona, Spain.}
\email{immaculada.baldoma@upc.edu}

\author{Anna Florio}
\address{\textsc{Ceremade}-Universit\'e Paris Dauphine-PSL\\
75775 Paris, France.}
\email{florio@ceremade.dauphine.fr}

\author{Martin Leguil}
\address{\'Ecole polytechnique, CMLS\\
	Route de Saclay, 91128 Palaiseau Cedex, France.}
\email{martin.leguil@polytechnique.edu}

 \author{Tere M.-Seara}
 \address{Departament de Matem\`atiques, Universitat Polit\`ecnica de Catalunya (UPC), IMTECH (UPC), Centre de Recerca Matem\`atica (CRM), Barcelona, Spain.}
 \email{tere.m-seara@upc.edu}

\begin{abstract}
We show that a generic analytic strongly convex billiard is ``maximally chaotic'' in the sense that, for every rational number $\frac{p}{q} \in \mathbb{Q} \cap (0,1)$, all intersections between the stable and unstable manifolds of maximizing periodic orbits with rotation number $\frac{p}{q}$ are transverse. 
\end{abstract}
\maketitle

\selectlanguage{english}
\tableofcontents

\section{Introduction}
A~\emph{planar billiard} is a dynamical system that describes the motion of a massless point moving within a domain of $\mathbb{R}^2$ along a straight line until it hits a boundary component. Upon collision, the angle of reflection equals the angle of incidence.
Planar billiards have attracted significant attention from mathematicians due to their fascinating dynamical properties, particularly because of the wide range of dynamical behaviors that emerge depending on the shape of the boundary. A paradigmatic example is the class of~\emph{strongly convex billiards}, first introduced by G.~D.~Birkhoff in~\cite{Birkhoff}. These are convex domains of $\mathbb{R}^2$ whose boundary is a smooth curve with nonvanishing curvature. Such billiards exhibit mixed dynamical behavior: near the boundary, Lazutkin's celebrated work~\cite{Lazutkin} demonstrates that the billiard map admits a positive-measure set of essential (i.e., homotopically non-trivial) invariant curves, giving rise to quasi-periodic dynamics. The existence of elliptic islands near certain periodic orbits has been also investigated, see, e.g.,~\cite{CarneiroOliffsonPC}. Additionally, hyperbolic sets can arise from transverse intersections between the invariant manifolds of hyperbolic periodic orbits.   
As emphasized above, the dynamical features can change drastically when considering other classes of billiards. For instance, Lazutkin's result fails for convex planar billiards with flat arcs, where no invariant curve prevents diffusive behavior near the boundary. A notable example is that of~\emph{Bunimovich stadia} (see~\cite{Bunimovich}), which are ergodic and exhibit non-uniformly hyperbolic dynamics. 
Furthermore, billiards within convex polygons display parabolic motion (see, e.g.,~\cite{Gutkin}), whereas dispersing billiards --- in which particles bounce off strongly convex obstacles --- serve as examples of uniformly hyperbolic systems (see, e.g.,~\cite{Sinai}).

In the present work, we focus on~\emph{strongly convex planar domains}. One fascinating line of research dates back to Birkhoff, who asked whether ellipses are the only examples of~\emph{integrable} convex billiards --- in the sense that an open region of phase space is foliated by essential invariant curves. Recent breakthroughs in this direction have been achieved in various works~\cite{AvilaKalSimoi,  huang2018nearly, KalSorr, BialyMironov}, demonstrating that, under certain conditions (such as when the domain exhibits central or axial symmetries), the presence of ``sufficiently many invariant curves'' forces the domain to be elliptic. Let us also mention the very recent work of Koval~\cite{Koval}, who proves a local version of the strong Birkhoff conjecture for most ellipses. Specifically, for almost every eccentricity $e \in (0,1)$, domains that are sufficiently close to an ellipse of eccentricity $e$ and integrable near the boundary are shown to be elliptic.

Complementing these results, another line of research explores the idea that a ``typical'' convex billiard should exhibit chaotic behavior incompatible with integrability. This chaos arises, in particular, due to the presence of transverse intersections between stable and unstable manifolds of hyperbolic periodic orbits. Donnay~\cite{Donnay2} proves the existence of smooth curves, which are $\mathcal{C}^2$-close to ellipses, whose billiard maps admit transverse homoclinic points. This result was later generalized in the works of Dias Carneiro, Kamphorst, and Pinto de Carvalho~\cite{PintoC}, as well as Xia and Zhang~\cite{XiaZhang}: they show that for a $\mathcal{C}^r$-generic billiard table, with $r\in \mathbb{N}_{\geq 2}\cup \{\infty\}$, all hyperbolic periodic orbits exhibit transverse homoclinic intersections. In~\cite{PinRam}, the authors demonstrate that a generic perturbation of a convex billiard table destroys any resonant invariant essential curve. A variational approach is developed in~\cite{Cheng} to study the relationship between smooth billiard tables and positive topological entropy. 
Higher dimensional results are also present in the literature. Delshams, Fedorov, and Ramírez-Ros~\cite{DelFedRam} perturb an ellipsoid in $\mathbb{R}^3$ to break the separatrices. The existence of a $\mathcal{C}^2$-open and dense set of convex bodies in $\mathbb{R}^d$ ($d > 2$) whose billiard map exhibits a hyperbolic set --- and thus has positive topological entropy --- is established in~\cite{entropyhigherdim}.

Let us emphasize that similar results have been obtained in the general framework of twist maps of the two-dimensional annulus. In~\cite{Forni}, the author proves that invariant circles of area-preserving twist maps can be destroyed under small analytic perturbations. Thus, using the result in~\cite{Angen}, it can be deduced that a generic area-preserving twist map has positive entropy.
However, they cannot be directly applied to billiards, where the class of allowed perturbations is more restrictive. Indeed, perturbing the billiard map within the class of twist maps of the annulus could result in leaving the framework of billiard maps. Instead, one must first perturb the boundary of the domain, which in turn perturbs the billiard map. 
In particular, one of the main difficulties is that local perturbations of the boundary lead to non-local perturbations of the billiard map. Specifically, any modification of the table translates into a perturbation of an open set of fibers in phase space, affecting all orbits that bounce off the perturbed segment of the boundary, regardless of their bouncing angles.

The aforementioned results obtained for convex billiards successfully address such fibered perturbations in the $\mathcal{C}^r$ category, where $r \in \mathbb{N} \cup \{\infty\}$. The present work pursues a two-fold objective:

\begin{itemize}
\item Extending these results to the class of strongly convex billiards with \emph{analytic boundaries}, which are even more rigid. Indeed, by analytic continuation, any localized perturbation of the boundary will globally affect the dynamics.
\item Obtaining stronger conclusions by demonstrating the genericity of ``maximally chaotic'' convex billiards, for which all maximizing periodic orbits are hyperbolic and all intersections between their stable and unstable manifolds are transverse (in the spirit of the Kupka-Smale theorem). In particular, such behavior is incompatible with integrability: by the twist property, any annular region will necessarily contain hyperbolic sets.
\end{itemize}

Informally, our main result can be stated as follows:
\begin{theoalph}
    For a generic analytic billiard table, the corresponding billiard map is such that, for every $\frac p q \in\mathbb{Q}\cap(0,1)$, there exists a unique maximizing hyperbolic orbit of rotation number $\frac p q$, which admits homoclinic intersections, all of which are transverse.
\end{theoalph}
In fact, the proof yields a stronger result: for any fixed $\frac{p}{q} \in \mathbb{Q} \cap (0,1)$ and any fixed $N\in \mathbb{N}$, any strongly convex analytic billiard table can be perturbed by a trigonometric polynomial so that the resulting table admits a unique maximizing hyperbolic periodic orbit of rotation number $\frac{p}{q}$, and all intersections between its local stable and unstable manifolds of size $N$ are transverse. 

In particular, we obtain the following stronger result for trigonometric polynomial billiard tables, answering a question by Berger: there exists an open and dense set of trigonometric polynomial billiard tables whose billiard map has positive topological entropy (see Corollary~\ref{coro entropy}).

As a consequence of our main result, we deduce that, for a generic analytic billiard table, there exists a family of horseshoes accumulating on the boundary. In particular, these horseshoes coexist in this region with the family of essential curves exhibiting quasi-periodic motion, as guaranteed by Lazutkin's result.

We also note that our main result has applications to questions related to the spectral rigidity of convex billiards. 
In~\cite{guankaloshin}, it is shown that for \( k \geq 3 \), the eigendata corresponding to Aubry-Mather periodic orbits of the billiard map can be recovered from the (maximal) marked length spectrum for a generic \( \mathcal{C}^k \) strictly convex domain. As a consequence of the present work, the same conclusion extends to the analytic category. Let us also mention that in ongoing work by the third author, together with Vadim Kaloshin and Ke Zhang, the authors investigate the rigidity of the marked length spectrum for generic strongly convex billiards with analytic boundaries. Here, the marked length spectrum is a function that associates to each rational number $\frac{p}{q} \in \mathbb{Q} \cap (0,1)$ the maximal perimeter of periodic orbits with rotation number $\frac{p}{q}$. This work focuses on a specific sequence of periodic orbits accumulating on a transverse homoclinic orbit associated with the maximizing two-periodic orbit. In particular, our present work verifies that the necessary conditions for this study --- especially the transversality of all intersections between the stable and unstable manifolds of the maximizing two-periodic orbit --- are indeed generic. A key step is the ability to select the relevant periodic orbits solely by their rotation number, which is ensured by the generic conditions satisfied by the billiards.

To our knowledge, there is a scarcity of research addressing the generic properties of convex billiards with analytic boundaries. In~\cite{ClarkeTuraev}, the authors demonstrate that, assuming the existence of a hyperbolic periodic orbit with a transverse homoclinic orbit, a generic real-analytic convex domain in $\mathbb{R}^d$, $d \geq 3$, exhibits diffusive trajectories. These trajectories begin near the tangent direction and reach almost vanishing angles of reflection. 

A related class of dynamical systems with constrained perturbations is that of geodesic flows. Contreras and Mazzucchelli~\cite{contreras2025closed} demonstrated that, given a closed Riemannian manifold $(M, g_0)$ of dimension at least two with a minimal closed geodesic $\gamma$, the metric $g_0$ can be perturbed to a $\mathcal{C}^\infty$-close metric $g$ such that $\gamma$ becomes a hyperbolic closed geodesic of $g$ with a transverse homoclinic orbit. 
Earlier, Contreras~\cite{Contreras} showed that, for closed manifolds of dimension at least two on which a $\mathcal{C}^2$-generic Riemannian metric admits infinitely many closed geodesics, a $\mathcal{C}^2$-small perturbation of the metric yields a transverse homoclinic orbit to a hyperbolic closed geodesic. Let us also mention the work of Clarke~\cite{Clarke} 
in the analytic category, where he proved that if the geodesic flow on a real-analytic closed hypersurface $M \subset \mathbb{R}^n$ ($n \geq 3$), with respect to the Euclidean metric, admits an elliptic periodic orbit, then $M$ can be perturbed analytically so that, generically, the geodesic flow on the perturbed hypersurface exhibits a hyperbolic periodic orbit with a transverse homoclinic point.

To address the analytic topology, we draw on ideas from the seminal work of Zehnder~\cite{Zehnder} on analytic symplectic diffeomorphisms of the disk near an elliptic fixed point. Specifically, Zehnder proved that a Baire-generic diffeomorphism in this class exhibits a hyperbolic periodic point with transverse homoclinic intersections in any neighborhood of the fixed point. 
Some years later, Genecand~\cite{Genecand} provided an alternative proof of the same result using a twist condition and the theory of Aubry-Mather sets, which guarantees the existence of homoclinic points (though not necessarily transverse). Once again, a second perturbation is usually required to ensure that the homoclinic point is transverse. Let us also mention that a different approach to obtaining similar results within the analytic topology can be found in the work of Broer and Tangerman~\cite{BroTan}.

Another interesting approach is to consider the case of homoclinic tangencies. In the context of generic symplectic diffeomorphisms of the disk with an elliptic fixed point, Gelfreich and Turaev~\cite{GelTur} prove that any such diffeomorphism can be perturbed --- even within the analytic topology --- to produce a quadratic homoclinic tangency. In particular, a generic diffeomorphism is~\emph{universal}, meaning that the closure of all its renormalizations contains any disk diffeomorphism. This universality is a consequence of the genericity of quadratic homoclinic tangencies, combined with the universality theory developed by Gonchenko, Shilnikov, and Turaev in~\cite{GST}. Notably, these genericity results hold in the analytic topology.
Let us highlight that for convex billiards with smooth boundary, universality results --- reminiscent of the Gonchenko-Shilnikov-Turaev theory --- were recently obtained by Callis~\cite{Callis}.  

Let us further mention that Zehnder's result on the abundance of transverse homoclinic points has been also adapted to the framework of $C^\infty$ Reeb flows on $3$-dimensional manifolds in~\cite{CDHR24}. By leveraging the powerful tool of broken book decompositions, the dynamics of the flow can generically be studied as return diffeomorphisms on Birkhoff sections. 
They build on earlier results by Le Calvez and Sambarino~\cite{LeCalvSam}, obtained in the broader framework of surface diffeomorphisms, concerning transverse homoclinic intersections of hyperbolic periodic points.\\

The paper is organized as follows. In Section~\ref{sec:main results}, we  give some definitions needed to stablish the main results: Theorem \ref{main_thm bis_AM}, which is a direct consequence of Theorem \ref{main_thm ter_AM}, which claims that a  certain set of analytic billiards is open and dense. The (easy) proof of the openness property is given right after the theorem, but the density, which is the main result in this paper,  is stated in Theorem \ref{main thm 2 bis}. We also provide Corollary \ref{coro entropy}, which is an easy consequence of both theorems.
The rest of the paper is devoted to prove Theorem \ref{main thm 2 bis}.
In Section  \ref{sec:proof of main thm 2 bis} define the fundamental notions and we give the steps and results which lead to the proof of Theorem \ref{main thm 2 bis} and the rest of the sections of the paper are devoted to the (more technical) proofs of these results.

\vspace{10pt}

\noindent \textbf{Aknowledgements}

\noindent A.F. is partially supported by the ANR project CoSyDy (ANR-CE40-0014) and the ANR project GALS (ANR-23-CE40-0001). M.L. is partially supported by the ANR project CoSyDy (ANR-CE40-0014), the ANR project Padawan (ANR-21-CE40-0012-01), the ANR project NO-LIMIT (ANR-25-CE40-1380-01) and by the LESET Math-AMSUD project. I.B. and T.M.S. have been partially supported by the grant PID-2021-122954NB-100 and PID2024-158570NB-I00 funded by
MCIN/AEI and “ERDF A way of making Europe”.   
This work is also supported by the Spanish State Research Agency, through the Severo Ochoa and Mar\'ia de Maeztu Program for Centers and Units of Excellence in R\&D (CEX2020-001084-M).

\section{Main results}\label{sec:main results}
\subsection{Topology}\label{subsec:topology}
For any $z\in\mathbb{C}$, we denote by $\Im z$ its imaginary part. Let $\mathbb{T}:=\mathbb{R}/\mathbb{Z}$. For any $r>0$, we consider the complex strip 
\[
\mathbb{T}_{r}:=\{z\in \mathbb{C} :\ \vert \Im z\vert < r\}/\mathbb{Z}.
\]
The set 
\[
\mathcal{C}_{r}^{\omega}(\mathbb{T},\mathbb{R}^2) = \{ \gna\colon \mathbb{T}_r\to \mathbb{C}^2\,:\, \text{real analytic and bounded} \},
\]
endowed with the norm 
\begin{equation*} 
\|\gna\|_{r}:=\sup_{z\in\mathbb{T}_{r}}\|\gna(z)\|,
\end{equation*} 
is a  Banach space; here, $\| (z_1,z_2)\| = \sqrt{|z_1|^2 + |z_2|^2}$ denotes the usual (complex) Euclidian norm on $\mathbb{C}^2$. We notice that $\mathcal{C}^{\omega}(\mathbb{T}, \mathbb{R}^2)=\cup_{r>0} \mathcal{C}_{r}^{\omega}(\mathbb{T}, \mathbb{R}^2)$ is the set of real analytic functions $\gna \colon \mathbb{T}\to \mathbb{R}^2$.  
    
We consider the open subset
\begin{align*}
\mathcal{B}_{r}^0:=
\{\gna\in \mathcal{C}^{\omega}_{r}(\mathbb{T},\mathbb{R}^2)\text{ such that } 
\gna\colon \T \hookrightarrow \mathbb{R}^2\text{ is an embedding}\}  
\end{align*}
and we note that for any $\gna \in \mathcal{B}_{r}^0$, $\gna(\mathbb{T})$ is a Jordan curve. Denote by $\Omega(\gna)$ the bounded connected component of $\mathbb{R}^2\setminus\gna(\mathbb{T})$.  
We then let 
\begin{equation}\label{def:mathcalBr}
\mathcal{B}_{r} = \{ \gna \in \mathcal{B}_{r}^0\text{ such that} \; \gna(\T) = \partial \Omega (\gna)\text{ is strongly convex}\}.
\end{equation}
In particular, for any $\gna \in \mathcal{B}_r$, $\Omega(\gna)$ is a convex analytic billiard table such that the curvature of $\partial \Omega (\gna)$ is everywhere non-zero.   
    
Moreover, $\mathcal{B}_r$ has the Baire property, as it is an open subset of the Banach space $(\mathcal{C}^{\omega}_r(\mathbb{T}, \mathbb{R}^2), \|\cdot \|_r)$. 
In other words, the countable intersection of open and dense subsets of $\mathcal{B}_r$ is a dense set. We say that a subset of $\mathcal{B}_r$ is Baire generic if it is a countable intersection of open and dense sets. 

\begin{remark}Observe that every strongly convex analytic billiard table $\Omega\subset \mathbb{R}^2$ can be seen as $\Omega(\gna)$, for some $r>0$ and some analytic function $\gna \in \mathcal{B}_r$. This representation is not unique, as $\gna$ could be replaced with $\gna\circ \sigma$, for some analytic diffeomorphism $\sigma \colon \T \to \T$. Moreover, since we are interested in dynamical properties of the billiard $\Omega(\gna)$, which are invariant under rigid motions of $\mathbb{R}^2$, we could as well identify it with $\Omega(\tau \circ \gna)$, for any $\tau$ which is the composition of a homothety and some isometry.
\end{remark}
    
\subsection{Statement of main results}
To state our main result Theorem~\ref{main_thm bis_AM}, we begin by fixing notation and recalling some fundamental concepts.

Given a billiard table $\Omega(\gna)$, for some $\gna\in \mathcal{C}^\omega_r(\T,\mathbb{R}^2)$, denote by $f=f_\gna\colon \mathbb{A}\to \mathbb{A}$ the corresponding analytic billiard map on the annulus 
\[
\mathbb{A}:=\T\times \left [-\frac{\pi}{2},\frac{\pi}{2} \right].
\]
Denote by $p_1\colon\mathbb{R}^2\to \mathbb{R}$ (resp. $p_2\colon\mathbb{R}^2\to \mathbb{R}$) the projection onto the first (resp. second) coordinate. 
Given a point $(\ss,\varphi)\in\mathbb{A}$, we say that it admits a rotation number if the following limit exists: 
\[
\lim_{m\to +\infty}\dfrac{p_1\circ F^m(S,\varphi)-p_1(S,\varphi)}{m}\, ,
\]
where $(S,\varphi)\in\mathbb{R}\times[-\frac{\pi}{2},\frac{\pi}{2}]$ is a lift of $(\ss,\varphi)$ and $F\colon \mathbb{R}\times[-\frac{\pi}{2},\frac{\pi}{2}]\to \mathbb{R}\times[-\frac{\pi}{2},\frac{\pi}{2}]$ is a lift of $f$. The limit exists for a lift if and only if it exists for every lift, and rotation numbers given by different lifts could differ only by integers. Thus, the rotation number of $(\ss,\varphi)\in \mathbb{A}$ with respect to $f$, when it exists, is well defined in $\mathbb{R}/\mathbb{Z}$.

For a given diffeomorphism $f\colon \mathbb{A} \to \mathbb{A}$, we say that $P\in \mathbb{A}$ is a hyperbolic $q-$periodic point, $q\in \mathbb{N}\backslash \{0\}$, if $f^q(P)=P$, $f^m(P)\neq P$ if $1\leq m \leq q-1$ and the eigenvalues of $Df^q(P)$ lie off the unit circle. We call its orbit $\mathcal{P}=\{f^j(P)\}_{j=0}^{q-1}$ a hyperbolic $q-$periodic orbit. 
   
A hyperbolic $q-$periodic point, $P$, has associated stable and unstable invariant manifolds 
\begin{align*}
W^{\sta}(P )= \{ (\ss,\varphi) \in \mathbb{A} : \|f^{m}(\ss,\varphi) - f^{m}(P )\| \to 0 \; \text{as }\,m \to +\infty\} \\
W^{\uns}(P )= \{ (\ss,\varphi) \in \mathbb{A} : \|f^{m}(\ss,\varphi) - f^{m}(P )\| \to 0 \; \text{as }\,m \to -\infty\}
\end{align*}
where the sign $+$ corresponds to the stable manifold and the sign $-$ to the unstable one. We also introduce the stable and unstable invariant manifolds of a hyperbolic $q-$ periodic orbit defined as the union of stable, resp. unstable manifolds of the points in $\mathcal{P}=\{P_j\}_{j=0}^{q-1}$:
\begin{equation} \label{defWsWu}
W^\sta(\mathcal{P}):=\bigcup_{j=0}^{q-1} W^\sta(P_j), \qquad     
W^\uns(\mathcal{P}):=\bigcup_{j=0}^{q-1} W^\uns(P_j).
\end{equation}  
 
\begin{definition}
Let $P\in \mathbb{A}$ be a hyperbolic periodic point of period $q \in \mathbb{N}\setminus \{0\}$ for a map $f:\mathbb{A}\to \mathbb{A}$. 
A point $Q\in \mathbb{A}$ is homoclinic to $P$ if $Q$ does not belong to the orbit of $P$ and it lies at the intersection of the stable and unstable manifolds of the orbit of $P$. 
A homoclinic point $Q$ is transverse if the stable and unstable manifolds of the orbit of $P$ intersect transversally at $Q$.

Analogously, if $\mathcal{P}=\{P_j\}_{j=0}^{q-1} \subset \mathbb{A}$ is a hyperbolic $q-$periodic orbit, we say that a point $Q$ is homoclinic to $\mathcal{P}$ if $Q\in \mathcal{W}^{\sta}(P_n) \cap \mathcal{W}^{\uns}(P_m)$, for some $n,m\in \{0,\cdots , q-1\}$. When the intersection is transverse, we say that the homoclinic point is transverse.  
\end{definition}

We can now state our main result. The notation refers to Section~\ref{subsec:topology}. See Section~\ref{Aubry Mather} for precise definitions related to Aubry-Mather theory. 
\begin{theorem}\label{main_thm bis_AM}
Let $r>0$. There exists a Baire generic set $\widetilde{\mathcal{B}}_r\subset \mathcal{B}_r$ such that for any $\gna \in \widetilde{\mathcal{B}}_r$, the billiard map $f_\gna$ associated to the billiard table $\Omega(\gna)$ satisfies the following properties for every $\frac{p}{q} \in \mathbb{Q}/\mathbb{Z}$:
\begin{itemize}
\item 
there exists a unique periodic orbit $\mathcal{P}=\mathcal{P}^{p/q}$ in the Aubry-Mather set of rotation number $\frac p q$, and it is hyperbolic;
\item 
all the homoclinic points to $\mathcal{P}$ are transverse\footnote{As recalled in Section~\ref{Aubry Mather} (see also Proposition \ref{prop:homoclinicpq}), Aubry-Mather theory guarantees that the set of homoclinic points to $\mathcal{P}^{p/q}$ is non-empty.};   
\item in particular, the Aubry-Mather set with rotation number $\frac{p}{q}$ is uniformly hyperbolic. 
\end{itemize}
\end{theorem}
    
In order to prove Theorem~\ref{main_thm bis_AM}, we find a countable collection of open and dense subsets satisfying suitable properties (see Theorem~\ref{main_thm ter_AM}). Let us first to introduce the notion of \textit{branch}. 
\begin{definition}\label{definition branch} 
Let $\mathcal{P}= \{P_j\}_{j=0}^{q-1}$ be a hyperbolic $q-$periodic orbit of a diffeomorphism $f: \mathbb{A} \to \mathbb{A}$. For $\diamond=\uns,\sta$, a \emph{branch}  $W$ of $W^{\diamond}(\mathcal{P})$ is a connected component of $W^{\diamond}(P_j)\setminus \{P_j\}$ for some $j \in \{0,\cdots,q-1\}$.      
     
Given a branch $W\subset W^{\diamond}(P_j)$ and a number $\varrho>0$, we denote by 
\begin{equation}\label{def:branch}
W_\varrho:=\{x\in W:d_W(P_j,x)<\varrho\}
\end{equation} 
the set of points $x\in W$ at distance at most $\varrho$ from $P_j$ with respect to the distance $d_W$ on $W$ induced by the Riemannian metric. 
\end{definition}
    
\begin{theorem}\label{main_thm ter_AM}
Let $r>0$, $\frac{p}{q} \in \mathbb{Q}/\mathbb{Z}$ and $N\in \mathbb{N}$. Consider the set $ {\mathcal{V}}^{p/q}_{r,N}\subset \mathcal{B}_r$ such that for any $\gna \in  {\mathcal{V}}^{p/q}_{r,N}$, the billiard map $f_\gna$ satisfies: 
\begin{enumerate}[label=(\Alph*)]
\item \label{it_a_main} there exists a unique  periodic orbit $\mathcal{P}=\mathcal{P}^{p/q}$ in the Aubry-Mather set of rotation number $\frac p q$, and it is hyperbolic;
\item \label{it_b_main} the set of transverse homoclinic points to $\mathcal{P}$ is not empty;
\item \label{it_c_main} for any branches $W\subset W^\sta(\mathcal{P})$, $W'\subset W^\uns(\mathcal{P})$, any homoclinic point $h\in W_N\cap W_N'$ is  transverse. 
\end{enumerate}

Then, the set $\mathcal{V}^{p/q}_{r,N}$ is open and dense in the analytic topology induced by $\| \cdot \|_r$. 
\end{theorem}
Note that, as said before, Theorem~\ref{main_thm bis_AM} follows directly from Theorem~\ref{main_thm ter_AM}, by considering the $G_\delta$-set 
\[
\widetilde{\mathcal{B}}_r:=\bigcap_{\frac{p}{q}\in\mathbb{Q}\cap(0,1)}\bigcap_{N \in \mathbb{N}} {\mathcal{V}}^{p/q}_{r,N}.
\]

\begin{remark}
Actually, our main result holds for any topology that makes the space of analytic billiards a Baire space.
\end{remark}

\begin{corollary}\label{coro entropy}
		Let $r>0$. Then, the set of billiards $\gna\in  \mathcal{B}_r$ such that the billiard map $f_\gna$ has positive topological entropy is open and dense (in the analytic topology). In fact, the same result also holds for billiards with boundaries defined by trigonometric polynomials. 
\end{corollary}

The proof of Theorem~\ref{main_thm ter_AM} has two different parts. We first prove that the set $\mathcal{V}^{p/q}_{r,N}$ is open for any $\frac{p}{q} \in \mathbb{Q}\backslash \mathbb{Z}$ and $N \in \mathbb{N}$. The proof of its density will be the conclusion of Theorem~\ref{main thm 2 bis}, whose proof will be given along Section \ref{sec:proof of main thm 2 bis}.

In fact, the proof that $\mathcal{V}_{r,N}^{p/q} \subset \mathcal{B}_r$ is an open set is immediate. Indeed, take $\gna \in \mathcal{B}_r$ such that its associated billiard map has a unique  hyperbolic periodic orbit $\mathcal{P}$ in the Aubry-Mather set of rotation number $\frac{p}{q}$. For $\diamond= \uns,\sta$, let us label by $\big \{ W^{j,\diamond}\}_{j=0, \cdots, 2q-1}$ the branches of $W^{\diamond}(\mathcal{P})$ and consider the set  
\begin{equation}\label{homoclinic set}
\mathcal{H}_N:=\bigcup_{i,j=0}^{2q-1} W_{N}^{i,\sta} \cap W^{j,\uns}_N,
\end{equation}
with $W_N^{i,\sta}, W_N^{j,\uns}$ defined as in~\eqref{def:branch}. We observe that $\mathcal{H}_N$ is finite. Indeed, we only need to check that, for all $i,j =\{0, \cdots , 2q-1\}$, the set $W_{N}^{i,\sta} \cap W^{j,\uns}_N$ is finite. To check this, we consider $h \in W^{i,\sta}_N \cap W^{j,\uns}_N$ a homoclinic point, which, by hypothesis is transverse. Then, denoting by $W^{i,\sta}_\rho(h) \subset W^{i,\sta}_N$ and $W^{j,\uns}_\rho(h) \subset W^{j,\uns}_N$ the $\rho$-neighborhoods of $h$, for the distance induced by the Riemannian metric on these leaves, we deduce by analiticty that, as the intersection is transverse, 
for $\rho>0$ small enough $W_\rho^{i,\sta}(h)\cap W_\rho^{j,\uns}(h)=\{h\}$. Therefore, since by definition~\eqref{def:branch}, $W_N^{i,\sta} \cap W_N^{j,\uns}$ is compact, we conclude that it is finite. 

As a result, the $\mathcal{C}^2$-openness of $\mathcal{V}^{p/q}_{r,N}$ is clear as a $\mathcal{C}^2$-small perturbation of the function $\gna$ defining the boundary of the billiard induces a $\mathcal{C}^1$-small perturbation of the associated billiard map (see e.g.~\cite[Section 2.11]{CheMar_book} for a reference). 
As a consequence, if the perturbation is small enough, the perturbed billiard map will have a unique periodic orbit $\tilde{\mathcal{P}}$  in the Aubry-Mather set of rotation number $\frac{p}{q}$ with the same (finite) number of transversal intersections, which are continuation of the ones for the unperturbed billiard map.

The density of the set $\mathcal{V}^{p/q}_{r,N}$ is given in next Theorem.

\begin{theorem}\label{main thm 2 bis}
Let $r>0$, $\frac p q\in \mathbb{Q}/\mathbb{Z}$, $N\in \mathbb{N}$ and take $\gna\in\mathcal{B}_{r}$. 
For any $\epsilon>0$, there exists a trigonometric polynomial $\lambda$ with $\|\lambda\|_r := \sup_{\ss\in \mathbb{T}_r} |\lambda(\ss)|<\epsilon$ such that the following holds. 

Let $\dgnalna\colon \ss \mapsto \gna(\ss)+\lambda(\ss) \nna(\ss)$, where $\nna(\ss)\in\mathbb{R}^2$ is the unitary outward normal vector at $\gna(\ss)$. The billiard map $f_{\dgnalna}$ associated to $\Omega(\dgnalna)$ satisfies:
\begin{enumerate}[label=(\Alph*)]
\item\label{it_a} there exists a unique periodic orbit $\mathcal{P}=\mathcal{P}^{p/q}$ in the Aubry-Mather set of rotation number $\frac p q$, and it is hyperbolic;
\item\label{it_b}  the set of transverse homoclinic points to $\mathcal{P}$ is not empty;
\item\label{it_c}  for any branches $W\subset W^\sta(\mathcal{P})$, $W'\subset W^\uns(\mathcal{P})$, any homoclinic point $h\in W_N\cap W_N'$ is  transverse.  
\end{enumerate}
\end{theorem}

The proof of Theorem~\ref{main thm 2 bis} is given in Section~\ref{sec:proof of main thm 2 bis} where we provide the main ingredients together with some technical results whose proofs are postponed to the subsequent sections.

\section{Proof of Theorem~\ref{main thm 2 bis}}\label{sec:proof of main thm 2 bis}
To build the trigonometric polynomial $\lambda$ in Theorem~\ref{main thm 2 bis} satisfying the required properties we design a strategy for the billiard setting combining the ideas of Zehnder~\cite{Zehnder} and Genecand~\cite{Genecand}. 
The main difficulty in adapting Zehnder's methodology in the framework of billiards comes from the fact that any perturbation of the billiard table reflects into a fibered perturbation of the corresponding twist map of the annulus. We use Aubry-Mather theory and we manage to overcome such a difficulty, obtaining periodic orbits whose projection on the first coordinate is injective; moreover, as in~\cite{Genecand}, we obtain the existence of homoclinic points. For that reason in Section~\ref{Aubry Mather} we give an overview about Aubry-Mather theory, stating the results we will use along the proof of Theorem~\ref{main thm 2 bis}. 

To begin the proof, take an analytic billiard table $\Omega = \Omega(\gna)\in\mathcal{B}_r$, fix the rational rotation number $\frac{p}{q}$, and  consider the Aubry-Mather set associated to $\frac{p}{q}$, denoted by $\mathcal{M}_{\frac p q}(\Omega)$.

While the initial billiard table $\gna$ and the perturbation we want to produce are analytic, our strategy is partially based on $\mathcal{C}^k$ perturbations. For that reason,  we begin by working with $\mathcal{C}^k$ billiards. Some of the results along the proof require  $k\geq k_0$ for $k_0$ big enough. We emphasize that it is not needed to pay special attention to the minimum differentiability $k_0$ needed because, ultimately, we will work with analytic billiards.

Therefore, we think of the initial billiard  $\gna\colon \T \hookrightarrow \mathbb{R}^2$ as a $\mathcal{C}^k$, $k \in \mathbb{N}_{k\geq 2}\cup \{\infty,\omega\}$, embedding and we assume that 
$\Omega =\Omega(\gna)\subset \R^2$ is a strongly convex billiard table. Following the convention in \cite{CheMar_book}, we assume that its curvature is strictly negative.

We consider the usual $\mathcal{C}^k-$norm:  
\[
\|\gna\|_{\mathcal{C}^k} := \max_{j=0,\cdots , k} \left \{ \max_{\ss\in \mathbb{T}} \|\gna^{(j)}(\ss)\|\right \} < \infty\, ,
\]
denoting by $\gna^{(j)}$ the $j$-th derivative of $\gna$, and by $\| \cdot \|$ the Euclidean norm of $\mathbb{R}^2$. We will also use the notation $\dergna(\ss) = \gna^{(1)}(\ss)$ and $\dertwogna(\ss) = \gna^{(2)}(\ss)$, and we denote by $\nna(\ss)$ the unit outward normal vector at the point $\gna(\ss)\in \partial \Omega$. 
 
We let $\mathbb{A}:=\T \times [-\frac{\pi}{2},\frac{\pi}{2}]$, 
and denote by 
\begin{equation}\label{eq:unperturbedmap}
\begin{split}
f_\gna \colon\mathbb{A} & \to\mathbb{A}    \\
(\ss,\vp) & \mapsto (\ss',\vp')
\end{split}
\end{equation}
the billiard map (or collision map) within $\Omega$, where $\ss$ corresponds to the position on the boundary, and $\vp$ is the oriented angle from the velocity vector to the inward normal vector $-\nna(\ss)$ at $\gna(\ss)$. The billiard map $f_\gna$ belongs to $\mathcal{C}^{k-1} (\mathbb{A}, \mathbb{A})$ with $\| f_\gna\|_{\mathcal{C}^{k-1}} < \infty$ (see e.g.~\cite{LeCalvez1990}) and, for $\ell \leq k-1$,
\[
\|f_\gna\|_{\mathcal{C}^\ell} := \max_{j=0,\cdots , \ell} \left \{ \max_{(\ss,\vp)\in \mathbb{A}} \|D^j f_\gna(\ss,\vp)\|\right \}
\]
with $D^jf_\gna (\ss,\vp)$ the differential of order $j$ of $f_\gna$, and where we denote by $\|\cdot\|$ the norm induced by the Euclidean norm of $\mathbb{R}^2$ on (multi)-linear maps. 

The billiard perturbations we are going to consider are of the form 
\begin{equation}\label{def gamma lambda}
\dgnalna(\ss) = \gna(\ss) + \lna(\ss) \nna(\ss)
\end{equation}
with $\lna\in \mathcal{C}^k(\mathbb{T},\mathbb{R})$. We emphasize that $\dgnalna\in \mathcal{C}^{k-1}(\mathbb{T},\mathbb{R}^2)$. We also denote by $f_{\dgnalna}$ the billiard map associated to the billiard table $\Omega(\dgnalna)$ and we note that $f_\gna=f_{\gna_0}$. 

Our strategy is to write the perturbed billiard table $\dgnalna$ in the arc-length coordinates associated to the unperturbed billiard $\gna$ and to obtain a new parameterization of the perturbed table $\dgnalna$, that we denote by $\dgalac$. 
We then will study  the billiard map $f_{\dgalac}$. The first step of the proof is to study the regularity of the functional $ \lna\mapsto f_{\dgalac}$. We will see that such correspondence is $\mathcal{C}^1$ Fr\'echet differentiable in a suitable $\mathcal{C}^k$ topologies and we will compute its derivative at $ \lna=0$. 

\subsection{Preliminaries}\label{sec: preliminaries billiard perturbations}
Consider  
\begin{equation}\label{arclength parameterization:section3}
\sigma(\ss) = \frac{1}{|\partial \Omega(\gna)|} \int_{0}^{\ss} \|\dergna(u)\| du, \qquad 
\ga(s)  := \frac{\gna \circ \sigma^{-1}(s)}{|\partial \Omega(\gna)| }\, ,
\end{equation}
such that $\ga $ is the anticlockwise arc length parameterization of $\gna(\mathbb{T})=\partial \Omega(\gna)$ satisfying $|\partial \Omega (\ga)| =1$. We observe that, letting  
\begin{equation}\label{change arc length}
(s,\varphi)=S(\ss,\vp) := \left (  \sigma(\ss), \vp \right ), \qquad DS (\ss,\vp) = \begin{pmatrix}
    \frac{\| \dergna(\ss)\|}{|\partial \Omega (\gna)|} & 0 \\ 0 & 1 
\end{pmatrix},
\end{equation}
it is clear that 
\begin{equation}\label{conjugation with arclength}
f_{\gna} = S^{-1} \circ f_{\ga} \circ S
\end{equation}
with
\begin{equation}\label{eq:unperturbedmap arclength}
\begin{split}
f_{\ga} \colon\mathbb{A} & \to\mathbb{A}    \\
(s,\varphi) & \mapsto (s',\varphi')
\end{split}
\end{equation}
the billiard map associated to the billiard table $\Omega(\ga)$. Observe that, while $f_{\ga}$ preserves the $2$-form $\omega=\cos(\varphi)d\varphi\wedge ds$, the map $f_\gna$ preserves the $2$-form $S^*\omega= \cos(\varphi)\sigma'(\ss)d\varphi\wedge d\ss$.

Let $\pi\colon \mathbb{R}\to \mathbb{T}$ be a universal covering map of the 1-torus. With an abuse of notation, we still refer to $\ga$ when considering $\ga\circ\pi\colon \mathbb{R}\to \mathbb{R}^2$. We denote by
\begin{equation}\label{generating function arclength}
\begin{split}
\tau \colon \R\times \R &\to \R_+ \\
(s,s')& \mapsto \tau(s,s'):=\|\ga(s)-\ga(s')\|     
\end{split}
\end{equation}
the Euclidean distance between the points of $\partial \Omega(\ga)$ with respective parameters $s\, (\mathrm{mod}\, 1)$ and $s'\, (\mathrm{mod}\, 1)$. We recall that $\tau$ is a generating function for the billiard map $f_{\ga}$, i.e., we have $f_{\ga}^*\alpha - \alpha=d\tau$, where $\alpha:=\sin \varphi\, ds$ is the Liouville $1$-form. In other words, on the set $\mathbb{R}\times \mathbb{R}\setminus \Delta$, with 
\begin{equation}\label{def Delta}
\Delta:=\{(s,s')\in \mathbb{R}:s'-s\in \mathbb{Z}\},
\end{equation}
we have that $\tau$ is a $\mathcal{C}^{k}$ function and
\begin{equation}\label{first gen fct:section3}
\partial_1 \tau(s,s')=-\sin \varphi,\qquad \partial_2 \tau(s,s')=\sin \varphi'.
\end{equation}
Note that the above formulas can be continuously extended to $\mathbb{R}\times \mathbb{R}$ by letting 
\begin{equation}\label{extension generating function}
\partial_1 \tau(s,s+k)=(-1)^k,\quad \partial_2 \tau(s,s+k)=(-1)^{k+1},\quad \forall\, k \in \mathbb{Z}. 
\end{equation}

As a result, coming back to the original variables $(\ss,\vp)$, taking 
\begin{equation}\label{generating function original:section3}
\begin{split}
\gentau \colon \R\times \R &\to \R_+ \\
(\ss,\ss')& \mapsto \gentau(\ss,\ss'):=  \|\gna(\ss)-\gna(\ss')\| =|\partial \Omega(\gna)|\tau(\sigma(\ss), \sigma(\ss')),  
\end{split}
\end{equation}
from~\eqref{first gen fct:section3}, we have that 
\begin{equation}\label{first gen fct no arclength}
\sin \vp = - \partial_1 \gentau (\ss,\ss') \frac{1 }{\|\dergna(\ss)\|}, 
\qquad \sin \vp' = \partial_2 \gentau (\ss,\ss') \frac{1  }{\|\dergna(\ss')\|}\, .
\end{equation}
Observe that the generating function of $f_\gna$ is then $\frac{\gentau}{\vert \partial\Omega(\gna)\vert}$. 
Indeed, one can directly check that 
\[
f_\gna^* S^*\alpha-S^*\alpha=d\left(\frac{\gentau}{\vert \partial\Omega(\gna)\vert}\right).
\]
Recall that $\gna\colon \T \hookrightarrow \mathbb{R}^2$ is a $\mathcal{C}^k$ embedding so that $\|\dergna(\ss)\|\geq C>0$. 

From now on we make the convention that whenever we write $(s,\varphi)$ we deal with a billiard map such that $s$ is the counterclockwise arclength parameterization of the corresponding boundary $\ga(\mathbb{T})$ which has unit perimeter. In contrast, the variables $(\ss,\vp)$ refer to the original variables of $\Omega(\gna)$. 

 \subsection{Remaining away from the boundaries}\label{billiard perturbationBoundary}
Fix $\gna \in \mathcal{C}^k(\mathbb{T};\mathbb{R}^2)$ and consider its arclength parameterization $\ga\in \mathcal{C}^k(\mathbb{T};\mathbb{R}^2)$ defined in~\eqref{arclength parameterization:section3}. 
Let $\gentau,\tau$ be the functions defined in~\eqref{generating function original:section3} and~\eqref{generating function arclength} respectively. As in~\eqref{eq:unperturbedmap} and~\eqref{eq:unperturbedmap arclength}, we use the notations
\[
(\ss',\vp')=f_\gna (\ss,\varphi), \qquad (s',\varphi')=f_{\ga}(s,\varphi).
\]

By the twist condition, one knows that there exists a diffeomorphism $\chi$, that changes coordinates, from $\R\times(-\frac{\pi}{2},\frac{\pi}{2})$ to $\{ (s,s')\in \mathbb{R}^2 :\ s<s'<s+1 \}$. Observe that $\chi (s+1,\varphi)=\chi(s,\varphi)+(1,1)$. 

The following lemma is then an immediate outcome of the existence and the property of $\chi$ and of some compactness argument along with the continuous extension (see~\eqref{extension generating function}) of $\partial_2 \tau(s,s')$ to the set $\Delta$ in~\eqref{def Delta}. We omit its proof.

\begin{lemma}\label{away from boundary 1}
Let $0<\nu\ll 1$. Denote by $\mathbb{A}_\nu$ the annulus
\begin{equation}\label{def:Anu}
\mathbb{A}_\nu =\mathbb{T}\times \left [-\frac{\pi}{2}+\nu,\frac{\pi}{2}-\nu \right ]
\end{equation}
and by $\tilde{\mathbb{A}}_\nu$ a lift of it in $\mathbb{R}\times(-\frac{\pi}{2},\frac{\pi}{2})$. Then, there exist constants $0<\mu <\frac 1 2$ and $\nu'>0$ such that
\begin{equation}\label{definition Delta mu}
\chi(\tilde{\mathbb{A}}_\nu)\subset \Delta_\mu:=\{ (s,s')\in \mathbb{R}^2 :\ s+\mu\leq s'\leq s+1-\mu \}\, 
\end{equation}
{and $\varphi' \in\left [-\frac{\pi}{2}+\nu',\frac{\pi}{2}-\nu' \right ]$. }
\end{lemma}

We are then concerned with the fact that, being at a positive distance from the boundaries, implies a uniform lower bound on the generating function $\tau$. Once again, this follows by continuity of $\tau$ away from the diagonal and by compactness and periodicity of some set.
\begin{lemma}\label{away from boundary 2}
Let $0<\mu<\frac 1 2$ and $\Delta_\mu$ defined in~\eqref{definition Delta mu}. 
There exists a constant $c>0$ such that, for every $(s,s')\in \Delta_\mu$, one has $\tau(s,s')\geq c$.
    
Conversely, for any $c'>0$, there exists a constant $0<\mu'<\frac{1}{2}$ such that if $\tau(s,s')\geq c'$, then $(s,s') \in \Delta_{\mu'}$.
\end{lemma}
\begin{proof}
First observe that, by definition of the generating function, $\tau(s,s')$ is zero if and only if $s\equiv s'\mod 1$. Since the function $\tau$ is smooth except at points of the set $\Delta$ defined in~\eqref{def Delta}, and since $\tau(s+1,s'+1)=\tau(s,s')$, by compactness of $\Delta_\mu$ intersected with a fundamental domain, the first statement follows. The second one is proved by observing that $\tau(s,'s) \geq c>0$ implies that $s-s'\notin \mathbb{Z}$ and the conclusion is clear by compactness.
\end{proof}

The following corollary is a straightforward consequence of the previous results and the fact that $\gna$ is an embedding. We omit its proof.

\begin{corollary}\label{corollary away boundary}
Recall that $\Delta_\mu = \{ (\ss,\ss')\in \mathbb{R}^2 :\ \ss+\mu\leq \ss'\leq \ss+1-\mu \}$. We have that:
\begin{itemize}
\item For any $0<\nu \ll 1$ small enough, there exist $0<\mu <\frac{1}{2}$ and $\nu'>0$ such that if $(\ss,\vp) \in \tilde{\mathbb{A}}_\nu$, then $(\ss,\ss') \in \Delta_\mu$ and 
$\vp' \in \left [-\frac{\pi}{2} + \nu', \frac{\pi}{2} - \nu'\right ]$.
\item For any $0<\mu <\frac{1}{2}$, there exists $c>0$ such that $\gentau(\ss, \ss') \geq c$ for all $(\ss,\ss') \in \Delta_\mu$.
\item For any $c'>0$, there exists $0<\mu' <\frac{1}{2}$ such that if $\gentau(\ss,\ss') \geq c'$, then $(\ss,\ss') \in \Delta_\mu$.
\end{itemize}
\end{corollary}

\subsection{Perturbations of the billiard table} \label{billiard perturbationTable:sec3} 
Consider a billiard map $f_{\ga}$, with $\ga$ the arc length anticlockwise parametrization of $\ga(\mathbb{T})$.  
For a given $\lac\in \mathcal{C}^k(\mathbb{T},\mathbb{R})$, we consider the one-parameter family $\{\Omega_\varepsilon\}_{\varepsilon\in [0,1]}$ of $\mathcal{C}^k$ deformations of $\Omega (\ga)$, with 
\begin{equation}\label{deformations}
\partial\Omega_\varepsilon:=\{\defga (s):=\ga(s)+\varepsilon \lac(s)\narc(s)\,;\, s\in \mathbb{T}\}\, ,
\end{equation}
where $\narc(s)$ is the outward normal vector to $\ga(\mathbb{T})$ at $\ga(s)$. Let us note that by keeping the same parameter $s$ to describe the perturbed boundary $\partial\Omega_\varepsilon$, the latter is \textit{a priori} not parametrized in arc length anymore. 

We introduce the $\mathcal{C}^k$ norm of $\lac\in \mathcal{C}^k(\T,\R)$ as
\begin{equation}\label{Ck norm of lambda}
\|\lac\|_{\mathcal{C}^k} = \max_{j=0\cdots, k} \max_{s\in \T} |\lac^{(j)}(s)|
\end{equation}
where, again, we have denoted by $\lac^{(j)}$ the derivative of order $j$ and, in general for a given function $g\in \mathcal{C}^\ell (\mathbb{T} \times [-\varepsilon_0,\varepsilon_0],\mathbb{R})$, $\ell \leq k$, $\varepsilon_0\geq 0$, we define analogously 
\[
\|g \|_{\mathcal{C}^{\ell}} = \max_{j=0\cdots, \ell} \max_{(s,\varepsilon)\in \T \times [-\varepsilon_0,\varepsilon_0]} |D^j g(s;\varepsilon)|
\]
where $D$ is the differential with respect to $(s,\varepsilon)$. 
We also introduce, for 
$0<\nu \ll 1$ and $\ell \leq k-1$ 
\begin{align*}
\|R\|_{\norm{\ell}{\mathbb{A}_\nu}}  & := \max_{j=0,\cdots , \ell} \left \{ \max_{(s,\varphi,\varepsilon)\in \mathbb{A}_\nu \times [0,\varepsilon_0]} \|D^j R(s,\varphi;\varepsilon)\|\right \}
\end{align*}
where $R$ is $\mathcal{C}^{\ell}$ in its arguments and $D$ here means the differential with respect to $(s,\varphi,\varepsilon)$. 
Recall that the set $\mathbb{A}_\nu$ was introduced in Lemma~\ref{away from boundary 1}. 

Consider $f_{\ga}\colon (s,\varphi)\mapsto (s',\varphi')$ the initial billiard map parameterized by arc-length, and   
\begin{equation}\label{eq:perturbedmap:section3}
f_{\defga}\colon (s,\varphi)\mapsto (s_\varepsilon',\varphi_\varepsilon')
\end{equation}
the billiard map for the perturbed domain $\Omega_\varepsilon$, defined in~\eqref{deformations}, with $s_\varepsilon'=s_\varepsilon'(s,\varphi)$ and $\varphi_\varepsilon'=\varphi_\varepsilon'(s,\varphi)$. 

Next proposition, whose proof is given in Section~\ref{billiard perturbationTable}, gives  the first order in $\varepsilon$ of the perturbed billiard map $f_{\defga}$. 

\begin{prop}\label{cor:expr:fepsilon} 
Fix $0<\nu \ll 1$. 
Then, there exist $\varepsilon_\nu>0$ and a constant $M_\nu$ (depending also on $\|\ga\|_{\mathcal{C}^k}$) such that, for $\varepsilon \in [-\varepsilon_\nu, \varepsilon_\nu]$ and $(s,\varphi)\in \mathbb{A}_\nu $, with $\mathbb{A}_\nu$ given in~\eqref{def:Anu}, the perturbed billiard map $f_{\defga}$ in~\eqref{eq:perturbedmap:section3} can be written as:
\begin{align*}
s_\varepsilon'&=s'-\varepsilon\frac{1}{\cos\varphi'}(\derlac(s)\hat{\tau}+\lac(s)\sin \varphi-\lac(s')\sin\varphi')+\varepsilon^2 \mathcal{R}_1(s,\varphi;\varepsilon)\\
\varphi_\varepsilon'&=\varphi' - \varepsilon \frac{\mathcal{K}(s')}{\cos \varphi'}(\derlac(s)\hat \tau+\lac(s)\sin\varphi- \lac(s')\sin \varphi') + \varepsilon (\derlac(s') - \derlac(s)) + \varepsilon^2 \mathcal{R}_2(s,\varphi;\varepsilon).
\end{align*}
where $(s',\varphi')=f(s,\varphi)$, $\hat \tau=\hat \tau(s,\varphi) = \tau(s,s'(s,\varphi))$, $\mathcal{K}(s')$ is the curvature of $\ga$ at the point $\ga(s')$ and the functions $\mathcal{R}_1,\mathcal{R}_2$ satisfy:
\[
\|\mathcal{R}_1\|_{\norm{k-3}{\mathbb{A}_\nu}}\le M_\nu \| \lac \|_{\mathcal{C}^{k-1}}^2, \qquad \|\mathcal{R}_2\|_{\norm{k-3}{\mathbb{A}_\nu}}\le M_\nu \| \lac \|_{\mathcal{C}^{k-1}}^2\, .
\]
Moreover, $\|f_{\defga}\|_{\norm{k-3}{\mathbb{A}_\nu}} \leq 2 \|f_{\ga}\|_{\norm{k-1}{\mathbb{A}_\nu}}$. 
\end{prop}

\subsection{From normal perturbations to billiard maps}\label{sec:Gammanu:sec3}
Recall that, given $\gna \in \mathcal{C}^k (\mathbb{T},\mathbb{R}^2)$  and $\lambda\in \mathcal{C}^k(\mathbb{T},\mathbb{R})$, in~\eqref{def gamma lambda} we have defined $\dgnalna$ as
\begin{equation}
\label{def gamma lambda bis}
\dgnalna(\ss) = \gna(\ss) + \lambda(\ss) \nna(\ss), \qquad \dgnalna \in \mathcal{C}^{k-1} (\mathbb{T};\mathbb{R}^2). 
\end{equation}

We introduce, for $\varrho>0$, the $\mathcal{C}^k$ open ball of radius $\varrho$, 
\begin{equation}\label{def:Brho}
\Bl{\varrho} = \{ \lna: \mathbb{T} \to \mathbb{R}, \, \lambda\in \mathcal{C}^k : \,  \|\lambda\|_{\mathcal{C}^k}< \varrho\},
\end{equation}
where $\|\lambda\|_{\mathcal{C}^k}$ denotes the usual $\mathcal{C}^k$-norm (see~\eqref{Ck norm of lambda}). 

We observe that, considering $\ss = \sigma^{-1}(s)$ and $\ga$ the arc length parameterization of $\gna$ as in~\eqref{arclength parameterization:section3}, it is clear that, letting $ \lac = |\partial \Omega(\gna)|^{-1} \lna\circ \sigma^{-1}$, 

\begin{equation}\label{def gamma lambda arc length}
\dgalac(s) := \frac{1}{|\partial \Omega (\gna)|} \dgnalna (\sigma^{-1}(s))=\ga(s) +  \lac(s) \narc(s),  
\end{equation}
with $\narc(s) = \nna(\sigma^{-1}(s))$. 
Using the Fa\`a di Bruno formula one can easily check that there exists a constant $\co>0$ that only depends on $\|\gna \|_{\mathcal{C}^k}$, such that for all $\lna\in \Bl{\varrho}$
\begin{equation}\label{lambdaast lambda}
\| \lac\|_{\mathcal{C}^k} = \frac{1}{|\partial \Omega (\gna)|} \|\lna\circ \sigma^{-1} \|_{\mathcal{C}^k} \leq\co \|\lna\|_{\mathcal{C}^k}  \leq \co   \varrho.
\end{equation}
We notice that $\dgalac$ can be written as
\begin{equation}\label{deformations gamma ast}
\dgalac(s)= \ga(s) + \varepsilon \lac_0(s) \narc(s):= \ga(s) + \varepsilon \frac{ \lac(s)}{\|\lac\|_{\mathcal{C}^k} }\narc(s), \qquad \varepsilon = \| \lac\|_{\mathcal{C}^k} \leq \co   \varrho
\end{equation}
in such a way that the billiard table $\Omega(\dgalac)$ belongs to the family of deformations of $\Omega(\ga)$ given by
\[
\partial \Omega_\varepsilon=\{ \defga(s)= \ga(s) + \varepsilon \lac_0(s) \narc(s) \}, \qquad \varepsilon \in [-\co\varrho,\co\varrho],
\]
as defined in~\eqref{deformations}.  

For a billiard map $f_{\ga}$, $\nu>0$ and $\varrho>0$, we denote
\begin{equation}\label{ball fgamma}
\Bf{\varrho}{f_{\ga}}:=\{\tilde f\colon \mathbb{A}_\nu\to\mathbb{A}, \text{ symplectic }\mathcal{C}^{k-2}\text{ diffeomorphism}: \Vert \tilde f-f_{\ga}\Vert_{\difknu{k-3}{\nu}}<\varrho\} .
\end{equation}
Notice that when $k\geq 5$, if $\tilde f \in \Bf{\varrho}{f_{\ga}}$, then $\|\tilde f -f_{\ga}\|_{\difknu{2}{\nu}} <\varrho$.

Next Proposition, whose proof is given in Section \ref{sec:Gammanu}, studies the regularity of the map $f_{\dgalac}$ respect to the original perturbation $\lambda$ where we recall that $\lac = |\partial \Omega(\gamma)| \lambda \circ \sigma^{-1}$. In particular proves that for $\lambda \in \Bl{\varrho}$ the billiard tables $f_{\dgalac}$ are strictly convex when $\varrho$ is small enough.
 
\begin{proposition}\label{prop Frechet Gamma nu} 
Fix $\nu>0$. There exist $\varrho_\nu >0$ and $b_\nu>0$,  such that for all $0<\varrho\leq \varrho_\nu $, the billiard table $\Omega (\dgalac)$ is strictly convex and the  map 
\begin{equation}\label{eq:gammatilde}
\begin{array}{rcl}
\Gamma_\nu \colon \Bl{\varrho} & \to & \Bf{b_\nu\varrho}{f_{\ga}}  \\
\lna& \mapsto &  f_{\dgalac} 
\end{array}      
\end{equation}
 is $\mathcal{C}^1$-Fr\'echet differentiable.     
\end{proposition}
We stress that, as $f_{\dgnalna} = S^{-1} \circ f_{\dgalac} \circ S $, (see \eqref{conjugation with arclength} and \eqref{change arc length}), the same is true for the map $f_{\dgnalna}$. This gives the following corollary.

\begin{corollary}\label{prop Frechet Gamma nu noal} 
Fix $\nu>0$. There exists $\varrho_\nu$  and $b_\nu>0$,  such that for all $0<\varrho\leq \varrho_\nu$, the billiard table $\Omega (\gna_\lambda)$ is strictly convex and the  map 
\begin{equation}\label{eq:gammatilde noarclenth}
\begin{array}{rcl}
\widetilde{\Gamma}_\nu \colon \Bl{\varrho} & \to & \Bf{b_\nu\varrho}{f_{\gna}} \\
\lna& \mapsto &  f_{\gna_\lambda} 
\end{array}      
\end{equation}
is $\mathcal{C}^1$-Fr\'echet differentiable. 
\end{corollary}

\subsection{First perturbation: keeping a single hyperbolic Aubry-Mather periodic orbit of type $(p,q)$}\label{sec:hyp:section3}

After the perturbation theory for billiard maps is done, we start constructing the suitable perturbation $\lambda$ which provide the results of Theorem~\ref{main thm 2 bis}. To do so, we strongly rely on the results given in Section~\ref{billiard perturbationTable:sec3} and the properties of the Aubry-Mather set $\mathcal{M}_{\frac{p}{q}} (\Omega(\dgnalna))$ of rotation number $\frac{p}{q}$, see Section~\ref{Aubry Mather}. Therefore, first, we need to ensure that, given an Aubry-Mather set at a positive distance from the boundaries, its perturbation remains at a positive distance from the boundaries.

\begin{lemma}\label{lazutkin}
Fix $\frac p q\in (0,1)$ and consider the Aubry-Mather set $\mathcal{M}_{\frac p q}(\Omega)$ with that rotation number for the map $f_\gna$. 
There exist $\nu>0$ and $\varepsilon_0>0$, such that the set $\mathcal{M}_{\frac p q}(\Omega)$ is contained in $\mathbb{A}_\nu$ and, for any $\varepsilon\in [0,\varepsilon_0]$, and  for every billiard table $\widetilde \Omega$ which is $\varepsilon$-close to the initial table $\Omega$ in the $\mathcal{C}^k$ distance (for $k$ large enough), one also has $\mathcal{M}_{\frac p q}(\widetilde \Omega)\subset \mathbb{A}_\nu$.
\end{lemma}
\begin{proof}
In \cite{Lazutkin}, it is proved that, for a billiard table $\Omega$ with sufficiently smooth boundary, there are a neighborhood $\mathcal{U}\subset \mathbb{A}$ of the lower boundary $\mathbb{T}\times \{-\frac{\pi}{2}\}$, a one-parameter family $(\mathfrak{C}(a))_{0\leq a\leq \alpha}$ of curves, $\alpha \in (0,1)$,  and a closed set $E\subset [0,\alpha]$ such that 
\begin{itemize}
\item $\mathfrak{C}(a)$ is an invariant curve if and only if $a\in E$;
\item for $a\in E$, $\mathfrak{C}(a)=\mathcal{M}_{a}(\Omega)$ is an invariant graph with rotation number $a \mod 1$;
\item $\cup_{a\in E} \mathfrak{C}(a)\subset \mathcal{U}$ has positive area;
\item $\exists\, (a_n)_n\in E^\mathbb{N}$ such that $\lim_{n \to +\infty} a_n=0$ and $\lim_{n \to +\infty} \mathfrak{C}(a_n)= \mathbb{T}\times \{-\frac{\pi}{2}\}$.   
\end{itemize}

Recall that we have fixed a rational rotation number $\frac p q$ and its corresponding Aubry-Mather set. Fix a Diophantine number $a_0 \in E\cap (0,\frac{p}{q})$. By the twist condition, the Aubry-Mather set $\mathcal{M}_{\frac p q}(\Omega)$ is separated from the lower boundary of $\mathbb{A}$ by the curve $\mathfrak{C}(a_0)$. Then, there exists $\nu>0$ (we can e.g. choose any $\nu>0$ smaller than a quarter of the minimum of the height of $\mathfrak{C}(a_0)$) such that $\mathcal{M}_{\frac p q}(\Omega)\subset \mathbb{A}_\nu$. 

By the Diophantine condition on $a_0$, for $k$ large enough,  the KAM theorem ensures that any sufficiently $\mathcal{C}^k$-small perturbation $\widetilde \Omega$ of the initial billiard $\Omega$ still admits an invariant graph $\Gamma$ of rotation number $a_0 \mod 1$. Since $a_0<\frac{p}{q}$, by the twist condition, $\Gamma$ lies below the Aubry-Mather set $\mathcal{M}_{\frac p q}(\widetilde \Omega)$ of the perturbed billiard. Moreover, by further restricting the size of the perturbation, we can ensure that $\nu$ is smaller than half the minimum height of $\Gamma$. Consequently, $\mathcal{M}_{\frac p q}(\widetilde \Omega)\subset \mathbb{A}_\nu$.
\end{proof}

The first perturbation is provided by the following result whose proof is delayed to Section~\ref{sec:hyp}.
\begin{proposition}\label{prop:hyperbolicfixedpoint:section3}
Let $\frac p q \in \mathbb{Q}\cap (0,1)$. For any $\epsilon>0$, there exists a trigonometric polynomial $\lambda$ (in particular, $\lna\in \mathcal{C}_r^{\omega}(\mathbb{T}, \mathbb{R})$) with $\| \lambda\|_{r}< \epsilon$, such that, for 
\[
\dgnalna(\ss)= \gna(\ss) + \lambda(\ss)  \nna(\ss)\, ,
\]
we have $\dgnalna\in \mathcal{B}_r$, and the billiard map $f_{\dgnalna}$ has a unique periodic orbit in the Aubry-Mather set $\mathcal{M}_{\frac{p}{q}}(\Omega(\dgnalna))$ of rotation number $\frac{p}{q}$, which, moreover, is hyperbolic. 
\end{proposition}

From Aubry--Mather theory, discussed in Section~\ref{Aubry Mather}, we also prove the following statement.

\begin{proposition}\label{prop:homoclinicpq}
    Let \(\frac p q\in \mathbb{Q}\cap (0,1)\). Assume that the billiard map \(f_{\gamma_\lambda}\) has a unique periodic orbit \(\mathcal{P}\) in the Aubry-Mather set \(\mathcal{M}_{\frac p q}(\Omega(\gamma_\lambda))\) of rotation number \(\frac p q\), which is hyperbolic. Then, there exists a homoclinic orbit \(\mathcal{Q}\) to \(\mathcal{P}\), belonging to the Aubry-Mather set \(\mathcal{M}_{\frac p q}(\Omega(\gamma_\lambda))\). In particular, the projection over the first coordinate of the orbit \(\mathcal{Q}\) is injective.
\end{proposition}

We refer to Proposition~\ref{prop:injectivity} and the discussion over there for the proof. 
We stress that, as a consequence of Proposition~\ref{prop:homoclinicpq}, the existence of homoclinic points for the perturbed billiard map is guaranteed, but \textit{a priori} not their transversality. 

From now on, we rename the billiard table $\dgnalna$ provided by Propositions~\ref{prop:hyperbolicfixedpoint:section3}  and \ref{prop:homoclinicpq} as $\gna$ and therefore we assume that the map $f_\gna$ associated to our billiard table $\gna$  has a unique periodic orbit in the Aubry-Mather set of rotation number $\frac{p}{q}$, which is moreover hyperbolic, and which has, at least, one homoclinic point. Next step is to study all its homoclinic intersections and their transversality.

\subsection{Existence of one fibered homoclinic points}\label{sec:existence one fibered:sec3}
We will recover analytic deformations of the billiard table from compactly supported ones. 
A key difficulty in constructing differentiable perturbations on billiard tables lies in the fact that any modification in the billiard table $\gamma$ induces a fibered perturbation of the phase space $\mathbb{A}$ of the billiard map $f_\gna$ (see Figure~\ref{fig:billards:anell}). 
In other words, even if the function $ \lambda$ (see \eqref{def gamma lambda bis}) modifies only a small portion of $\partial \Omega$, the resulting effect propagates across a large region of the annulus $\mathbb{A}$. 

\begin{figure}[h]
\subfloat 
{
\begin{overpic}[width=0.25\textwidth]{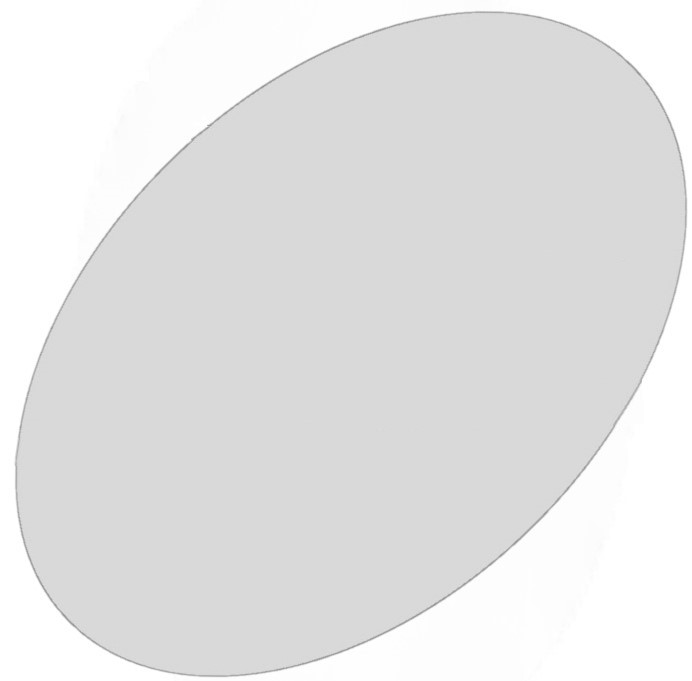}
	\put(60,75){$\Omega(\gna)$ }
    \end{overpic}
}
\hspace{0.3cm}
\subfloat{
\begin{overpic}[width=0.25\textwidth]{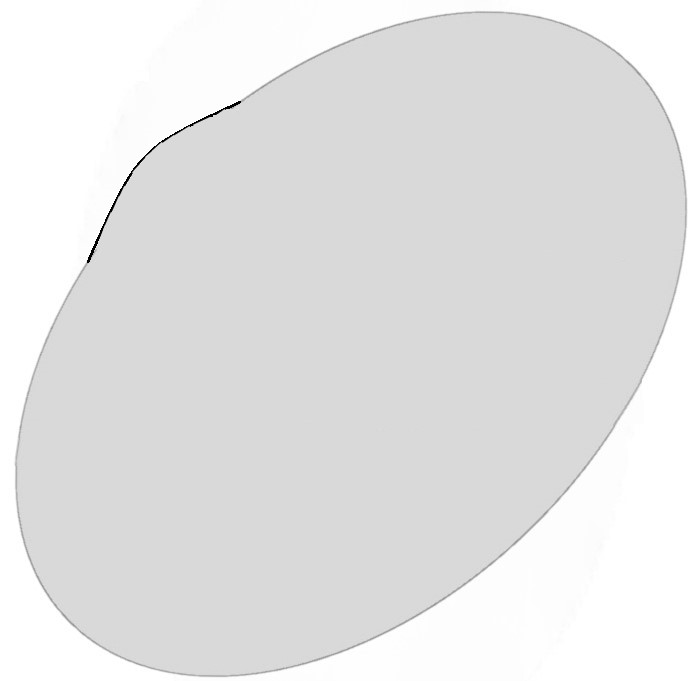}
	\put(60,75){$\Omega (\tilde{\gna})$ }
    \end{overpic}
}
\hspace{0.3cm}
\subfloat
{
\begin{overpic}[width=0.20\textwidth]{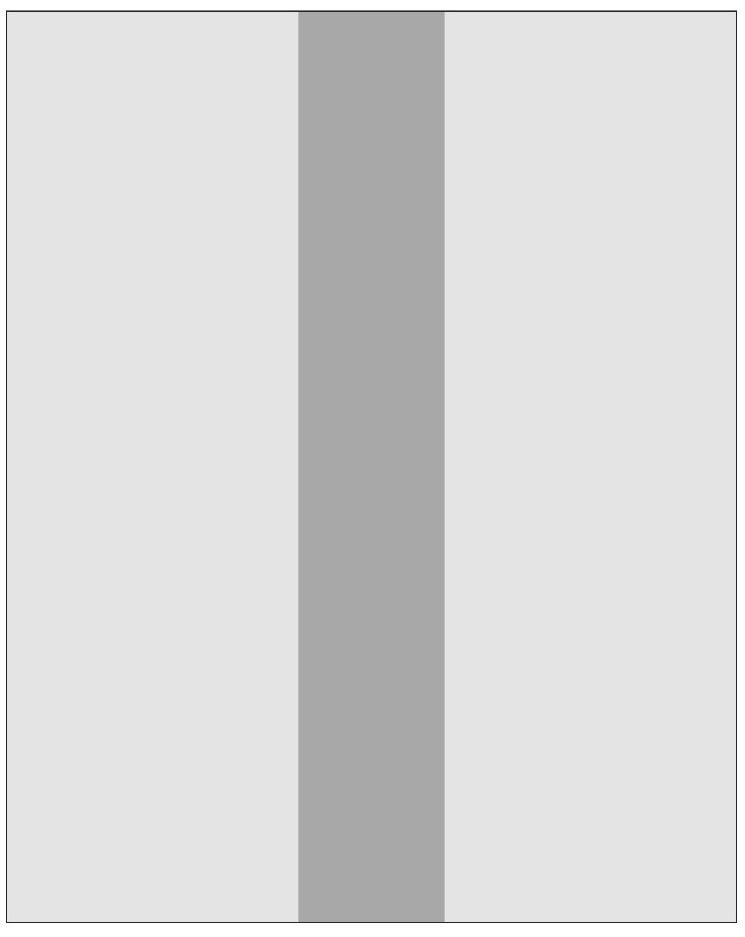}
	\put(60,75){$\mathbb{A}$ }
    \end{overpic}
}
\caption{On the left, the initial billiard table. In the middle, a compactly supported perturbation of the billiard table, this would correspond to $\dgnalna = \gna + \lna\cdot \nna$ with $\lna$ a compact supported function. All the incidence angles are affected by the perturbation. On the right, the region in the phase space $\mathbb{A}$ affected by the perturbation.}  
\label{fig:billards:anell}
\end{figure}

Our strategy to overcome this difficulty relies in the study, for a given homoclinic orbit $\{f_{\gna}^{k}(Q) \}_{k\in \mathbb{Z}}$, of how many iterates of it belong to a given fiber $\{\ss=\ss_*\}$. 

Even if the following discussion can be done for a general twist map, we are going to focus on the billiard map $f_{\gna}$. 

Fix $\frac{p}{q} \in \mathbb{Q} \cap (0,1)$ and $\gna \in \mathcal{C}^k (\mathbb{A},\mathbb{R}^2)$, with $k$ big enough. We recall that the invariant manifolds of a periodic orbit $\mathcal{P}=\{P_j\}_{j=0,\cdots ,q-1}$ are the union of the branches (see~\eqref{defWsWu}):
\[
W^{\diamond}(\mathcal{P})= \bigcup_{j=0}^{q-1} W^{\diamond}(P_j)\, ,\quad \text{for }\diamond=\uns,\sta\,.
\]
Since the initial billiard table $\Omega = \Omega(\gna)$ is such that there is only one periodic orbit $\mathcal{P}$ in $\mathcal{M}_{\frac{p}{q}} (\Omega)$, the Aubry-Mather set of rotation number $\frac p q$, 
by Proposition~\ref{prop:homoclinicpq},
the billiard map has a homoclinic point \(Q\) belonging to $\mathcal{M}_{\frac{p}{q}}(\Omega)$, that is   
$W^{\uns}(\mathcal{P})\cap W^{\sta}(\mathcal{P}) \cap \mathcal{M}_{\frac{p}{q}}(\Omega)\setminus\{\mathcal{P}\}\neq \emptyset$. 
In addition, again by Proposition~\ref{prop:homoclinicpq}, if $Q \in W^{\uns}(\mathcal{P})\cap W^{\sta}(\mathcal{P}) \cap \mathcal{M}_{\frac{p}{q}}(\Omega)\setminus\{\mathcal{P}\}$, the projection of its orbit onto the first component is injective. 
In other words, letting $(\ss_k,\varphi_k)=f_{\gna}^k(Q)$, we have that $\ss_i \neq \ss_j$ if $i\neq j$. 
We call such homoclinic points \emph{one-fibered}. 
More concretely:
\begin{definition}\label{definition one fibered}
A point $Q=(\ss_0,\varphi_0)\in \mathbb{A}$ is~\emph{one-fibered} if  there exists a closed neighborhood $U_Q\subset \mathbb{T}$ of $\ss_0$ such that 
\[
\forall\, j \in \mathbb{Z}\setminus \{0\},\quad \ss_j\notin U_Q,\quad \text{where }(\ss_j,\varphi_j)=f_{\gna}^i(Q)\, .
\]   
The neighborhood $U_Q$ is called the \emph{one-fibered neighborhood} of $Q$.
\end{definition}

As we already claimed, if a homoclinic point belongs to $\mathcal{M}_{\frac{p}{q}}(\Omega)$ all the points in its orbit are one-fibered, meanwhile
if it does not belong to $\mathcal{M}_{\frac{p}{q}}(\Omega)$, 
it might be not one-fibered. However, some points of the homoclinic orbit might be one-fibered. 

While any homoclinic point belonging to $\mathcal{M}_{\frac{p}{q}}(\Omega)$ has an orbit consisting entirely of one-fibered points, this property may not hold for points outside $\mathcal{M}_{\frac{p}{q}}(\Omega)$. In the latter case, the homoclinic orbit is not necessarily one-fibered, although it may still contain some one-fibered points. This leads us to the following definition.

\begin{definition}\label{defi one fibered orbit}
Let $Q$ be a homoclinic point. 
The orbit of $Q$ is \emph{one-fibered} if there exists $\ell\in\mathbb{Z}$ such that the point $f^\ell_{\gna}(Q)$ is one-fibered. 
The orbit of $Q$ is  \emph{not one-fibered} if, for every $n \in\mathbb{Z}$, the point $f^n_{\gna}(Q)$ is not one-fibered.
\end{definition}

Let us introduce also the following definition. 

\begin{definition} \label{definition local invariant manifold:section3} The local stable (resp. unstable) manifold of a point $P$ of a periodic orbit $\mathcal{P}$, denoted by $W^{\sta}_{\mathrm{loc}}(P)$ (resp. $W^{\uns}_{\mathrm{loc}}(P)$), refer to 
\[
W^{\sta}_{\mathrm{loc}}(P):=W^{\sta}_\rho(P)= \{ x \in W^{\sta}(P)\,:\, d_W(P,x)<\rho\},   
\]
(resp. $W^{\uns}_\rho(P)$), where $d_W$ is the distance on the invariant manifold induced by the Riemannian metric, with $\rho$ small enough such that $W^{\sta}_\rho(P)$ (resp. $W^{\uns}_\rho(P)$) is a graph with respect to either $\ss$ or $\varphi$.    
\end{definition}
We stress that when $\mathcal{P}$ belongs to the Aubry-Mather set $\mathcal{M}_{\frac{p}{q}}(\Omega)$, the local manifolds $W^{\diamond}_{\mathrm{loc}}(P)$, $\diamond=\uns,\sta$, are graphs with respect to $\ss$. This comes from the fact that, since points in the Aubry-Mather set do not have conjugate points, the stable and unstable directions are transverse to the vertical.

Finally we define the notion of \emph{localised one-fibered} homoclinic orbit as follows. 

\begin{definition}\label{definition localised fibered:section3}
Let $Q$ be a homoclinic point in $W^{\sta}(P)\cap W^{\uns}(P')\setminus\{P,P'\}$. The orbit of $Q$ is \emph{localised one-fibered} if there exists $\ell\in\mathbb{Z}$ such that $f_{\gna}^\ell(Q)$ is one-fibered and $f_{\gna}^\ell(Q)\in W^{\sta}_{\mathrm{loc}}(P)\cup W^{\uns}_{\mathrm{loc}}(P')\setminus\{P,P'\}$. 

The orbit of $Q$ is said \emph{not localised one-fibered} if, for every $\ell\in\mathbb{Z}$ such that $f_{\gna}^\ell(Q)\in W^{\sta}_{\mathrm{loc}}(P)\cup W^{\uns}_{\mathrm{loc}}(P')\setminus\{P,P'\}$, the point $f_{\gna}^\ell(Q)$ is not one-fibered.
\end{definition}
Observe that if the orbit of a homoclinic point is localised one-fibered, then the orbit is also one-fibered.

Our first result counts the maximum number of points belonging to the same orbit on the same fiber. We refer to Section~\ref{sec:existence one fibered} for the proof.

\begin{lemma}\label{lemma at most two fibered points:section3}
Let $Q=\{\ss_Q,\vp_Q)$ be a homoclinic point. For all integers $\ell \in \mathbb{Z}$, except at most finitely many of them, there are at most two points $Q^{\sta}, Q^{\uns}$ in the orbit of $Q$ on the same vertical fiber $\{\ss=\ss_Q^\ell\}$, where $f_{\gna}^\ell(Q)=(\ss_Q^\ell,\varphi_Q^\ell)$. Moreover, $Q^{\sta}$ belongs to the local stable manifold of the periodic orbit, while $Q^{\uns}$ belongs to the local unstable manifold of the periodic orbit.
\end{lemma}

Next proposition shows that, except in the case that the periodic orbit is of period 2, that is $\frac p q =\frac 1 2$, all homoclinic orbits are localised one-fibered (not only the ones in the Aubry-Mather set).

\begin{proposition}\label{nec cond for not one fib:section3} 
Let $Q$ be a homoclinic point in $W^{\sta}(P)\cap W^{\uns}(P')$, $P,P' \in \mathcal{P}$. If the orbit of $Q$ is not localised one-fibered, then $\mathcal{P}$ is a periodic point of period 2, that is $\mathcal{P}\in \mathcal{M}_{\frac{1}{2}} (\Omega)$.
\end{proposition}

Again, we refer to Section~\ref{sec:existence one fibered} for the proof of this Proposition. We want now to say something in the very special case when we have a homoclinic point $Q$ whose orbit is not fibered. Thanks to the Proposition~\ref{nec cond for not one fib:section3}, we know that $Q$ is a homoclinic point of a $2$-periodic orbit. 

Let $\mathcal{P}=\{P,f_{\gna}(P)\}$ be the unique $2-$periodic orbit belonging to $\mathcal{M}_{\frac{1}{2}}(\Omega)$. 
In particular, it is contained in the zero section, i.e., $P=(\ss_0,0)$ and $f_\gna(P)=(\ss_1,0)$. 
Consider the involution: 
 \begin{equation}\label{eq:involutionI}
 \begin{split}
 \mathcal{I}\colon\mathbb{A} & \to \mathbb{A},\\
 (\ss,\vp)& \mapsto (\ss,-\vp)
 \end{split}     
 \end{equation}
Thanks to the past/future symmetry of the billiard, we have
\begin{equation}\label{involution:section3}
\mathcal{I}\circ f_{\gna}=f^{-1}_{\gna}\circ \mathcal{I}\, .
\end{equation}
In particular, we deduce that the stable and unstable manifolds of the periodic orbit $\mathcal{P}$ are symmetric with respect to the zero section. That is, the image through $\mathcal{I}$ of the stable manifold is the unstable manifold.

We can actually say more, as discussed in the next statement, whose proof is done in Section~\ref{sec:existence one fibered}.

\begin{proposition}\label{good property in the bas case:section3} 
    Let $Q$ be a homoclinic point to $\mathcal{P} \in \mathcal{M}_{\frac{1}{2}}(\Omega)$. 
    Assume that the orbit of $Q$ is not localised one-fibered and assume that it is not transverse. 
    Then there exists $\ell\in\mathbb{Z}$ such that $f^{\ell}_{\gna}(Q)\in \mathbb{T}\times\{0\}$. 
\end{proposition}

Observe that, by Proposition~\ref{good property in the bas case:section3} and by~\eqref{involution:section3}, if the orbit of $Q$ is not localised one-fibered and assuming, without loss of generality, that $Q\in\mathbb{T}\times\{0\}$, we obtain that for every $n\in\mathbb{N}$
\[
\mathcal{I}\circ f^n_{\gna}(Q)=f^{-n}_{\gna}(Q)\, .
\]

\subsection{One fibered and two fibered points}\label{sec:transversality one point}
The next step in the proof of Theorem~\ref{main thm 2 bis} is to show that, for a given non transverse homoclinic point $Q$ of the periodic orbit $\mathcal{P}=\{P_j\}_{j=0}^{q-1} $, it is possible to construct
an arbitrarily close analytical perturbation, in fact a trigonometric polynomial, of the original billiard table such that the perturbed invariant manifolds intersect transversally close to $Q$.

Since, by Aubry-Mather theory (see Proposition~\ref{prop:homoclinicpq}), the billiard map has a homoclinic point (not necessarily transverse), the set of homoclinic points is not empty. Let 
\[
Q\in  W^{\sta}(\mathcal{P}) \cap W^{\uns}(\mathcal{P})\setminus \{\mathcal{P}\} 
\]
be a non transverse homoclinic point (notice that if $Q$ is transverse, it will remain transverse by sufficiently small perturbations). 

By Propositions~\ref{nec cond for not one fib:section3} and~\ref{good property in the bas case:section3}, there are two possibilities: 
\begin{itemize}
\item[(i)] \label{i}
either $Q$ belongs to a non transverse localised one-fibered homoclinic orbit, that is, by Definition~\ref{definition localised fibered:section3}, there exists $\ell \in \mathbb{Z}$ such that $f^{\ell}_{\gna}(Q)$ is one-fibered and $f^\ell_{\gna}(Q)$ belongs to the local stable or the local unstable manifold of $\mathcal{P}$; 
\item[(ii)] \label{ii}
or, when $q=2$, the point $Q$ is a non transverse homoclinic point to $\mathcal{P}\in \mathcal{M}_{\frac{1}{2}}(\Omega)$ and for some $\ell \in \mathbb{Z}$, $f^\ell_{\gna} (Q) \in \mathbb{T}\times \{0\}$. 
\end{itemize}

To deal with the $2$-periodic case, let us to introduce the following definition.
\begin{definition}\label{def two fibered}
Let $\mathcal{P}=\{P_0,P_1=f_{\gna}(P_0)\}\in \mathcal{M}_{\frac{1}{2}}(\Omega)$ and $Q\in W^{\sta}(\mathcal{P}) \cap W^{\uns}(\mathcal{P})\setminus \{\mathcal{P}\} $ such that its orbit is not localised one-fibered. 
Assume that $Q=(\ss_0,0)$\footnote{The existence of a point on the zero section is guaranteed by Proposition~\ref{good property in the bas case:section3}.}. 
The point $Q=(\ss_0,0) $ is called a \emph{symmetric two fibered} homoclinic point. 

Let Let $(\ss_i,\varphi_i)=f^{i}_{\gna}(Q)$, and $\ell=\ell(Q)>0$ be an integer such that Lemma~\ref{lemma at most two fibered points:section3} applies. 
Let $U_Q \subset \mathbb{T}$ be a closed neighborhood of $\ss_\ell$ such that $\ss_\ell \in U_Q$ and $ \ss_i\notin U_Q $ for $i\neq \ell , -\ell$. 
The neighborhood $U_Q$ is called a \emph{symmetric two-fibered neighborhood} of $Q$.

\end{definition}

\begin{remark}
Observe that, even if  the neighborhood $U_Q$ in definition 
\ref{def two fibered} is called a \emph{symmetric two-fibered neighborhood} of $Q$, it is, in fact, a neighborhood of $\ss_\ell$, which is the first coordinate of the point $f^{\ell}_{\gna}(Q)$.
\end{remark}

As observed above, using that the billiard map possesses the past/future symmetry in~\eqref{involution:section3}, we get
\[
f^{\ell}_{\gna}(Q) = (\ss_{\ell} , \varphi_\ell), \qquad f^{-\ell}_{\gna}(Q) = (\ss_{-\ell}, \varphi_{-\ell})=(\ss_{\ell}, -\varphi_{\ell})\, .
\]

For our purposes, we are then going to deal with just two possible kinds of homoclinic points for our billiard map $f_{\gna}$, that we will need to make transverse.
\begin{enumerate}
    \item [(i)] 
    We consider a non transverse homoclinic point $h$ to $\mathcal{P}$, which is one-fibered,  belonging to the local stable or the local unstable manifold, and such that $W^{\diamond}_{\mathrm{loc}}(\mathcal{P})$, $\diamond=\uns,\sta$, are local graphs with respect to $\ss$ at $h$. 
    \item [(ii)] 
    We consider a non transverse homoclinic point $h$ to a $2$-periodic orbit, which is symmetric two-fibered point and such that $W^{\diamond}_{\mathrm{loc}}(\mathcal{P})$, $\diamond=\uns,\sta$, are local graphs with respect to $\ss$ at $f^\ell_{\gna}(h)$, for some $\ell>0$. 
\end{enumerate}

\subsection{The $\mathcal{C}^1$ distance between the perturbed invariant manifolds}\label{sec:C1distance}

By Proposition \ref{prop:hyperbolicfixedpoint:section3}, we know that the perturbed billiard map $f_{\dgnalna}$ has a (unique) hyperbolic periodic orbit $\mathcal{P} \in \mathcal{M}_{\frac p q} ( \Omega (\dgnalna) )$. Moreover, by Lemma \ref{lazutkin}, there exists $\nu>0$ such that we can assume that 
$\mathcal{M}_{\frac p q} ( \Omega (\dgnalna) )\subset \mathbb{A}_\nu$ if $\|\lambda\|_r$ is small enough.

As the periodic orbit $\mathcal{P}$ is in the Aubry-Mather set, we know that its manifolds intersect, but not necessarily in a transverse way. In this section, we show how to perturb (again) the map to obtain a transverse intersection between the manifolds.

The results of this section apply to any map, no necessarily associated to a billiard, and are independent of the previous results. Therefore, we will assume that we have a  $\mathcal{C}^k$ diffeomorphism $g: \mathbb{A} \to \mathbb{A}$, having a hyperbolic periodic orbit $\mathcal{P}$, whose stable and unstable manifolds intersect at one homoclinic point $Q$, but the intersection is not transverse.

It is a well known fact that (see e.g. \cite{HirschPughShub,Zehnder}) that, for any perturbation $\tilde g$  which is $\mathcal{C}^2$-close to $g$, the invariant manifolds of $\tilde {\mathcal{P}}$, the continuation of $\mathcal{P}$, behave $\mathcal{C}^2$ smoothly.

This allows a well established  perturbation theory (called in the literature Melnikov-Poincar\'e theory), which gives the first order perturbation of these manifolds. We will not reproduce the computations here, as they are done in \cite{Zehnder} (see also \cite{DelshamsRR96}, \cite{DelshamsRR97}) and, more suitable to us, in \cite{Genecand}, who
characterizes the $\mathcal{C}^1$ distance between the perturbed  invariant manifolds, close to a non transverse homoclinic point $Q$, for diffeomorphisms $\mathcal{C}^2$ close to $g$ as follows (see Figure~\ref{figure_billiards_1}):
\begin{itemize}
\item Let $\mathcal{U}_Q \subset \mathbb{A_\nu}$,  be a small neighbourhood of $Q$.
\item 
Let $\mathbf{t}_0$ be the common unitary tangent vector to $W^{\uns}(\mathcal{P})$ and $W^{\sta}(\mathcal{P})$ at $Q$.
\item 
Let $l_Q$ be a straight line through $Q$ orthogonal to $\mathbf{t}_0$.
\item 
For a given 
$\varrho>0$ and $\nu\ge 0$ consider the set $\Bf{\varrho}{g}$ defined as in~\eqref{ball fgamma}. 
Observe that, if $k\ge 5$, then $\|\tilde{g}-g\|_{\difknu{2}{\nu}} <\varrho$.

Let $\tilde{g} \in\Bf{\varrho}{g}$. 
From the classical perturbation theory (see e.g.~\cite{HirschPughShub}), we deduce that, if $\varrho$ is small enough, $\tilde{g}$ has a saddle periodic orbit $\widetilde {\mathcal P} =\big \{\tilde{P}_i\big \}_{i=0}^{q-1}$ close to $ \mathcal{P}$ with associated stable and unstable manifolds $W^{\diamond}(\widetilde{\mathcal{P}})$, $\diamond=\uns,\sta$ (see Definition~\ref{definition local invariant manifold:section3}). 

Let $\rho$ be such that $Q\in W_\rho \cap W'_\rho$,  with 
$W_\rho\subset W^{\uns}(P_n)$, $W'_\rho \subset W^{\sta}(P_m)$ defined in \eqref{def:branch}, for some $n,m\in \{0,\cdots, q-1\}$. 
In addition, $W_\rho^{\diamond}({\widetilde{\mathcal{P}}})$ are $\mathcal{C}^2$-close to $W_\rho^{\diamond}({\mathcal{P}})$, $\diamond=\uns,\sta$  and then we can define the points
\[
Q^{\uns}_0(\tilde{g}):=W^{\uns }_\rho(\tilde{P}_n)\cap l_Q  , \qquad 
Q^{\sta}_0(\tilde{g}):=W^{\sta}_\rho(\tilde{P}_m)\cap l_Q.
\]
Notice that $Q_0^{\uns}(g)=Q_0^{\sta}(g)=Q$.
\item 
For $\diamond=\uns,\sta$, let $\mathbf{t}^{\diamond}(\tilde{g}) $ be the tangent vector to $W^{\diamond}({\widetilde{\mathcal{P}}})$
at $Q^{\diamond}_0(\tilde{g})$ such that the vector $\mathbf{t}^{\diamond} (\tilde{g})- \mathbf{t}_0$ is orthogonal to $\mathbf{t}_0$.
\item 
There exists $\varrho_0>0$ small enough such that, for all $0<\varrho \leq \varrho_0$,
\begin{equation}\label{definition Phi, Phistauns}
\begin{array}{rcl}
\Phi \colon\Bf{\varrho}{g} &\to & \mathbb{R}^2\\
\tilde{g}  &\mapsto &\left  ( {\mathbf{t}_0} \wedge \big ( Q^{\uns}_0(\tilde{g}) - Q^{\sta}_0(\tilde{g})\big ), 
{\mathbf{t}_0} \wedge \big ( \mathbf{t}^{\uns}_0(\tilde{g}) - \mathbf{t}^{\sta}_0(\tilde{g})\big )
\right )
\end{array}
\end{equation}
is well defined. Here the wedge product is defined as
\[
\mathbf{u} \wedge \mathbf{v} = u_1 \, v_2 - u_2 \, v_1, \qquad \mathbf{u}= (u_1,u_2)^\top, \; \mathbf{v}=(v_1,v_2)^\top,
\]
and we recall that $\mathbf{u}\wedge \mathbf{v} = \| \mathbf{u} \| \| \mathbf{v}\| \sin \theta$ where $\theta$ is the angle between the vectors $\mathbf{u}, \mathbf{v}$. 
\end{itemize}

\begin{figure}[ht]
\begin{overpic}[height=0.9\textwidth, angle=90]{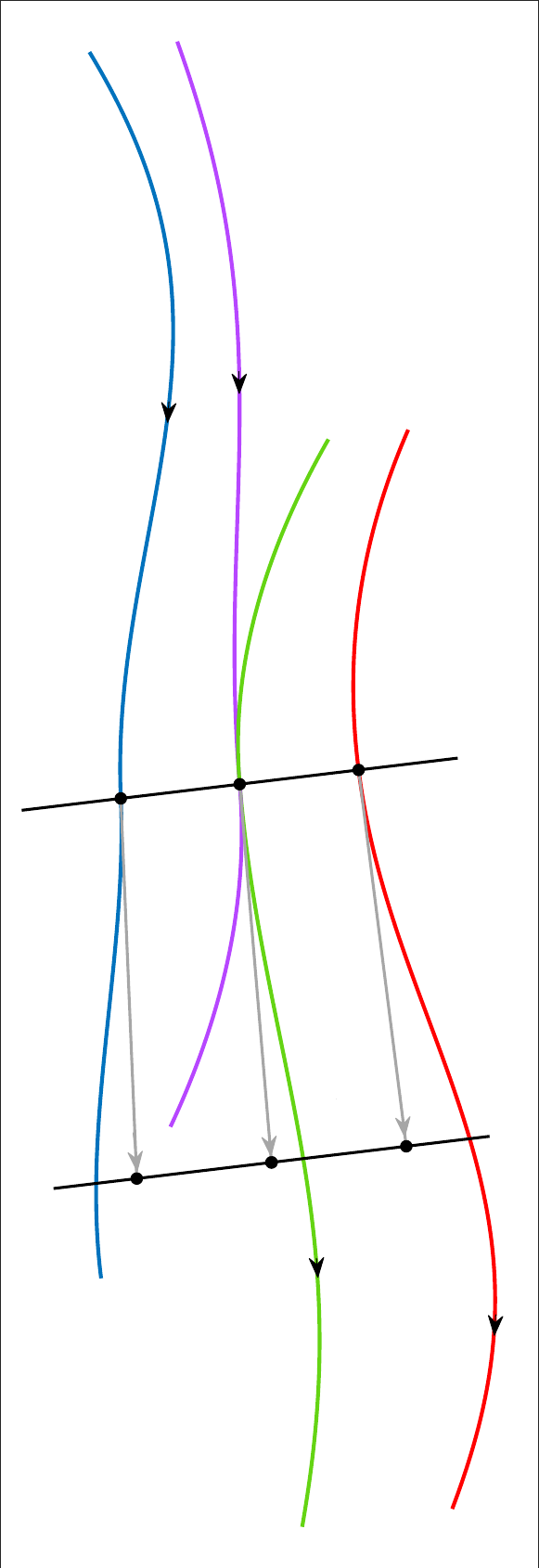}
	\put(43,5){$Q_0^{\mathrm{u}}(\tilde{g})$ }
    \put(42,20){$Q_0^{\mathrm{s}}(\tilde{g})$}
    \put(53,9){$\mathbf{t}^{\mathrm{u}}(\tilde{g})$}
    \put(65,22){$\mathbf{t}^{\mathrm{s}}(\tilde{g})$}
    \put(15,7){$W^{\mathrm{u}}(\tilde{P}_n)$}
    \put(12,17){$W^{\mathrm{u}}({P}_n)$}
    \put(80,28){$W^{\mathrm{s}}(\tilde{P}_m)$}
    \put(85,17){$W^{\mathrm{s}}({P}_m)$}
    \put(47,12){$Q$}
    \put(67,14){$\mathbf{t}_0$}
    \put(45,29){$l_Q$}
    \end{overpic}
    \caption{Depiction of the $\mathcal{C}^1$ distance between  $W^{\uns}(\widetilde{\mathcal{P}})$ and $W^{\sta}(\widetilde{\mathcal{P}})$.
}
\label{figure_billiards_1}  
\end{figure}
Observe that $\Phi (g)=(0,0)$ and 
\begin{equation}\label{property Phi, Phistauns}
\Phi(\tilde{g}) = \left (\pm \| Q_0^{\uns}(\tilde{g}) - Q_0^{\sta}(\tilde{g}) \|, \pm \| \mathbf{t}_0^{\uns} ( \tilde{g}) - \mathbf{t}_0^{\sta}(\tilde{g}) \|\right ).
\end{equation}

\begin{remark}
Later, our  goal will be to find a billiard map $\tilde{f}$ such that $\Phi(\tilde{f})= (a,b)$ for any prescribed small enough $(a,b) \in \mathbb{R}^2$. Notice that, when $a=0$, $b\neq 0$, the perturbed billiard $\tilde{f}$ possesses a transverse homoclinic point whereas $a\neq 0$ means that there are not homoclinic points in a neighbourhood of $Q$. 
\end{remark}

To this end we will need to compute  (see Proposition~\ref{prop:Genecand} below) the quantity $d\Phi (g)h$ with $h \in \mathcal{C}^{k-1}$. For $i\in \mathbb{Z}$, we introduce $Q_{i+1}=f^{i+1}(Q)=f(Q_{i})$, where $Q_0=Q$, 

\begin{equation}\label{defQti}
Q_{i+1}^{\uns,\sta} (\tilde{g})= \tilde{g}^{i+1}(Q_0^{\uns,\sta} (\tilde{g})),\qquad \mathbf{t}_{i+1} = D{g}^i (Q_0) \mathbf{t}_0  
\end{equation}
and we notice that $\mathbf{t}_{i+1} = Dg(Q_i) \mathbf{t}_i$, if $i\geq 0$ and $\mathbf{t}_{i} = Dg^{-1} (Q_{i+1}) \mathbf{t}_{i+1}$ when $i<0$. 

In addition, for a given $\mathcal{C}^{k-1}$ diffeomorphism $h$, letting $g_\sigma= g+ \sigma h$, we define
\[
\delta_h Q^{\uns,\sta}_i(g) = \frac{d}{d \sigma} \big ( Q^{\uns,\sta}_i(g_\sigma) \big )_{|\sigma=0} = \lim_{\sigma\to 0} \frac{Q^{\uns,\sta}_i(g+\sigma h) - Q^{\uns,\sta}_i(g) }{\sigma}.
\]

Next proposition, which was proven in~\cite{Genecand}, will be useful for our purposes.
 
\begin{proposition}\label{prop:Genecand} Fix $\nu>0$. 
There exists $\varrho_\nu>0$ such that, for all $0<\varrho<\varrho_\nu$ the map 
$ \Phi:\Bf{\varrho}{g} \subset \mathcal{C}^{k-1}\to \mathbb{R}^2$, given in \eqref{definition Phi, Phistauns}, is $\mathcal{C}^1$ Fr\'echet differentiable and
\[
d\Phi(g) h = \big (d \Phi_1(g) h, d \Phi_2 (g) h\big )
\] 
with
\begin{align*}
d\Phi_1(g) h & = \sum_{i\in \mathbb{Z}} \mathbf{t}_{i+1} \wedge h(Q_i)
\\ 
d\Phi_2 (g) h &= \sum_{i\in \mathbb{Z}} \mathbf{t}_{i+1} \wedge 
\big [ D h (Q_i)  \mathbf{t}_i + D^2 g (Q_i) \big ( \delta_h Q_i^\sigma(g), \mathbf{t}_i\big )\big ],
\end{align*}
taking $\sigma=\sta$ if $i\geq 0$ and $\sigma  = \uns$ if $i<0$ and where $Dh(Q_i)$ denotes the differential of $h$ at $Q_i$ 
and 
$D^2 g (Q_i) \big ( \delta_h Q_i^\sigma(g), \mathbf{t}_i\big )$ denotes the second differential of $g$ at $Q_i$ applied at $\big ( \delta_h Q_i^\sigma(g), \mathbf{t}_i\big )$.
\end{proposition}

\subsection{Compactly supported perturbations and surjectivity of $d\Psi(0)$}\label{sec:compactly suported:sec3}
Consider the billiard table  obtained in Proposition \ref{prop:hyperbolicfixedpoint:section3} and rename it as $\gamma$. 
Assume that, for  the associated map $f_\gna$,  the stable and unstable manifolds of $\mathcal{P}$ intersect in a non transverse homoclinic point $Q$.
Recall that in Proposition \ref{prop Frechet Gamma nu} we considered the functional $\Gamma_\nu$ which sends every perturbed billiard table of the form~\eqref{def gamma lambda bis}, the billiard map $f_{\dgalac}$, associated to the perturbed billiard table written in the arc lengh coordinates of $\gamma$ \eqref{def gamma lambda arc length}. 
Then, we consider the function $\Phi$ defined in~\eqref{definition Phi, Phistauns} applied to this map.
From Propositions~\ref{prop:Genecand} and~\ref{prop Frechet Gamma nu} one deduces that the functional 
\begin{equation}\label{definiton operator Psi}
\Psi:=\Phi \circ \Gamma_\nu: \Bl{\varrho} \to \mathbb{R}^2
\end{equation}
is $\mathcal{C}^1$-Fr\'echet differentiable. 

The key result in the proof of Theorem \ref{main thm 2 bis},  that will be done in Section~\ref{sec2:final proof}, is the surjectivity of  $d\Psi (0)$ on $\mathcal{C}^\omega_r(\mathbb{T})$. 
As we know that $\Psi$ is $\mathcal{C}^1$-Fr\'echet differentiable, it would be enough to check that 
$d\Psi (0)$ is surjective when it is restricted to compactly supported functions. 
Notice that, in this case,  formulas in Proposition~\ref{prop:Genecand} become simpler.  

Let us first introduce the matrix
\begin{equation}\label{def mathcalT}
\mathcal{T}(s,\varphi,\varphi')= \left (
\begin{array}{cc}
\sin \varphi (\cos \varphi + \cos \varphi') & 0 
\\ 
-\mathcal{K}(s) \sin \varphi (\cos \varphi - \cos \varphi') & \cos \varphi (\cos\varphi + \cos \varphi')   
\end{array}\right ).
\end{equation}

The following two results, Lemma~\ref{lem: computing the differential} and~\ref{lem: computing the differential ii}, which are proven in Section \ref{sec:prooflemmas}, provide explicit formulas for $d\Psi(0)\lambda$ when the homoclinic point $Q$ is on the case (i) and (ii) respectively, described at the end of Section~\ref{sec:transversality one point}.

\begin{lemma} \label{lem: computing the differential}
Let $Q=(\ss_0,\varphi_0)$ be a one-fibered (see Definition~\ref{definition one fibered}), non transverse homoclinic point to the periodic orbit  $\mathcal{P}=(P_j)_{j=0,\cdots, q-1}$ of the billiard map $f_\gna$ and let $U_Q$ be a one-fibered neighborhood of $Q$. Let $\mathcal{C}^\infty_{\mathrm{supp}} (U_Q, \mathbb{R}) $ be the set of $\mathcal{C}^\infty$ functions compactly supported on $U_Q\subset \mathbb{T}$. 

Then, the function 
$d\Psi(0)=d(\Phi \circ \Gamma_\nu)(0) : \mathcal{C}^\infty_{\mathrm{supp}} (U_Q,\mathbb{R}) \to \mathbb{R}^2 $ is given by
\[
d(\Phi_1 \circ \Gamma_\nu)(0) \lna=  
\frac{1}{\cos \varphi_0 \cos\varphi_1} \mathbf{t}_0 \wedge  
\mathcal{T}(s_0,\varphi_0,\varphi_1) \left (\begin{array}{c}  \lac(s_0) \\ \derlac(s_0) \end{array}\right )
\]
with $ \lac(s)=\frac{1}{|\partial \Omega (\gna)|} \lna\circ \sigma^{-1} (s)$, $s_0=\sigma (\ss_0)$, $(s_1,\varphi_1)=f_{\ga}(s_0,\vp_0)$ and $\mathcal{T}$ defined in~\eqref{def mathcalT}.

In addition, when $\lambda(\ss_0)=\derlna(\ss_0)=0$ 
\[
d(\Phi_2 \circ \Gamma_\nu)(0) \lna= 
\dertwolac(s_0) (\pi_1 \mathbf{t}_0)^2 \left (\frac{\cos \varphi_0}{\cos \varphi_1} + 1 \right )  + \mathbf{w}
\]
where $\mathbf{t}_0=(\pi_1 \mathbf{t}_0,  \pi_2 \mathbf{t}_0)$ and $\mathbf{w}$ depending on $d\Gamma_\nu(0)\lambda$ 
but independent on $D[d\Gamma_\nu(0)\lambda]$. 
\end{lemma}

Now we compute $d\Psi(0)\lambda$ when the homoclinic point is in the case (ii). Indeed, when $q=2$  we cannot guarantee the existence of localised one-fibered homoclinic points so that the formulas for $d\Psi(0)\lambda$ will be more complicated (see Definition~\ref{def two fibered}).  

\begin{lemma} \label{lem: computing the differential ii} 
Let $Q=(\ss_0,0)$ 
be a symmetric two fibered non transverse homoclinic point to the periodic orbit  $\mathcal{P}=(P_j)_{j=0,1} \in \mathcal{M}_{\frac{1}{2}}(\Omega)$ of the billiard map $f_\gna$ and let $U_Q $, a
symmetric two-fibered neighborhood of $Q$, and $\ell$ as in Definition~\ref{def two fibered}. 
Let $\mathcal{C}^\infty_{\mathrm{supp}} (U_Q,\mathbb{R}) $ be the set of $\mathcal{C}^\infty$ functions compactly supported on $U_Q \subset \mathbb{T}$. 

Then, the function $d\Psi(0)=d(\Phi \circ \Gamma_\nu)(0) : \mathcal{C}^\infty_{\mathrm{supp}} (U_Q,\mathbb{R}) \to \mathbb{R}^2 $ is 
\begin{align*}
d(\Phi_1 \circ \Gamma_\nu)(0) \lna=  &\frac{1}{\cos \varphi_\ell \cos\varphi_{\ell+1}}
\mathbf{t}_{\ell} \wedge  
\mathcal{T}(s_\ell,\varphi_\ell,\varphi_{\ell+1}) \left (\begin{array}{c}  \lac(s_\ell) \\ \derlac(s_\ell) \end{array}\right )  
\\ 
& + \frac{1}{\cos \varphi_\ell \cos\varphi_{-\ell +1}}
\mathcal{I}(\mathbf{t}_{\ell}) \wedge  
\mathcal{T}(s_{\ell},-\varphi_{\ell},\varphi_{-\ell+1}) \left (\begin{array}{c}  \lac(s_\ell) \\ \derlac(s_\ell) \end{array}\right ) 
\end{align*}
with $ \lac(s)=\frac{1}{|\partial \Omega (\gna)|} \lna\circ \sigma^{-1} (s)$, $s_0=\sigma(\ss_0)$,
with $(s_{i},\varphi_i)=f_{\ga}^i(s_0, 0)$, $i\in \mathbb{Z}$, $\mathcal{T}$ defined in~\eqref{def mathcalT}, $\mathbf{t}_\ell = Df^{\ell}(Q) \mathbf{t}_0$, $\mathbf{t}_0$ is a common tangent vector to $W^{\sta,\uns}(\mathcal{P})$ at $Q$ (see~\eqref{defQti}) and $\mathcal{I}$ the involution $\mathcal{I}(s,\varphi)=(s,-\varphi)$. 
 
In addition, if $\lambda(\ss_\ell)=\derlna(\ss_\ell)=0$, with $\ss_\ell = \sigma^{-1}(s_\ell)$, then
\[  
d(\Phi_2 \circ \Gamma_\nu)(0) \lna=    
\dertwolac(s_\ell) (\pi_1 \mathbf{t}_\ell)^2 \left (2+ \frac{\cos \varphi_\ell}{\cos \varphi_{\ell+1}} + 
\frac{\cos \varphi_\ell}{\cos \varphi_{-\ell+1}}\right )  + \mathbf{w}
\]
where $\mathbf{t}_\ell=(\pi_\ell \mathbf{t}_\ell,  \pi_2 \mathbf{t}_\ell)$  and $\mathbf{w}$ depends on $d\Gamma_\nu(0)\lambda$ but it is independent on $D[d\Gamma_\nu(0)\lambda]$. 
\end{lemma}

These results enable us to prove, in  Section \ref{sec:prooflemmas}, the following statement.

\begin{lemma}\label{exhaustive}
Let $Q=(\ss_0,\varphi_0)$ be either a one-fibered or a two-fibered symmetric non transverse homoclinic point of the billiard map $f_\gna$. 
In the former $U_Q$ is a one-fibered neighborhood of $Q$ and in the latter, let $\ell>0$ and $U_Q $ (a
symmetric two-fibered neighborhood of $Q$) be as in Definition~\ref{def two fibered}. 

Assume, if $Q$ is one-fibered, that $\pi_1\mathbf{t}_0 \neq 0$ and if $Q\in \mathbb{T}\times \{0\}$ is two-fibered, then $\pi_1 \mathbf{t}_\ell \neq 0$.  
Denote by $\mathcal{C}^\infty_{\mathrm{supp}} (U_Q,\mathbb{R})$ the set of $\mathcal{C}^\infty$ compactly supported at $U_Q$. 

The linear map $d (\Phi \circ \Gamma_\nu)(0): \mathcal{C}^\infty_{\mathrm{supp}} (U_Q,\mathbb{R}) \to \mathbb{R}^2$ is surjective.
\end{lemma}

Let us show now that $d\Psi(0)$ is surjective on $\mathcal{C}^\omega_r(\mathbb{T})$, in fact in the set of trigonometric polynomials. We follow now Zehnder's approach. Recall that, given two vectors $\mathbf{u},\mathbf{v}\in\mathbb{R}^2$, the notation $\vert \mathbf{u}\ \ \mathbf{v}\vert$ stands for the determinant of the matrix $(\mathbf{u}, \mathbf{v})$.

\begin{proposition}\label{invertible}
Assume that the billiard map $f_\gna$ has either $Q=(\ss_0,\varphi_0)$ a one-fibered non transverse homoclinic point with $\pi_1\mathbf{t}_0\neq 0$ or $Q=(\ss_0,0)$ a two-fibered symmetric non transverse homoclinic point with $\pi_1 \mathbf{t}_\ell \neq 0$ where $\ell $ is defined in Definition~\ref{def two fibered}. 

Then there exist trigonometric polynomials $\lambda_1,\lambda_2$ such that
\[
\left | d(\Phi\circ\Gamma_\nu)(0)\lambda_1  \quad  d(\Phi\circ\Gamma_\nu)(0)\lambda_2\right | \neq 0\, .
\]
As a consequence, $d(\Phi \circ \Gamma_\nu)(0): \mathcal{C}^{\omega}_r(\mathbb{T},\mathbb{R}) \to \mathbb{R}^2$ is surjective. 
\end{proposition}
\begin{proof}
By Lemma~\ref{exhaustive} we know that
$d(\Phi\circ\Gamma_\nu)(0)\colon \mathcal{C}^\infty_{\mathrm{supp}}(U_Q,\mathbb{R}) \to \mathbb{R}^2$ is surjective; this means that there exist $\tilde{\lambda}_1,\tilde{\lambda}_2\in \mathcal{C}^\infty_{\mathrm{supp}}(U_Q,\mathbb{R}) \subset \mathcal{C}^\infty(\mathbb{T},\mathbb{R})$ such that
\[
\left | d(\Phi\circ\Gamma_\nu)(0)\tilde{\lambda}_1  \quad  d(\Phi\circ\Gamma_\nu)(0)\tilde{\lambda}_2\right | \neq 0\, .
\]
Since the functional $\Phi\circ \Gamma_\nu$ is $\mathcal{C}^1$-Fr\'echet differentiable, this last condition will still hold for sufficiently small perturbations of $\tilde{\lambda}_1,\tilde{\lambda}_2$ in the $\mathcal{C}^2$-norm. To obtain the required perturbations $\lambda_1,\lambda_2$, we can then consider the trigonometric polynomials obtained by truncating the Fourier series of $\tilde{\lambda}_1,\tilde{\lambda}_2$ to high enough order. The result follows.
\end{proof}

\subsection{Analytic perturbation: transversality of the invariant manifolds}
\label{sec:analytic perturbations}
We are now going to show that we can perturb analytically our initial billiard table $\Omega (\gamma)$ so that the perturbed billiard map has a periodic point with a transverse homoclinic point. 

From Proposition~\ref{invertible}, we can then deduce the following result.

\begin{corollary}\label{corollary inverstible}
Let $Q$ be a one-fibered non transverse homoclinic point for $f_\gna$ with $\pi_1 \mathbf{t}_0\neq 0$ or a two-fibered symmetric non transverse homoclinic point to a hyperbolic periodic orbit $\mathcal{P}$. Then, for any $\varepsilon>0$ small enough, there exists a trigonometric polynomial  $\lna\in \mathcal{C}^\omega(\mathbb{T},\mathbb{R})$ with $\|\lna\|_r \leq \varepsilon$ such that the billiard map $f_{\dgnalna}$, associated to the perturbed billiard table $\dgnalna=\gna+\lna \cdot \nna $, has a transverse homoclinic point $\widetilde{Q}$ which is the continuation of $Q$, associated to the continuation $\widetilde{\mathcal{P}}$ of $\mathcal{P}$.    
\end{corollary}
\begin{proof}
Let $\lambda_1,\lambda_2$ be trigonometric polynomials given by Proposition~\ref{invertible}. Consider the function
\begin{equation*}
F\colon \left\{
\begin{array}{rcl}
\R^3 &\to& \mathbb{R}^2\\
(c_1,c_2,b) &\mapsto& F(c_1,c_2,b)=\Phi\circ\Gamma_\nu(c_1\lambda_1+c_2\lambda_2)-(0,b)\, .
\end{array}
\right. 
\end{equation*}
Observe that $F(0,0,0)=(0,0)$, and $\frac{\partial F}{\partial(c_1,c_2)}(0,0,0)$ is invertible by Proposition~\ref{invertible}. By Implicit Function Theorem, there exists $c_1=c_1(b),c_2=c_2(b)$ such that
\[
F(c_1,c_2,b)=(0,0)\, ,
\]
i.e., the corresponding perturbation 
$ \lna= c_1(b) \lambda_1 + c_2(b) \lambda_2$ 
of the billiard table provides a billiard map 
$f_{\dgalac}$, associated then to the boundary 
$\dgalac (s)= \ga(s) + \narc(s) \lac(s)$ (see \eqref{def gamma lambda arc length}) which exhibits a transverse homoclinic point $\widetilde{\mathcal{Q}}$ as soon as $b\neq 0$. On the other hand, as $c_1(0) = c_2(0)=0$ and $c_1,c_2$ are smooth functions of $b$, 
\[
\|\lna\|_r \leq |c_1(b)| \|\lambda_1\|_r + |c_2(b) | \|\lambda_2\|_r \leq \varepsilon
\]
provided $b$ is small enough. 

Finally, we stress that, as $f_{\dgnalna} = S^{-1} \circ f_{\dgalac} \circ S $, (see \eqref{conjugation with arclength} and \eqref{change arc length}), and $S$ is a diffeomorphism the map $f_{\dgnalna}$ also has a transversal homoclinic point $\widetilde Q=S^{-1}(\widetilde{\mathcal{Q}})$. 
\end{proof}

\subsection{Density of $\mathcal{V}^{p/q}_{r,N}$. End of the proof of Theorem~\ref{main thm 2 bis}}\label{sec2:final proof} 
Let $\gna \in \mathcal{B}_r$. Let us observe that, by Proposition~\ref{prop:hyperbolicfixedpoint:section3}, we can make an arbitrarily small perturbation through a trigonometric polynomial $\lna\in \mathcal{B}_r$ such that $\dgnalna=\gna+\lna\cdot \nna$ is a strongly convex analytic billiard and that the associated billiard map $f_{\dgnalna}$ has a unique periodic orbit $\mathcal{P}= (P_i)_{i=0,\dots, q-1}$ in the Aubry Mather set of rotation number $\frac{p}{q}$, which moreover is hyperbolic. Therefore property in item~\ref{it_a} of Theorem~\ref{main thm 2 bis} is proven. Let us denote again by $\gna$ the new billiard parametrization $\dgnalna$.

Item~\ref{it_b} in Theorem~\ref{main thm 2 bis} is a consequence of the following proposition. 

\begin{prop}\label{prop:existence homoclinic point:section3}
Let $r>0$ and $\frac p q \in\mathbb{Q}/\mathbb{Z}$. Let $\gna\in  \mathcal{B}_r$ be such that its billiard map has a unique periodic orbit $\mathcal{P}$ in the Aubry-Mather set of rotation number $\frac p q$, which is hyperbolic. 
For any $\epsilon>0$ there exists a trigonometric polynomial $\lambda\in \mathcal{C}^\omega_r(\mathbb{T},\mathbb{R})$ with $\Vert \lambda\Vert_r<\epsilon$ such that, for the billiard map associated to $\gna+\lna\cdot \nna$, the set of transverse homoclinic points to the continuation of $\mathcal{P}$ \( ( \textit{which remains the only periodic orbit in the Aubry-Mather set of rotation number $\frac p q$} )\) is not empty.
\end{prop}

\begin{proof}
By Proposition~\ref{prop:homoclinicpq}, as $\mathcal{P}$ is the unique periodic orbit in the Aubry-Mather set $\mathcal{M}_{\frac{p}{q}} (\Omega)$, there exists a homoclinic point $Q\in  W^\uns(\mathcal{P}) \cap W^{\sta}(\mathcal{P})\setminus\{\mathcal{P}\}$. Moreover, any point of the orbit of $Q$ is one-fibered and by the twist condition (up to select another point of the orbit of $Q$), we can assume that, $\mathbf{t_0}$, the tangent vector at $Q$ of $W^{\sta}({\mathcal{P}})$ satisfies $\pi_1 \mathbf{t_0}\neq 0$. 
Thus, for a given $\epsilon>0$, by Corollary~\ref{corollary inverstible}, there exists a trigonometric polynomial $\lambda\in \mathcal{C}^\omega_r(\mathbb{T},\mathbb{R})$, $\|\lambda\|_r<\epsilon$, so that $\tilde Q$, the continuation of $Q$, is a transverse homoclinic point for $\tilde{\mathcal{P}}$, the continuation of $\mathcal{P}$. For $\epsilon>0$ small enough, $\mathcal{P}$ is still the only periodic orbit in the Aubry-Mather set of rotation number $\frac p q$ and moreover it is hyperbolic.
\end{proof}

From Proposition~\ref{prop:existence homoclinic point:section3}, and renaming again the new $\dgnalna$ by $\gna$, we assume that the billiard map $f_\gna$ associated to $\gna \in \mathcal{B}_r$ has at least one transverse homoclinic point which belongs to $W^{\sta}(\mathcal{P}) \cap W^{\uns}(\mathcal{P})\setminus\{\mathcal{P}\}$ with $\mathcal{P}$ the unique periodic orbit in $\mathcal{M}_{\frac{p}{q}}(\Omega)$, which moreover is hyperbolic.
In the remainder of this section we prove the density of the property in item~\ref{it_c}. 

\begin{remark}\label{rmk cases non transverse homoclinic}
As already observed at the end of Section~\ref{sec:transversality one point}, recall that, if $Q \in W^{\sta}(\mathcal{P}) \cap W^{\uns} (\mathcal{P})\setminus\{\mathcal{P}\}$ is a non transverse homoclinic point, by Propositions~\ref{nec cond for not one fib:section3} and~\ref{good property in the bas case:section3}, there are only two possibilities.
\begin{enumerate}[label=(\alph*)]
\item \label{pos_a} Either there exists $\ell_Q \in \mathbb{Z}$ such that $f^{\ell_Q}(Q) \in W_{\mathrm{loc}}^{\sta}(\mathcal{P}) \cap W^{\uns} (\mathcal{P})$ is a non-transverse one-fibered homoclinic point (see Definitions~\ref{defi one fibered orbit} and~\ref{definition localised fibered:section3}). In particular, since the periodic orbit $\mathcal{P}$ belongs to the Aubry-Mather set $\mathcal{M}_{\frac{p}{q}} (\Omega)$, $W_{\mathrm{loc}}^{\sta}(\mathcal{P})$ satisfies $\pi_1 \mathbf{t} \neq 0$ with $\mathbf{t}$ the tangent vector of $W_{\mathrm{loc}}^s(\mathcal{P})$ at $f^{\ell_Q} (Q)$.
\item \label{pos_b} Or, $q=2$ and $Q$ is a non-transverse homoclinic point to $\mathcal{P}\in \mathcal{M}_{\frac{1}{2}}(\Omega)$ such that, for some $\ell_Q \in \mathbb{Z}$, the point $f^{\ell_Q} (Q)$ is a symmetric two fibered homoclinic point (see Definition~\ref{def two fibered}). 
\end{enumerate}

That is, for any non-transverse homoclinic point $Q$ there exists $\ell_Q$ such that $f^{\ell_Q}(Q)$ satisfy the conditions of Corollary \ref{corollary inverstible}. 
\end{remark}

We first perform a perturbation to avoid the degenerate case of coincident branches. 

\begin{lemma}\label{non coincident branches} 
Assume that $f=f_\gna$ has a hyperbolic periodic orbit $\mathcal{P} = \big \{P_i\big\}_{i=0,\cdots, q-1}$ belonging to the Aubry-Mather set of rotation number $\frac{p}{q}$.  
For $\diamond=\uns,\sta$, consider $\{W^{i,\diamond}\}_{i=0,\cdots,2q-1}$ the branches  of 
$W^\diamond(\mathcal{P})$.  

Then, for all $\epsilon>0$ there exists a trigonometric polynomial $\lambda$, with  $\|\lambda\|_r <\epsilon$, such that  the billiard map $f_{\dgnalna}$ associated to $\dgnalna$ has a periodic orbit, which is the continuation of $\mathcal{P}$, whose stable and unstable manifolds satisfy $W^{i,\sta}\neq W^{j,\uns}$, for all $i,j \in \{1, \cdots, 2q-1\}$.
\end{lemma}
\begin{proof}
The proof follows straightforwardly from Corollary~\ref{corollary inverstible}. 
If for some $i,j$ one has $W^{i,\sta} = W^{j,\uns}$, taking any $Q \in W^{i,\sta} = W^{j,\uns}$ (which is then a homoclinic point), there exists $\ell_Q$ such that $f^{\ell_Q}(Q)$ satisfies the conditions on  Corollary \ref{corollary inverstible} and, after an arbitrarily small perturbation of the billiard table, the continuation of the branches, $\widetilde{W}^{i,\sta}$, $\widetilde{W}^{i,\uns}$ intersect transversally in a neighbourhood of $Q$. In particular, they do not coincide. Since the number of branches is finite, after this procedure we end up with a billiard table having no coinciding branches. 
\end{proof}

Now, we need to introduce the notions of adapted charts and order of a tangency.
In order to do so, we consider the $k$-jet of a function. Let $I\subset \mathbb{R}$ be an open interval and let $g\in C^\infty(I,\mathbb{R})$. Let $s\in I$ and $k\geq 0$. The $k$-jet of the function $g$ at the point $s$ is the vector in $\mathbb{R}^{k+1}$
\[
J^k_s(g):=(a_0,a_1,\dots,a_k)\in\mathbb{R}^{k+1}\, ,
\]
where $a_i=g^{(i)}(s)$ for all $0\le i\le k$.

\begin{definition}[Adapted charts, order of a homoclinic point] \label{definition adapted chart}
Let $Q$ be a homoclinic point to a hyperbolic periodic orbit $\mathcal{P}$ of an analytic diffeomorphism. 
An \emph{adapted chart} for $Q$ is a pair $(U,\phi)$ where $U$ is a neighborhood of $Q$ and $\phi\colon U\to\mathbb{R}^2$ is an analytic diffeomorphism onto its image $V:=\phi(U)$, such that:
\begin{itemize}
\item 
$\phi(Q)=0$;
\item 
denoting by $cc(W^{\sta}(\mathcal{P})\cap U, Q)$ the connected component of $W^{\sta}(\mathcal{P})\cap U$ containing $Q$, we have $\phi(cc(W^{\sta}(\mathcal{P})\cap U, Q))=(\mathbb{R}\times\{0\})\cap V$;
\item 
denoting by $cc(W^{\uns}(\mathcal{P})\cap U, Q)$ the connected component of $W^{\uns}(\mathcal{P})\cap U$ containing $Q$, we have that $\phi(cc(W^{\uns}(\mathcal{P})\cap U, Q))$ is the graph of a function $\g\in C^\omega((\mathbb{R}\times\{0\})\cap V,\mathbb{R})$.
\end{itemize}
For $k\ge 0$, let us denote by  $J^k_0(\g)\in \mathbb{R}^{k+1}$ the $k$-jet of the function $\g$ at $0$. 
The~\emph{order} of the homoclinic point $Q$ is then defined as 
\[
\mathrm{order}(Q):=\min\{k\ge 0 :\ J^k_0(\g)\neq 0_{\mathbb{R}^{k+1}}\}\, .
\]
Note that $\mathrm{order}(Q)$ is independent of the choice of adapted charts. 
    
Moreover, if $\mathcal{P}$ has no homoclinic connection (i.e., no coinciding branches), by analyticity of the invariant manifolds, we have $\mathrm{order}(Q)<+\infty$, for any homoclinic point $Q$. 
\end{definition}

The following result assures that, if two branches do not coincide, then the set of homoclinic points belonging to them is countable. 
 
\begin{lemma}\label{lemm_discrete}
Let $g$ be a two-dimensional analytic diffeomorphism with a hyperbolic periodic orbit $\mathcal{P}$. Let $W\subset W^\sta(\mathcal{P})$ and $W'\subset W^\uns(\mathcal{P})$ be two branches such that $W\neq W'$. Then the set of homoclinic points $h \in W\cap W'$ is discrete, hence at most countable. 

As a result, for any $N\in\mathbb{N}$ the set  $W_N\cap W_N'$ is finite.
\end{lemma}
\begin{proof}
Let us a fix a homoclinic point $h\in W\cap W'$. 
Given $\varrho>0$, let us denote by $W_\varrho(h),W_\varrho'(h)$ the $\varrho$-neighborhoods of $h$ within $W,W'$, respectively, for the distance induced by the Riemannian metric on these leaves. 
We want to show that for $\varrho>0$ sufficiently small, $W_\varrho(h)\cap W_\varrho'(h)=\{h\}$. 
By analyticity of $W,W'$, after passing to an analytic chart, we can assume that $h=(0,0)$, $W_\varrho(h)\subset\mathbb{R} \times \{0\}$ is a horizontal segment, and $W_\varrho'(h)=\{(x,\g(x)):|x|\text{ small}\}$ is the graph of some analytic function $\g\colon x \mapsto \sum_{k=0}^{+\infty} a_k x^k$, with $\g(0)=a_0=0$. 
If $a_k=0$ for all $k\geq 0$, then by analyticity, $W=W'$, contrary to our assumption. 
Otherwise, $k_0:=\min\{k\geq 0:a_k \neq 0\}<+\infty$, and then $\g(x)\sim_0 a_{k_0} x^{k_0}$. 
Therefore, $\g$ has an isolated zero at $0$, and for $\varrho>0$ sufficiently small, $W_\varrho(h)\cap W_\varrho'(h)=\{h\}$. 

In particular, recalling that $W_N $, resp. $W_N' $ (see~\eqref{def:branch}) are the points in $W \subset W^\sta(\mathcal{P})$, resp. $W'\subset W^{\uns}(\mathcal{P})$, at distance (induced on the stable, resp. unstable, branch) at most $N$ from the periodic point at the which the branch is based, there are at most finitely many homoclinic points in $W_N\cap W_N'$.
\end{proof}

By Lemma~\ref{non coincident branches} we can assume that our initial billiard map $f_\gna$ has no coinciding branches. Then, by Lemma~\ref{lemm_discrete}, for any $N\in\mathbb{N}$ and for any $i,j\in \{0,\cdots,2q-1\}$, the set of homoclinic points $ W^{i,\sta}_N\cap W^{j,\uns}_N$ is finite.  Let
\begin{equation}\label{eq;HN}
\mathcal{H}_N=  \bigcup_{i,j=0}^{2q-1} W^{i,\sta}_N \cap W^{j,\sta}_N \, ,  \end{equation}
which can be split into the set of transverse homoclinic points and the complementary, namely 
$\mathcal{H}_N= \mathcal{N}_N \cup \mathcal{T}_N$, with $\mathcal{T}_N$ the set of transverse homoclinic points. We write 
\[
\mathcal{N}_N = \{ Q_{n}\}_{n=1}^{n_0}, \qquad  \mathcal{T}_N = \{ Q_{n}\}_{n=n_0+1}^{n_1}.
\]
By Definition~\ref{definition adapted chart}, any homoclinic point has associated an adapted chart $(U_n,\phi_n)$ and we call (see Definition~\ref{definition adapted chart}) 
\[
\iota_n = \mathrm{order} (Q_n). 
\]
Take any collection of compact sets $\{K_n\}_{n=1}^{n_1}$ such that 
\begin{equation}\label{eq:Kn}
K_n \subset U_n, \qquad K_n \cap \mathcal{H}_N = \{Q_n\}.
\end{equation}

The following result studies how many homoclinic points belonging to a given compact set can (at most) arise after a sufficiently small perturbation.

\begin{lemma}\label{maximum of homoclinic points}
Fix a compact set $K\subset \mathbb{A}_\nu$. Let $\mathcal{W}_{K,N}$ be the set 
\[
\mathcal{W}_{K,N}:= K \cap \mathcal{H}_N
\]
and assume that $\mathrm{card} (\mathcal{W}_{K,N}) \leq 1$, that is, either $\mathcal{W}_{K,N}=\emptyset$ or $\mathcal{W}_{K,N} = \{Q\}$. 
    
There exists $\epsilon_0>0$ small enough such that for any $\lna\in \mathcal{C}^\omega_r(\mathbb{T},\mathbb{R})$, with $\|\lambda\|_r <\epsilon_0$, the billiard map $f_{\dgnalna}$ with $\dgnalna= \gna + \lna \cdot \nna$ satisfies the following. The set
\[
\widetilde{\mathcal{W}}_{K,N}:=K\cap \left (\bigcup_{i,j=0}^{2q-1} \widetilde{W}_N^{i,\sta} \cap \widetilde{W}_N^{j,\uns}\right ),
\]
where $\widetilde{W}_{N}^{i,\sta}, \widetilde{W}_N^{j,\uns}$ are the continuation of $W_N^{i,\sta}, W_N^{j,\uns}$, falls into one of the following three cases. 
\begin{enumerate}
\item If $\mathcal{W}_{K,N}=\emptyset$, then $\widetilde{\mathcal{W}}_{K,N} = \emptyset$. 
\item If $\mathcal{W}_{K,N}=\{Q\}$ with $Q$ a transverse homoclinic point,  then $\widetilde{\mathcal{W}}_{K,N} = \{\widetilde{Q}\}$, with  $\widetilde{Q}$ a transverse homoclinic point which is the continuation of $Q$. 
\item If $\mathcal{W}_{K,N}=\{Q\}$ with $Q$ a non transverse homoclinic point, then
\[
\mathrm{card} \big (\widetilde{\mathcal{W}}_{K,N} \big ) \leq  \mathrm{order}(Q).
\]
\end{enumerate}
\end{lemma}

\begin{proof} The first two items are straightforward by $\mathcal{C}^2$-openness of thansversality. 
Assume then that $Q$ is a non transverse homoclinic point. In the adapted chart, $(U,\phi)$, the connected component of the unstable manifold is described as the graph of a function $\g(x) = x^{\iota} + \mathcal{O}(x^{\iota +1})$, with $\iota = \mathrm{order}(Q)$, having a unique zero with multiplicity $\iota$. 
Let 
\[
M= \min_{x \in \partial U} |\g (x)|>0.
\]
Considering any small perturbation of the billiard table of the form $\gna+ \lna \cdot \nna$, with $\|\lambda\|_r < \epsilon$, the connected component of the unstable manifold $W^{\uns}(\widetilde{\mathcal{P}})$ belonging to $U$ is described by 
$\g^{\uns} (x;\epsilon) = \g (x) + \epsilon \tilde \g^{\uns} (x;\epsilon)$  with $\tilde \g^{\uns}$ an analytic function. 
Similarly, the connected component of the stable manifold $W^\sta(\widetilde{\mathcal{P}})$ belonging to $U$ is given by $\g^s(x;\epsilon)=\epsilon\tilde \g^s(x;\epsilon)$, with $\tilde \g^s$ an analytic function. 
Then, for all $x \in \partial U$, 
\[
|\g (x) -  (\g^{\uns}(x;\epsilon)- \g^\sta(x;\epsilon))| \leq C |\epsilon| <M \leq |\g (x)|
\]
if $\epsilon$ is small enough. 
Then, by Rouche's theorem, $\g$ and $\g^{\uns}- \g^{\sta}$ have the same number of zeros on $\overline{U}$, counting their multiplicity. 
Up to take a smaller $\epsilon$, we have that $\widetilde{\mathcal{W}}_{K,N} \cap \big (K\backslash U) = \emptyset$ and the proof is complete.  
\end{proof}

We are now under the conditions of finishing the proof of Theorem~\ref{main thm 2 bis}. We are going to perform an inductive argument considering different perturbations. To do so, let us first to introduce the following notation. For a given $\lambda_n^{(m)} \in \mathcal{C}^\omega_r(\mathbb{T},\mathbb{R})$ we set 
\[
\gna_n^{(m)} (s) = \gna (s) + \lambda_n^{(m)}(s)  \nna(s)
\]
and the corresponding billard map $f_{n}^{(m)}$. In addition, we also introduce 
\[
\mathcal{H}_{n}^{(m)} = \bigcup_{i,j=0}^{2q-1} \big (W_N^{i,\sta} \big )_n^{(m)} \cap \big (W_N^{i,\uns}\big )_n^{(m)}\, ,
\]
where $\big (W_N^{i,\sta} \big )_n^{(m)}, \big (W_N^{i,\uns}\big )_n^{(m)}$ are the continuation of 
$W_N^{i,\sta} ,  W_N^{i,\uns} $, the set of homoclinic points in the corresponding branches after the perturbation. We also consider the decomposition
\[
\mathcal{H}_{n}^{(m)} = \mathcal{N}_{n}^{(m)} \cup \mathcal{T}_{n}^{(m)} 
\]
of non transverse and transverse homoclinic points, and we call $\iota_n^{(m)} = \mathrm{card} \big (\mathcal{N}_{n}^{(m)} \big)$

Recall that for the unperturbed billiard map, in $\mathcal{H}_N$ (see \eqref{eq;HN}), we have $n_0$ non transverse homoclinic points and $n_1-n_0$ transverse homoclinic points. Consider now the decomposition of $\mathbb{A}_\nu$ given by
\[
\mathbb{A}_\nu = \left (\bigcup_{n=1}^{n_0} K_n \right ) \cup \widehat{\mathcal{K}}.
\]
where $K_n$ are the compact sets  given in \eqref{eq:Kn}
In other words, inside the compact set $\widehat{\mathcal{K}}$, if any, the homoclinic points are transverse. 

By Lemma~\ref{maximum of homoclinic points}, there exists $\epsilon_0$ such that
if $\|\lambda\|_r<\epsilon_0$, then the billiard map $f_{\widetilde{\gna}}$ with $\widetilde{\gna} = \gna + \lna \cdot \nna$,  satisfies that 
\begin{itemize}
\item in $\widehat{\mathcal{K}}$, the number of homoclinic points is the same as for $f$ and all of them are transverse.  
\item the number of homoclinic points on $K_n$ is at most $\iota_n$, where $\iota_n$ is the order of the unique homoclinic point (for the unperturbed dynamics) in $K_n$.
\end{itemize}

Take $\epsilon<\epsilon_0$ and assume that $n_0\geq 1$, otherwise we are done. By Remark~\ref{rmk cases non transverse homoclinic}, there exists $\ell_1$ such that $f^{\ell_1}(Q_1)$ satisfies the conditions of Corollary~\ref{corollary inverstible}. Let $\lambda^{(1)}_1$ be such that $\|\lambda_1^{(1)}\|_r<\epsilon$ and the corresponding perturbed billard $f_1^{(1)}$ has a transverse homoclinic point $Q_1^{(1)}$ in $K_1$. 
Since by Lemma~\ref{maximum of homoclinic points}, $f_{1}^{(1)}$ has at most $\iota_1$ homoclinic points, 
we deduce that 
\[
\mathrm{card} (\mathcal{N}_1^{(1)} \cap K_1) = \iota_1^{(1)}\leq \iota_1-1.  
\]
This is because, the perturbation $f_1^{(1)}$ makes one homoclinic point transverse by construction. 
In addition, since the perturbation $\lambda_1^{(1)}$ has size less than $\epsilon_0$, by Lemma~\ref{maximum of homoclinic points}, we have that
\[
\mathrm{card} (\mathcal{N}_1^{(1)} \cap K_n) = \iota_n^{(1)}\leq \iota_n, \qquad 2\leq n \leq n_0.  
\]

Since the condition of being transverse is an open condition, there exists $\epsilon^{(1)}_0 \leq \epsilon_0$ such that any perturbation of size less than $\epsilon^{(1)}_0$ keeps the continuation of $Q_1^{(1)}$ transverse. On the other hand, again using Remark~\ref{rmk cases non transverse homoclinic} and Corollary~\ref{corollary inverstible}, there exists a trigonometric polynomial $\lambda^{(2)}_1$, $\|\lambda^{(2)}_1\|_r<\epsilon^{(1)}_0 <\epsilon$ such that perturbed billiard map has $2$ homoclinic transverse homoclinic points at $K_1$. Since the number of homoclinic points at $K_1$ is at most $\iota_1$, after this second perturbation, 
\[
\mathrm{card} (\mathcal{N}_1^{(2)} \cap K_1) = \iota_1^{(1)}\leq \iota_1-2
\]
and, taking $\epsilon^{(1)}_0$ small enough, 
\[
\mathrm{card} (\mathcal{N}_1^{(2)} \cap K_n) = \iota_n^{(1)}\leq \iota_n, \qquad 2\leq n \leq n_0.  
\]
Notice that the perturbation $f_1^{(2)}$ is also a perturbation of $f$. 

We proceed inductively, $\iota^{(1)}_1$ times and we end up with a billiard map $f_{\widetilde{\gna}}$, having all the homoclinic points at $K_1 \cup \widehat{\mathcal{K}}$ transverse with associated analytic billiard table $\widetilde{\gna}$ that is $\epsilon$-close to $\gna$. 
 
Now we consider the decomposition of $\mathbb{A}_\nu$ given by 
\[
\mathbb{A}_\nu = \left (\bigcup_{n=2}^{n_0} K_n \right ) \cup \widehat{\mathcal{K}}^{(1)}, \qquad \widehat{\mathcal{K}}^{(1)} = \widehat{\mathcal{K}} \cup K_1.
\]
Analogously as before, we set $\epsilon_1\leq \epsilon_0$ such that after a perturbation of size $\epsilon_1$, the set $\widehat{\mathcal{K}}^{(1)}$ only contains transverse homoclinic points. As before we choose a perturbation $\lambda_2^{(1)}$ small enough such that the perturbed billiard map $f_2^{(1)}$ has a transverse homoclinic point at $K_2$. Since by Lemma~\ref{maximum of homoclinic points}, the number of homoclinic points at $K_2$ is $\iota_2$, 
\[
\mathrm{card} (\mathcal{N}_2^{(1)} \cap K_2) = \iota_2^{(1)}\leq \iota_2-1
\]
and 
\[
\mathrm{card} (\mathcal{N}_2^{(1)} \cap K_n) = \iota_n^{(2)}\leq \iota_n, \qquad 3\leq n \leq n_0.  
\]
This procedure needs a finite number of steps to obtain a billiard map satisfying that the set $
 \bigcup_{i,j=0}^{2q-1} \widetilde{W}_N^{i,\sta} \cap \widetilde{W}_N^{j,\uns} 
$ only contains transverse homoclinic points. 

\section{On Aubry-Mather sets of a twist map}\label{Aubry Mather}
The aim of this section is presenting some well-known results about Aubry-Mather sets of a twist map. Our main references are \cite{Bangert,Gole,ArnaudSalto}.

Denote by $\mathbb{A}$ the bounded annulus $\mathbb{T}\times\left [-\frac{\pi}{2}, \frac{\pi}{2} \right ]$. Endow it with the standard area form $\omega=ds\wedge dr$.
\begin{definition}
A positive conservative twist map is a $\mathcal{C}^1$ diffeomorphism $f\colon\mathbb{A}\to\mathbb{A}$ such that
\begin{enumerate}
\item $f$ is isotopic to the identity\footnote{In particular, $f$ preserves the orientation and sends each boundary into itself.};
\item $f$ preserves the area form $\omega$, i.e., $f^*\omega=\omega$;
\item for every $(x,y)\in \mathbb{A}$ it holds
\[
D(\pi_1\circ f)(x,y)\begin{pmatrix}
0 \\ 1
\end{pmatrix}>0\, ,
\]
where $\pi_1\colon \mathbb{A}\to \mathbb{T}$ is the projection on the first coordinate.
\end{enumerate}
\end{definition}

Let $F\colon \mathbb{R}\times\left [-\frac{\pi}{2}, \frac{\pi}{2} \right ]\to \mathbb{R}\times\left [-\frac{\pi}{2}, \frac{\pi}{2} \right ]$ be a lift of $f$. Denote by $p_1\colon\mathbb{R}\times\left [-\frac{\pi}{2}, \frac{\pi}{2} \right ]\to\mathbb{R}$, the projection on the first coordinate. Consider the domain 
\[
\mathcal{D}:=\{(X,X')\in \mathbb{R}^2 :\ \Pi_1\circ F(X,0)\leq X'\leq \Pi_1\circ F(X,1)\}\, .
\]
A conservative twist map can be described through its generating function
\[
H\colon  \mathcal{D}\to\mathbb{R}\, ,
\]
which is defined (up to adding an additive constant) by
\begin{equation}\label{generating function}
F(X,y)=(X',y') \quad \Longleftrightarrow \quad 
\begin{cases}
y = -\partial_1 H(X,X')\, , \\
y'=\partial_2 H(X,X')\, .
\end{cases}
\end{equation}
Following \cite[Section 7]{Bangert}, it is possible to extend the generating function to a $\mathcal{C}^2$ function defined on $\mathbb{R}^2$, that we still denote by $H\colon\mathbb{R}^2\to \mathbb{R}$, such that
\begin{enumerate}
\item for every $(X,X')\in\mathbb{R}^2$ it holds $H(X+1,X'+1)=H(X,X')$;
\item there exists $\delta>0$ such that for every $(X,X')\in\mathbb{R}^2$ it holds
\[
\partial_2\partial_1 H(X,X')\leq -\delta <0\, .
\]
\end{enumerate}

We denote by $\rho(F)$ the rotation interval of the lift $F$, i.e., the set of real values that can be obtained as
\[
\lim_{n\to+\infty}\dfrac{p_1\circ F^n(X,y)-X}{n}\, ,
\]
for some $(X,y)\in\mathbb{R}\times\left [-\frac{\pi}{2}, \frac{\pi}{2} \right ]$. Then, $\rho(f)$ will denote the rotation interval $\rho(F)\mod \mathbb{Z}$.
 
We are going to work in the space of bi-infinite sequences 
\[
\mathbb{R}^{\mathbb{Z}}=\{(X_i)_{i\in\mathbb{Z}} :\ X_i\in\mathbb{R}\quad \forall i\in \mathbb{Z}\}.
\]
A sequence $(X_i)_{i\in\mathbb{Z}}$ is a configuration for $F$ if and only if 
\begin{equation}
\partial_2H(X_{i-1},X_i)+\partial_1H(X_i,X_{i+1})=0\qquad\forall i\in\mathbb{Z}\, .
\end{equation}
Being a configuration corresponds then to be the projection on the first coordinate of a $F$-orbit of a point.

For any fixed $j,k\in \mathbb{Z}, j<k$, we consider the functional
\[
H_{j,k}\colon \mathbb{R}^{k-j+1}\to \mathbb{R}
\]
defined by
\[
(X_j,\dots, X_k)\mapsto H_{j,k}(X_j,\dots, X_k):=\sum_{i=j}^{k-1}H(X_{i},X_{i+1})\, .
\]
A segment $(X_j^*,\dots, X_k^*)$ is minimal if
\[
H_{j,k}(X_j^*,\dots, X_k^*)\leq H_{j,k}(X_j,\dots, X_k)
\]
for every $(X_j,\dots, X_k)\in\mathbb{R}^{k-j+1}$ such that $X_j=X_j^*,X_k=X_k^*$.
\begin{definition}
A bi-infinite sequence $(X_i)_{i\in\mathbb{Z}}$ is minimal if every finite segment of it is minimal. The set of minimal sequences is denoted by $\mathcal{M}(H)$.
\end{definition}
Observe that, since minimal segments are in particular critical ones, every minimal sequence is a configuration. Moreover, it is possible to show that a minimal sequence has a well-defined rotation number:
\[
\rho((X_i)_{i\in\mathbb{Z}})=\lim_{n\to+\infty}\dfrac{X_n-X_0}{n}\, .
\]

A sequence $(X_i)_{i\in\mathbb{Z}}$ is periodic of type $(p,q), q\neq 0$, if
\[
X_{i+q}=X_i+p\qquad\forall i\in\mathbb{Z}\, .
\]
Two sequences $(X_i)_{i\in\mathbb{Z}},(\tilde{X}_i)_{i\in\mathbb{Z}}$ are
$\omega$-asymptotic if $\lim_{i\to+\infty}\vert X_i-\tilde{X}_i\vert =0$, and $\alpha$-asymptotic if $\lim_{i\to-\infty}\vert X_i-\tilde{X}_i\vert =0$.

\begin{theorem}[\cite{Bangert}]
For every $\alpha\in\rho(F)$, there exists a sequence $(X_i)_{i\in\mathbb{Z}}$ in $\mathcal{M}(H)$ whose rotation number is equal to $\alpha$. The set of minimal sequences with rotation number $\alpha$ is the Aubry-Mather set of rotation number $\alpha$ and we denote it by $\mathcal{M}_\alpha(H)$.
\end{theorem}

In the sequel we will be interested into minimal configurations with rational rotation number. We describe then the corresponding Aubry-Mather set $\mathcal{M}_{\frac p q}(H)$, for $\frac p q\in\mathbb{Q}\cap\rho(F)$; it is the disjoint union of three sets of configurations:
\[
\mathcal{M}_{\frac p q}(H)=\mathcal{M}^{\mathrm{Per}}_{\frac p q}(H)\sqcup \mathcal{M}^+_{\frac p q}(H)\sqcup\mathcal{M}^-_{\frac p q}(H)\, .
\]
The set $\mathcal{M}^{\mathrm{Per}}_{\frac p q}(H)$ is the set of minimal configurations of rotation number $\frac p q$ which are periodic of type $(p,q)$. 

Two configurations $(X_i)_i,(\tilde X_i)_i\in\mathcal{M}^{\mathrm{Per}}_{\frac p q}(H)$ are \textit{neighboring} if there not exists a configuration $(\hat X_i)_i\in \mathcal{M}^{\mathrm{Per}}_{\frac p q}(H)$ such that $X_i<\hat X_i<\tilde X_i$ for every $i$. Given two neighboring configurations $(X_i)_i,(\tilde X_i)_i$ such that $X_i<\tilde X_i$ for every $i$, the set $\mathcal{M}^+_{\frac p q}(H)((X_i)_i,(\tilde X_i)_i)$ is the set
\[
\{(Z_i)_i\in \mathcal{M}_{\frac p q}(H) :\ (Z_i)_i, (X_i)_i \text{ are $\alpha$-asymptotic and }(Z_i)_i, (\tilde X_i)_i \text{ are $\omega$-asymtptotic}\}\, .
\]
The set $\mathcal{M}^-_{\frac p q}(H)((X_i)_i,(\tilde X_i)_i)$ is the set
\[
\{(Z_i)_i\in \mathcal{M}_{\frac p q}(H) :\ (Z_i)_i, (X_i)_i \text{ are $\omega$-asymptotic and }(Z_i)_i, (\tilde X_i)_i \text{ are $\alpha$-asymtptotic}\}\, .
\]
Thus, the set $\mathcal{M}^+_{\frac p q}(H)$ (resp. $\mathcal{M}^-_{\frac p q}(H)$) is the union of all the sets $\mathcal{M}^+_{\frac p q}(H)((X_i)_i,(\tilde X_i)_i)$ (resp. $\mathcal{M}^-_{\frac p q}(H)((X_i)_i,(\tilde X_i)_i)$) over all pairs of neighboring configurations $(X_i)_i,(\tilde X_i)_i\in\mathcal{M}^{\mathrm{Per}}_{\frac p q}(H)$.

\begin{theorem}[\cite{Bangert}]\label{thm_Bangert}
If $(X_i)_i, (\tilde X_i)_i$ are neighboring configurations in $\mathcal{M}_{\frac p q}^{\mathrm{Per}}(H)$, then the sets $\mathcal{M}^+_{\frac p q}(H)((X_i)_i,(\tilde X_i)_i)$ and $\mathcal{M}^-_{\frac p q}(H)((X_i)_i,(\tilde X_i)_i)$ are not empty.
\end{theorem}

Denote by $p_0\colon (X_i)_i\in \mathbb{R}^{\mathbb{Z}}\to X_0\in \mathbb{R}$ the projection on the $0$-th entry. It is possible to show that the function $p_0$ restricted to the set
\[
\mathcal{M}^{\mathrm{Per}}_{\frac p q}(H)\cup\mathcal{M}^+_{\frac p q}(H)\, ,
\]
respectively the set
\[
\mathcal{M}^{\mathrm{Per}}_{\frac p q}(H)\cup\mathcal{M}^-_{\frac p q}(H)\, ,
\]
is injective. This comes from the fact that the considered set is totally ordered with respect to the partial order
\[
(X_i)_i<(\tilde X_i)_i \quad\text{if and only if} \quad X_i<\tilde X_i \quad\forall i\, .
\]
Recall that a set is totally ordered if for any $(X_i)_i\neq (\tilde{X}_i)_i$ belonging to the set, either $(X_i)_i<(\tilde{X}_i)_i$ or $(X_i)_i>(\tilde{X}_i)_i$.

\begin{proposition}\label{prop:injectivity}
Let $f\colon \mathbb{A}\to\mathbb{A}$ be a twist map. Let $\frac p q \in \mathbb{Q}\cap\rho(f)$. 
\begin{itemize}
\item Either there is a not homotopically trivial circle made up of periodic points of rotation number $\frac p q$;
\item or there exists two periodic points $P,P'$ of rotation number $\frac p q$ and a point $Q$, which belongs to a heteroclinic orbit from $P$ to $P'$, such that the following holds. Denote by $\textbf{p}_1\colon \mathbb{A}\to \mathbb{T}$ the projection on the first coordinate and by $\mathcal{O}(Q)$ the orbit of $Q$. The map
\[
\textbf{p}_1\colon \mathcal{O}(Q)\to \mathbb{T}
\]
is injective.
\end{itemize}
\end{proposition}

Observe that Proposition~\ref{prop:homoclinicpq} is a direct consequence of the previous proposition and Proposition~\ref{prop:hyperbolicfixedpoint:section3}.

The result comes from the fact that we are considering sets that are totally ordered and from the following lemma.
\begin{lemma}\label{lem:injectivity}
Let $\mathcal{A}\subset \mathbb{R}^{\mathbb{Z}}$ be a periodic, totally ordered set of minimal configurations. Let $\Upsilon \colon\mathcal{A}\subset \mathbb{R}^{\mathbb{Z}}\to \mathbb{R}^2$ be the function that associates to each minimal configuration $(X_i)_i$ the point $(X,Y)\in\mathbb{R}^2$ such that $p_1\circ F^n(X,Y)=X_n$   for all $n\in\mathbb{Z}$. Let $\pi\colon \mathbb{R}^2\to \mathbb{T}\times\mathbb{R}$ be the universal cover of the unbounded annulus. Then $\pi\circ\Upsilon(\mathcal{A})=:A$ is a graph over its projection on $\mathbb{T}$.
\end{lemma}
\begin{proof}
Assume by contradiction that there exists $(x,Y),(x,\tilde Y)\in A$ such that $Y<\tilde Y$. In particular, there exists $k\in \mathbb{Z}$ such that $X=\tilde X+k$ where $\pi_0(X)=\pi_0(\tilde X)= x$ and $\pi_0\colon\mathbb{R}\to\mathbb{T}$ is a universal cover of the torus. Without loss of generality, because of the periodicity of the set, we can assume that $k=0$. Let $(X_i)_i\in \mathcal{A}$, respectively $(\tilde X_i)_i\in \mathcal{A}$, be the configuration corresponding to the point $(X,Y)\in \Upsilon({\mathcal{A}})$, respectively $(X, \tilde Y)\in \Upsilon({\mathcal{A}})$, such that $X_0=X$, respectively $\tilde X_0=X$. By the twist condition and since $Y<\tilde Y$, we can deduce that $X_1<\tilde X_1$, that is, the configurations $(X_i)_i$ and $(\tilde X_i)_i$ cross. This contradicts the fact that the set $\mathcal{A}$ is totally ordered and we conclude the proof.
\end{proof}

\section{Billiard perturbations. Proof of Propositions~\ref{cor:expr:fepsilon} and~\ref{prop Frechet Gamma nu}}
\label{billiard perturbations}
Along this section we will study the map $f_{\varepsilon}$ associated to the perturbed billiard table $\defga$, as in~\eqref{eq:perturbedmap:section3}. 
To this end, we need some extra properties of the generating function $\tau (s,s')$ of the unperturbed map $f_{\ga}$ (see~\eqref{eq:unperturbedmap arclength}) defined in~\eqref{generating function arclength}.
When restricted to the set $\Delta_\mu$ defined in~\eqref{definition Delta mu}, besides equations~\eqref{first gen fct:section3} the generating function satisfies  (see e.g.~\cite[Lemma 2.1]{KaloshinZhang}),
\begin{equation}\label{der seconde gen}
\begin{aligned} 
&\partial_{11} \tau(s,s')=\mathcal{K}(s)\cos \varphi+\dfrac{\cos^2 \varphi}{\tau(s,s')},\\
&\partial_{12} \tau(s,s')=\partial_{21} \tau(s,s')=\dfrac{\cos \varphi\cos\varphi'}{\tau(s,s')},\\
&\partial_{22} \tau(s,s')=\mathcal{K}(s')\cos \varphi'+\dfrac{\cos^2 \varphi'}{\tau(s,s')},
\end{aligned}
\end{equation}
with $\varphi=\varphi(s,s')$ and $\varphi'=\varphi'(s,s')$ defined by~\eqref{first gen fct:section3} and $\mathcal{K}(s)$ is the curvature of $\ga(\mathbb{T})$.  

It is convenient to introduce the function 
\begin{equation}\label{def:tau_hat}
\hat \tau\colon \mathbb{A}\to \R_+,\qquad  (s,\varphi)\mapsto\tau(s,s'),  
\end{equation}
where $s'=s'(s,\varphi)$ is defined by~\eqref{eq:unperturbedmap arclength}, i.e., $(s',\varphi')=f_{\ga} (s,\varphi)$. That is, given $(s,s') \in \mathbb{R} \times \mathbb{R}$, we let $\varphi(s,s')$ be defined by~\eqref{first gen fct:section3} so that $\hat \tau(s,\varphi(s,s'))=\tau(s,s')$.  
 
\begin{lemma}
For any $(s,\varphi)$, letting $(s',\varphi')=f_{\ga}(s,\varphi)$, we have that
\begin{align*}
\partial_1 \hat \tau(s,\varphi)&=-\sin \varphi -(\hat \tau(s,\varphi)\mathcal{K}(s)+\cos \varphi) \tan \varphi',\\
\partial_2 \hat \tau(s,\varphi)&=-\hat \tau(s,\varphi)\tan \varphi' .
\end{align*}
\end{lemma}
\begin{proof}
By~\eqref{first gen fct:section3}, \eqref{extension generating function} and~\eqref{der seconde gen}, for $\varphi=\varphi(s,s')$, we obtain:
\begin{equation}\label{eq partial varphi 1}
\partial_1 \varphi(s,s')=-\frac{\partial_{11} \tau(s,s')}{\cos \varphi(s,s')}=-\mathcal{K}(s)-\frac{\cos \varphi}{\tau(s,s')},
\end{equation}
and  
\begin{equation}\label{eq partial varphi 2}
\partial_2 \varphi(s,s')= -\frac{\partial_{12} \tau(s,s')}{\cos \varphi(s,s')} =-\frac{\cos \varphi'}{\tau(s,s')},
\end{equation}
with $\varphi'=\varphi'(s,s')$ as in~\eqref{first gen fct:section3}. Since $\hat \tau(s,\varphi)=\tau(s,s')$ for $\varphi=\varphi(s,s')$, we deduce that
\begin{equation*}
\partial_2\hat \tau(s,\varphi)=\frac{\partial_2\tau(s,s')}{\partial_2 \varphi(s,s')}=-\hat \tau(s,\varphi)\frac{\sin \varphi'}{\cos \varphi'}.
\end{equation*}
Similarly, from \eqref{eq partial varphi 1}, we obtain:
\begin{equation*}
\begin{aligned}
\partial_1\hat \tau(s,\varphi) & =\partial_1 \tau(s,s')-\partial_1\varphi(s,s')\cdot \partial_2 \hat \tau(s,\varphi) \\ 
& =-\sin \varphi -\hat \tau(s,\varphi)\left(\mathcal{K}(s)+\frac{\cos \varphi}{\hat\tau(s,\varphi)}\right)\frac{\sin \varphi'}{\cos \varphi'}.
\end{aligned}
\end{equation*}
\end{proof}

Let us also recall the expression of the differential of the (unperturbed) billiard map $f_{\ga}\colon(s,\varphi)\mapsto (s',\varphi')$ (see e.g. \cite[p. 35] {CheMar_book}) that will be needed in the following sections:

\begin{equation} \label{formula:Dfsphi}
Df_{\ga}(s,\varphi)=-\frac{1}{\cos \varphi'}\begin{bmatrix}
\hat \tau \,   \mathcal{K}(s)+\cos \varphi & \hat \tau   \\
\hat \tau  \,  \mathcal{K}(s)\mathcal{K}(s')+\mathcal{K}(s)\cos \varphi'+\mathcal{K}(s') \cos \varphi & \hat \tau  \, \mathcal{K}(s')+\cos \varphi'
\end{bmatrix}
\end{equation}
where $\hat{\tau}= \hat{\tau}(s,\varphi)$. 

\subsection{Perturbations of the billiard table. Proof of Proposition \ref{cor:expr:fepsilon}} \label{billiard perturbationTable} 
The goal in this section is to compute the first order in $\varepsilon$ of the perturbed billiard map $f_\varepsilon=f_{\defga}$, with $\defga$ given in~\eqref{deformations}. We recall that the perturbed boundary $\partial\Omega_\varepsilon$, is  not parametrized in arc-length. 

\begin{lemma}\label{lem:normtangent}
For each $s \in \mathbb{T}$, it holds
\begin{equation}\label{norme gamma prime eps}
\begin{aligned}
\|\derdefga(s)\| = & \sqrt{(1-\varepsilon \lac(s)\mathcal{K}(s))^2+\varepsilon^2 (\derlac(s))^2}=1-\varepsilon
\lac(s)\mathcal{K}(s)+\varepsilon^2 R_1(s;\varepsilon), \\
\mathcal{K}_\varepsilon(s)  =& \Big [1- \varepsilon  \lac(s) \mathcal{K}(s)) (\mathcal{K}(s) + \varepsilon (\dertwolac(s) -  \lac(s) (\mathcal{K}(s))^2)) \\ 
& + \varepsilon^2 \derlac(s) (2 \derlac(s) \mathcal{K}(s) +  \lac(s) \der{\mathcal{K}}(s)) 
\Big ] \cdot \Big[(1- \varepsilon  \lac(s) \mathcal{K}(s))^2 + \varepsilon^2 (\derlac(s))^2
\Big ]^{-\frac{3}{2}} 
\\
=& \mathcal{K}(s)+\varepsilon \big (\dertwolac(s)+  \lac(s) (\mathcal{K}(s))^2 \big )+ \varepsilon^2 R_2(s;\varepsilon)
\end{aligned}
\end{equation}
where $\mathcal{K}_\varepsilon(s)\leq 0$ denotes the curvature of $\partial\Omega_\varepsilon $ at the point $\defga(s)$.

Moreover there exist $\varepsilon_1$ and a constant $M_1$ depending only on $\|\ga\|_{\mathcal{C}^k}$, such that for all $\varepsilon \in [0, \varepsilon_1]$, $\mathcal{K}_\varepsilon(s)<0$, 
\[
\|R_1 \|_{\mathcal{C}^{k-2}} \leq M_1 \|\lac\|_{\mathcal{C}^{k-1}}^2, \qquad \|R_2\|_{\mathcal{C}^{k-3}} \leq M_1 \|\lac \|_{\mathcal{C}^{k-1}}^2.
\] 	
\end{lemma}
\begin{proof}
Fix $s \in \mathbb{T}$. As $\ga$ is an arclength anticlockwise parametrization of $\partial \Omega$, we have $\narc(s)=R_{-\frac \pi 2}\derga(s)$, where $R_{-\frac \pi 2}$ is the rotation of angle $-\frac \pi 2$, and $\dertwoga(s)=\mathcal{K}(s)\narc(s)$. Thus,
\begin{equation}\label{derivative of n arclength}
\frac{d}{ds}\narc (s)=R_{-\frac \pi 2}\dertwoga(s)=R_{-\frac \pi 2}\mathcal{K}(s)\narc (s)=-\mathcal{K}(s)\derga(s),
\end{equation}
and then,
\[
\derdefga(s)=(1-\varepsilon  \lac(s)\mathcal{K}(s))\derga(s)+\varepsilon \derlac(s) \narc(s).
\]
As a consequence 
\[
\| \derdefga (s) \|^2 = (1- \varepsilon \lac(s) \mathcal{K}(s))^2 + \varepsilon^2 (\derlac(s))^2 
\]
and we easily obtain the first identity in~\eqref{norme gamma prime eps}. On the other hand, 
\begin{align*}
\dertwodefga(s)= & (1-\varepsilon \lac(s)\mathcal{K}(s))\mathcal{K}(s)\narc(s)-\varepsilon \frac{d}{ds} [\lac(s)\mathcal{K}(s)]\derga(s)+\varepsilon \dertwolac(s) \narc(s) \\ &-\varepsilon \derlac(s)\mathcal{K}(s)\derga(s)\\
= &\big [\mathcal{K}(s)+\varepsilon(\dertwolac(s)-  \lac(s) (\mathcal{K}(s))^2) \big ]\narc(s)-\varepsilon \big [2\derlac(s)\mathcal{K}(s)+ \lac(s)\der{\mathcal{K}}(s) \big ]
\derga(s)
\end{align*}
and we notice that the terms in the expansion of $\dertwodefga$ are $\mathcal{C}^{k-3}$ functions with $\mathcal{C}^{k-3}$ norm of the same order as $\| \lac\|_{\mathcal{C}^{k-1}}$.
 
Using that $\det(\narc(s),\dot \ga(s))=1$, by standard formulas expressing curvature, we thus obtain:
\begin{align}\label{curvature arc length}
\mathcal{K}_\varepsilon(s)=&-\frac{\det (\derdefga(s), \dertwodefga(s))}{\|\derdefga(s)\|^3} \\
=&\Big [(1- \varepsilon  \lac(s) \mathcal{K}(s)) (\mathcal{K}(s) + \varepsilon (\dertwolac(s) -  \lac(s) (\mathcal{K}(s))^2)) \notag \\  &+ \varepsilon^2 \derlac(s) (2 \derlac(s) \mathcal{K}(s) +  \lac(s) \der{\mathcal{K}}(s)) \Big ] \cdot \Big[
(1- \varepsilon  \lac(s) \mathcal{K}(s))^2 + \varepsilon^2 (\derlac(s))^2\Big ]^{-\frac{3}{2}} \notag \\
=&(\mathcal{K}(s)-\varepsilon  \lac(s)(\mathcal{K}(s))^2+\varepsilon(\dertwolac(s)-  \lac(s)(\mathcal{K}(s))^2))(1+3\varepsilon \lac(s)\mathcal{K}(s))+\varepsilon^2 \hat R_2(s;\varepsilon) \notag \\
=&\mathcal{K}(s)+\varepsilon(\dertwolac(s)+ \lac(s)
(\mathcal{K}(s))^2)+\varepsilon^2 R_2(s;\varepsilon) \notag
\end{align}
and the bound for $\|R_2\|_{\mathcal{C}^{k-3}}$ follows easily. 
\end{proof}

In the following, we let 
\begin{equation}\label{eq:taueps}
\tau_\varepsilon\colon (s,s')\mapsto \|\defga(s)-\defga(s')\|.    
\end{equation}
We also introduce, for $0<\mu <\frac{1}{2}$, $0<\nu \ll 1$ and $\ell \leq k-1$ 
\begin{equation}\label{defnormDeltamu}
\begin{aligned}
\|R\|_{\norm{\ell}{\Delta_\mu}} & := \max_{j=0,\cdots , \ell} \left \{ \max_{(s,s',\varepsilon)\in \Delta_\mu \times [0,\epsilon_0]} \|D^j R(s,s';\varepsilon)\|\right \}\\
\|\hat R\|_{\norm{\ell}{\mathbb{A}_\nu}}  & := \max_{j=0,\cdots , \ell} \left \{ \max_{(s,\varphi,\varepsilon)\in \mathbb{A}_\nu \times [0,\epsilon_0]} \|D^j \hat R(s,\varphi;\varepsilon)\|\right \}
\end{aligned}
\end{equation}
where $R, \hat{R}$ are $\mathcal{C}^{\ell}$ in its arguments, $D$ means in the former the differential with respect to $(s,s',\varepsilon)$ and in the latter with respect to $(s,\varphi,\varepsilon)$. Recall that the sets $\Delta_\mu$ and $\mathbb{A}_\nu$ were introduced in Lemma~\ref{away from boundary 1}. If there is no danger of confusion we omit the sets $\Delta_\mu, \mathbb{A}_\nu$ in the notation of the norms in~\eqref{defnormDeltamu}. 

\begin{lemma}\label{lemma pertr lenght func}
For any $(s,s') \in \T \times \T$ we have 
\[
\tau^2 _\varepsilon(s,s') = \tau^2 (s,s')+2\varepsilon \tau(s,s') (\lac(s) \cos \varphi+\lac(s')\cos \varphi') +
\varepsilon^2 \| \lac(s)\narc(s)-\lac(s')\narc(s')\|^2
\]
where $\varphi=\varphi(s,s')$ and $\varphi'=\varphi'(s,s')$ are defined by $\sin \varphi=-\partial_1 \tau(s,s')$, and $\sin \varphi'=\partial_2\tau(s,s')$.

Moreover, for any $0<\mu <\frac{1}{2}$ there exist $\varepsilon_\mu$ and a constant $M_{\mu}$, depending also on the $\mathcal{C}^k$ norm of $\ga$, such that if $s, s' \in \Delta_\mu$ and $\varepsilon \in [-\varepsilon_\mu, \varepsilon_\mu]$, then 
\begin{equation}\label{exp tau epss}
\begin{aligned}
\tau_\varepsilon(s,s')=&\tau(s,s')+\varepsilon (\lac(s)\cos \varphi+\lac(s')\cos \varphi')\\
&+\frac{\varepsilon^2}{2\tau(s,s')}\big(\|\lac(s)\narc(s)-\lac(s') \narc(s')\|^2-(\lac(s)\cos \varphi+\lac(s')\cos \varphi')^2\big) \\ & +\varepsilon^3 R_3(s,s';\varepsilon), 
\end{aligned}
\end{equation}
where $ \|R_3\|_{\norm{k}{\Delta_\mu}} \leq M_{\mu} \|\lac \|_{\mathcal{C}^k}^3$.
\end{lemma}
\begin{proof}
In the following computation, we denote by $v(s,s')$ the unit vector $v(s,s'):=\frac{\ga (s)-\ga(s')}{\|\ga(s)-\ga(s')\|}$. We have:
\begin{align*}
\tau_{\varepsilon}^2 (s,s')   -  & \tau^2(s,s') =  \langle \defga(s) - \defga(s') , \defga(s) - \defga(s')\rangle - 
\langle \ga(s) - \ga(s') , \ga(s) - \ga(s')\rangle 
\\ 
 &=  2\varepsilon \langle  \lac(s)\narc(s) -  \lac(s')\narc(s'), \ga(s) - \ga(s')\rangle 
+\varepsilon^2 \| \lac(s) \narc(s) -  \lac(s')\narc(s')\|^2 \\
&= 2\varepsilon  \lac(s) \|\ga(s) - \ga(s')\| \langle n *(s) , v(s,s')\rangle   - 
2\varepsilon  \lac(s') \|\ga(s) - \ga(s')\| \langle \narc(s') , v(s,s')\rangle 
 \\ & \hspace{0.3cm}+ \varepsilon^2 \| \lac(s) \narc(s) -  \lac(s')\narc(s')\|^2 
\\ &= 
2\varepsilon \tau(s,s') ( \lac(s) \cos \varphi+ \lac(s')\cos \varphi') +
\varepsilon^2 \| \lac(s)\narc(s)- \lac(s')\narc(s')\|^2.
\end{align*}
Therefore, for a fixed $\mu \in \left (0,\frac{1}{2}\right )$, by Lemma~\ref{away from boundary 2}, $\tau(s,s')\geq c_\mu$ for some constant $c_\mu>0$ and then
\[
\frac{\tau_{\varepsilon} (s,s')}{\tau (s,s')}= \left [ 1 +\frac{2\varepsilon}{\tau(s,s')} ( \lac(s) \cos \varphi+ \lac(s')\cos \varphi') +
\frac{\varepsilon^2}{\tau^2(s,s')}\| \lac(s)\narc(s)- \lac(s')\narc(s')\|^2 \right ]^{1/2}
\]
which gives the result using the Taylor expansion of the square root. \qedhere
\end{proof}

Given $(s,s')\in \mathbb{T} \times \mathbb{T}$, we denote by $\varphi_\varepsilon=\varphi_\varepsilon(s,s')$, resp. $-\varphi_\varepsilon'=-\varphi_\varepsilon'(s,s')$, the angle between the inward normal to $\Omega_\varepsilon$ at $\defga(s)$, resp. $\defga(s')$ and the vector $\defga(s')-\defga(s)$, resp. $\defga(s)-\defga(s')$.

\begin{corollary}\label{coro angl} 
Let $0<\mu <\frac{1}{2}$. There exist $\varepsilon_\mu>0$ and a constant $M_{\mu}$ also depending on  $\|\ga\|_{\mathcal{C}^k}$, such that for any $(s,s')\in \Delta_\mu$ and $\varepsilon \in [-\varepsilon_\mu , \varepsilon_\mu]$, we have
\begin{align*}
\varphi_\varepsilon(s,s')&=\varphi-\varepsilon \left (\derlac(s)+ \lac(s)\frac{\sin \varphi}{\tau}-\lac(s')\frac{\sin\varphi'}{\tau}\right) +\varepsilon^2 R_4(s,s';\varepsilon),\\
\varphi_\varepsilon'(s,s')&=\varphi'+\varepsilon \left (\derlac(s')+\lac(s)\frac{\sin\varphi}{\tau}- \lac(s')\frac{\sin \varphi'}{\tau}\right ) +\varepsilon^2 R_5(s,s';\varepsilon),
\end{align*}
where $\varphi=\varphi(s,s')$, $\varphi'=\varphi'(s,s')$ are defined by~\eqref{first gen fct:section3},  $\tau=\tau(s,s')$ and the reminders satisfy $\|R_4\|_{\norm{k-2}{\Delta_\mu}}, \|R_5\|_{\norm{k-2}{\Delta_\mu}}\leq M_{\mu} \|\lac\|^2_{\mathcal{C}^k}$.
\end{corollary}
\begin{proof}
It holds (see \eqref{eq:taueps})
\begin{equation}\label{gennnnbi}
\partial_1 \tau_\varepsilon(s,s')=\left\langle \derdefga (s),\frac{\defga(s)-\defga(s')}{\|\defga(s)-\defga(s')\|}\right\rangle=-\|\derdefga(s)\|\sin \varphi_\varepsilon,
\end{equation}
using  formula \eqref{norme gamma prime eps} for $\|\derdefga(s)\|$ we obtain
\[
\sin\varphi_\varepsilon= (1+\varepsilon \lac(s)\mathcal{K}(s) + \varepsilon^2 \widetilde{R}_1(s,s';\varepsilon))\cdot (-\partial_1 \tau_\varepsilon(s,s'))
\]
with $\|\widetilde{R}_1 \|_{\mathcal{C}^{k-2}  } \leq M \| \lac\|_{\mathcal{C}^{k-1}}^2$ for some constant $M>0$. 

By Lemma~\ref{away from boundary 2}, $\tau(s,s')\geq c_\mu>0$ if $(s,s') \in \Delta_\mu$. Then, differentiating~\eqref{exp tau epss} with respect to $s$ and using that $\tau(s,s')\geq c_\mu>0$,
\[
-\partial_1 \tau_{\varepsilon} (s,s') = \sin \varphi - \varepsilon (\derlac(s) \cos \varphi +  \lac(s) \partial_1 \cos \varphi(s,s')   +  \lac(s') \partial_1 \cos \varphi'(s,s') )+ \varepsilon^2 \widetilde{R}_3(s,s'; \varepsilon)
\]
where $\|\widetilde{R}_3\|_{\mathcal{C}^k} \leq M_{c_\mu} \| \lac\|^2_{\mathcal{C}^k}$, with $M_{c_\mu}$ a constant depending on $c_\mu$. Therefore, 
\begin{align*}
\sin\varphi_\varepsilon=& 
\sin \varphi+\varepsilon ( \lac(s)\mathcal{K}(s)\sin \varphi-\derlac(s)\cos \varphi- \lac(s)\partial_1 \cos \varphi- \lac(s')\partial_1\cos \varphi') \\ & +\varepsilon^2 \widetilde{R}_4 (s,s';\varepsilon),
\end{align*}
with $\varphi=\varphi(s,s')$ and $\varphi'=\varphi'(s,s')$ as in~\eqref{first gen fct:section3}, and $\|\widetilde{R}_4\|_{\mathcal{C}^{k-2}}\| \leq {M}_{\mu} \| \lac\|^2_{\mathcal{C}^k}$.
    
By \eqref{first gen fct:section3}, \eqref{der seconde gen}-\eqref{eq partial varphi 1},  we have: 
\begin{equation}\label{partial cos phi}
\begin{aligned}
\partial_1 \cos \varphi(s,s')&=-\sin\varphi(s,s') \partial_1 \varphi(s,s')=\mathcal{K}(s)\sin \varphi+\frac{\sin \varphi\cos \varphi}{\tau},\\
\partial_1 \cos \varphi'(s,s')&=\partial_1 \sqrt{1-(\sin \varphi')^2}=-\frac{\sin \varphi'}{\cos \varphi'} \partial_{12} \tau(s,s')=-\frac{\sin\varphi'\cos \varphi}{\tau}.  
\end{aligned}
\end{equation}
Plugging these expressions into the previous expansion of $\sin \varphi_\varepsilon$, we obtain:
\begin{align*}
\sin\varphi_\varepsilon=\sin \varphi-\varepsilon \left (\derlac(s)+\lac(s)\frac{\sin \varphi}{\tau}-\lac(s')\frac{\sin\varphi'}{\tau}\right)\cos \varphi+\varepsilon^2 \widetilde{R}_4(s,s';\varepsilon),
\end{align*}
from which we can easily obtain the expansion of $\varphi_\varepsilon
$. To obtain the expansion of $\varphi_\varepsilon'$, we note that  $\varphi'(s,s')=-\varphi(s',s)$, $\varphi(s,s')=-\varphi'(s',s)$, and  $\varphi_\varepsilon'(s,s')=-\varphi_\varepsilon(s',s)$.
\end{proof}

Let us denote by $f_\ga   \colon (s,\varphi)\mapsto (s',\varphi')$ the initial billiard map, and by 
\begin{equation}\label{eq:perturbedmap}
f_\varepsilon\colon (s,\varphi)\mapsto (s_\varepsilon',\varphi_\varepsilon')
\end{equation}
the billiard map for the perturbed domain $\Omega_\varepsilon$, defined in~\eqref{deformations}, with $s_\varepsilon'=s_\varepsilon'(s,\varphi)$ and $\varphi_\varepsilon'=\varphi_\varepsilon'(s,\varphi)$. Note that we slightly abuse notation by keeping the same notation as for the function $\varphi_\varepsilon'(s,s')$ previously defined, so that 
\begin{equation}\label{link_notation}
\varphi_\varepsilon'(s,\varphi)=\varphi_\varepsilon'(s,s_\varepsilon'). 
\end{equation} 

Finally we finish the proof of Proposition \ref{cor:expr:fepsilon}.
\begin{proof}[End of the proof of Proposition~\ref{cor:expr:fepsilon}]
Along this proof we will denote by $M>0$ a generic constant depending on $\nu$ and $\|\ga\|_{\mathcal{C}^{k}}$ that can (and usually does) change its value. 

Fix $(s,\varphi)\in \mathbb{A}_\nu$. We look for $s_\varepsilon'=s_\varepsilon'(s,\varphi)$ satisfying $\varphi_\varepsilon(s,s_\varepsilon')=\varphi$ with $\varphi_\varepsilon(s,s')$ defined by Corollary~\ref{coro angl} and we recall that $s'=s'(s,\varphi)$ is the $s-$ component of the original billiard map $f_\ga $. In other words, $s_\varepsilon$ is such that
\begin{equation}\label{et:varphi}
\varphi_\varepsilon(s,s_\varepsilon'(s,\varphi))= \varphi(s,s'(s,\varphi))
\end{equation}
where $\varphi=\varphi(s,s')$ defined through $\partial_1 \tau(s,s')= -\sin \varphi$. We write 
\begin{equation}\label{def:sepsilon}
s_\varepsilon'(s,\varphi)=s'(s,\varphi) + \varepsilon \varpi(s,\varphi;\varepsilon)
\end{equation}
with $\varpi\colon\mathbb{A}_\nu \times (-\varepsilon_\nu, \varepsilon_\nu)\to\mathbb{R} $. Then~\eqref{et:varphi} is equivalent to
\begin{equation}\label{et:varphi_2}
0=\varphi_\varepsilon (s,s_\varepsilon')- \varphi(s,s') = \varphi_\varepsilon(s,s_\varepsilon') - \varphi(s,s_\varepsilon') + \varphi(s, s'_\varepsilon) - \varphi(s,s').
\end{equation}
On the one hand by Corollary~\ref{coro angl}, letting
\begin{equation}\label{def:phi_1_cor}
\varphi^1 (s,t'):= -\derlac(s) + \lac(s) \frac{\sin \varphi}{\tau(s,t')} -  \lac(t') \frac{\sin \varphi'(s,t')}{\tau(s,t')}, \qquad (s,t') \in \mathbb{R}\times \mathbb{R},
\end{equation}
with $\varphi'(s,t')$ defined by~\eqref{first gen fct:section3}, we have that 
\[
\varphi_\varepsilon (s,s_\varepsilon')- \varphi(s,s'_\varepsilon) =  \varepsilon \varphi^1 (s,s'_\varepsilon) + 
\varepsilon^2 R_4(s,s_\varepsilon';\varepsilon).
\]
On the other hand
\[
\varphi(s,s_\varepsilon') - \varphi(s,s')= \varepsilon \varpi(s,\varphi;\varepsilon) \int_0^1 \partial_2 \varphi(s,s'(s,\varphi)+ \mu \varepsilon \varpi(s,\varphi;\varepsilon) d\mu.
\]
Therefore, \eqref{et:varphi_2} can be rewritten as
\begin{equation}\label{et:varphi_3}
\varphi^1(s,s'+\varepsilon \varpi) + \varepsilon R_4(s,s'+\varepsilon \varpi;\varepsilon) + \varpi \int_{0}^1 \partial_2 \varphi(s,s'+ \mu \varepsilon \varpi) d\mu =0.
\end{equation}
with, again, $s'=s'(s,\varphi)$ and $\varpi=\varpi(s,\varphi;\varepsilon)$. 

We recall now that by~\eqref{eq partial varphi 2} $\partial_2 \varphi(s,s') = -\tau^{-1} \cos \varphi'$ with $\varphi'=\varphi'(s,\varphi)$ the $\varphi-$ component of the billiard map $f_\ga $. Then, by Lemmas~\ref{away from boundary 1} and~\ref{away from boundary 2}, there exists a constant $c$ such that
\begin{equation}\label{bound_implicit}
\partial_2 \varphi(s,s'(s,\varphi)) = -\frac{\cos \varphi'(s,\varphi)}{\tau(s,s'(s,\varphi))} 
=-\frac{\cos \varphi'(s,\varphi)}{\hat{\tau}(s,\varphi)}\leq -c <0
\end{equation}
and, as a consequence, 
\begin{equation}
\begin{aligned}\label{def:varpi0}
\varpi_0(s,\varphi)& = -\frac{\varphi^1(s,s')}{\partial_2 \varphi(s,s')} =  -\left (\derlac(s) +  \lac(s) \frac{\sin \varphi}{\tau} -  \lac(s') \frac{\sin \varphi'}{\tau} \right ) \frac{\hat{\tau}} {\cos \varphi'} \\ & = -\frac{1}{\cos \varphi'} (\derlac(s) \hat{\tau} +  \lac(s) \sin \varphi -  \lac(s') \sin \varphi')
\end{aligned}
\end{equation}
is the unique solution of~\eqref{et:varphi_3} for $\varepsilon=0$. We observe that $\varpi_0 \in \mathcal{C}^{k-1} (\mathbb{A}_\nu, \mathbb{R})$  with $\|\varpi_0\|_{\norm{k-1}{\mathbb{A}_\nu}} \leq M \| \lac\|_{\mathcal{C}^k}$. 
Writing $\varpi = \varpi_0 + \varpi_1 $, equation~\eqref{et:varphi_3} can be rewording as the fixed point equation
\begin{align}\label{et:varphi_4}
\varpi_1   =   \mathcal{F}[ \hat \varpi_1]   := &\frac{1}{  \partial_2 \varphi(s,s' )} \Big  ( 
\varphi^1 (s,s' + \varepsilon \varpi_0 + \varepsilon  \varpi_1   ) - \varphi^1(s,s')+ \varepsilon R_4(s,s'+\varepsilon \varpi_0+ \varepsilon \varpi_1;\varepsilon)  \notag \\ &  + (\varpi_0 +  \varpi_1) \int_{0}^1 \big (\partial_2 \varphi(s,s'+\mu \varepsilon (\varpi_0+\varpi_1))- \partial_2 \varphi(s,s') \big ) d\mu
\Big ).
\end{align}
Let $\mu,c,\nu'$ be as in Lemma~\ref{away from boundary 1} so that $(s,s')\in \Delta_{\mu}$, $\tau(s,s')\geq c$ and $\varphi'\in \left [-\frac{\pi}{2} - \nu' , \frac{\pi}{2} - \nu'\right ]$ when $(s,\varphi) \in \mathbb{A}_\nu$. Then, recalling that by Corollary~\ref{coro angl}, $\|R_4\|_{\norm{k-2}{\Delta_\mu}}\leq M \| \lac\|^2_{\mathcal{C}^k}$ and that, by definition~\eqref{def:phi_1_cor}, $\|\varphi^1\|_{\norm{k-1}{\Delta_\mu}} \leq M \| \lac\|_{\mathcal{C}^k}$, an straightforward application of the fixed point theorem proves that there exists $\varepsilon_0>0$ small enough such that~\eqref{et:varphi_4} has a unique solution $\varpi_1$ which is $\mathcal{C}^{k-2}$ on $\mathbb{A}_\nu $ satisfying that for any $\varepsilon\in (-\varepsilon_0,\varepsilon_0)$, $\|\varpi_1 \|_{\norm{k-3}{\mathbb{A}_\nu}} \leq  M \varepsilon  \| \lac\|_{\mathcal{C}^k}$. 
From~\eqref{et:varphi_3}, we obtain that  
\begin{equation}\label{expresion for svarepsilon}
s_{\varepsilon}' (s,\varphi)= s'(s,\varphi) + \varepsilon \varpi_0(s,\varphi) + \varepsilon^2 R_6(s,\varphi;\varepsilon)
\end{equation}
with $\varpi_0$ defined in~\eqref{def:varpi0} and $\|R_6\|_{\norm{k-3}{\mathbb{A}_\nu}} \leq M \| \lac\|_{\mathcal{C}^k}$. 

Now, in order to obtain the expansion of $\varphi_\varepsilon'$, as observed in~\eqref{link_notation}, we use the fact  that $\varphi_\varepsilon'=\varphi_\varepsilon'(s,\varphi)=\varphi_\varepsilon'(s,s_\varepsilon')$; by Corollary \ref{coro angl}, we obtain
\begin{equation*}
\varphi_\varepsilon'=\varphi'(s,s_\varepsilon')+\varepsilon \frac{1}{\tau}(\derlac(s_\varepsilon')\tau+ \lac(s)\sin\varphi- \lac(s_\varepsilon')\sin \varphi')  +\varepsilon^2 R_5(s,s_\varepsilon';\varepsilon)
\end{equation*}
with $\|R_5\|_{\norm{k-2}{\Delta_\mu}} \leq M \| \lac\|_{\mathcal{C}^k}$. Moreover, by the expansion of $s_\varepsilon'$ obtained before and using the notation $\hat \tau= \hat{\tau}(s,\varphi) = \tau(s,s'(s,\varphi))$ in~\eqref{def:tau_hat}, 
\begin{equation}\label{dev varphi eps prime}
\begin{aligned}
\varphi_\varepsilon'= & \varphi'(s,s')+ \varepsilon \partial_2 \varphi' (s,s') \varpi_0(s,\varphi) + \varepsilon \frac{1}{\hat \tau}(\derlac(s')\hat \tau+ \lac(s)\sin\varphi- \lac(s')\sin \varphi')  \\ & +\varepsilon^2  R_7(s,\varphi;\varepsilon)
\end{aligned}
\end{equation} 
with $(s',\varphi')=f_\ga (s,\varphi)$ and $\| {R}_7\|_{\norm{k-3}{\mathbb{A}_\nu}} \leq M \| \lac\|_{\mathcal{C}^k}$. Using that $\varphi'$ is defined through $\partial_2 \tau(s,s')=\sin \varphi'$   and \eqref{der seconde gen},  we obtain that
\[
\cos \varphi' \partial_2 \varphi'(s,s') = \partial_{22} \tau(s,s') = \mathcal{K}(s') \cos \varphi' + \frac{\cos^2 \varphi'}{\tau}.
\]
Therefore, using expression~\eqref{def:varpi0} of $\varpi_0$
\begin{align*}
\varphi'_{\varepsilon} = &\varphi' + \varepsilon \left (\mathcal{K}(s') +  \frac{\cos \varphi'}{\hat \tau} \right ) \varpi_0(s,\varphi) +\varepsilon \frac{1}{\hat \tau}(\derlac(s')\hat \tau+ \lac(s)\sin\varphi-
 \lac(s')\sin \varphi') 
\\ & +\varepsilon^2 R_7(s,\varphi;\varepsilon) \\ 
= & \varphi'   -\varepsilon    \frac{\mathcal{K}(s')}{\cos \varphi'}   \big (\derlac(s) \hat \tau +  \lac(s) \sin \varphi -  \lac(s') \sin \varphi' \big ) + \varepsilon \big (\derlac (s') - \derlac(s))+\varepsilon^2 R_7(s,\varphi;\varepsilon).
\end{align*} 
Together with~\eqref{expresion for svarepsilon} this finishes the proof of of Proposition \ref{cor:expr:fepsilon}. Notice that
by Lemmas~\ref{away from boundary 1} and~\ref{away from boundary 2}, $\|f_\varepsilon\|_{\norm{k-3}{\mathbb{A}_\nu}} \leq 2 \|f_\ga \|_{\mathcal{C}^{k-1}}$. 
\end{proof}

\subsection{Fr\'echet  differentiability of the operator $\Gamma_\nu$. Proof of Proposition \ref{prop Frechet Gamma nu}}\label{sec:Gammanu}  

Recall that the operator $\Gamma_\nu$ sends perturbations $\lambda \in \Bl{\varrho}$ to billiard maps $f_{\dgalac}$ where $\dgalac$ is
given by \eqref{def gamma lambda arc length}, that is:
\[
\dgalac(s) =\ga(s) +  \lac(s) \narc(s),
\]
with $\narc(s) $ the normal vector at $\ga(s)$, which is in arc-length parametrization of $\gna(\mathbb{T})$, and 
$ \lac(s)=\frac {1}{ |\partial \Omega(\gna)|} \lna\circ \sigma^{-1} (s)$, with $\lambda \in \Bl{\varrho}$ and satisfying (see \eqref{lambdaast lambda}):
\begin{equation}\label{lambdaast lambda section 5}
\| \lac\|_{\mathcal{C}^k}  \leq \co\varrho .
\end{equation}
The first result is about the convexity of these billiard tables.

\begin{corollary} \label{cor strictly convex perturbations} 
There exists $\varrho_0>0$ small enough such that, for $0<\varrho \leq \varrho_0$, if $ \lac\in \Bl{\varrho}$, the billiard table $\Omega(\dgalac)$ with $\dgalac$ defined in~\eqref{def gamma lambda arc length} is strictly convex.    
\end{corollary}
\begin{proof} 
Writing the perturbed billiard as 
\[
\dgalac(s) =\ga(s) + \varepsilon\lac_0 (s) \narc(s), \qquad \lac_0=\frac{ \lac}{\| \lac\|_{\mathcal{C}^k}}
\]
with $\varepsilon \le \varrho$, we apply Lemma~\ref{lem:normtangent} to $\dgalac$ and we conclude that if $\varrho$ is small enough the curvature of $\ga + \varepsilon  \lac_0$, is strictly negative and as a result the billiard table $\Omega(\dgalac)$ is strictly convex. 
\end{proof}

As a consequence of Corollary~\ref{cor strictly convex perturbations}, the billiard map $f_{\dgalac}$ associated to $\Omega (\dgalac)$ is well defined and, since $\dgalac \in \mathcal{C}^{k-1} (\mathbb{T},\mathbb{R}^2)$, the general theory assures that $f_{\dgalac} \in \mathcal{C}^{k-2} (\mathbb{A},\mathbb{A})$. 
Let  
\begin{equation}\label{generating function lambda}
\tau_{ \lac} (s, s') = \| \dgalac (s) -\dgalac(s') \|.
\end{equation}
Observe that, even if $\dgalac$ is not parameterized in arclength anymore, the function $\frac{\tau_{ \lac}}{\vert\partial \Omega(\dgalac)\vert}$ is a  generating function of $f_{\dgalac}$ (see Section \ref{sec: preliminaries billiard perturbations}).
Moreover, analogously to~\eqref{first gen fct no arclength}, the map associated to $\dgalac$ is given by 
$f_{\dgalac}(s,\varphi)=(s_\lac',\varphi_\lac')$, 
where 
\begin{equation}\label{eq:implicitperturbedmap}
\|\derdgalac(s)\|  \sin \varphi + \partial_1 \tau _{ \lac} ( s, s_{\lac}')=0 , 
\qquad 
\|\derdgalac( s_{\lac}')\|
\sin \varphi_{\lac}' - \partial_2 \tau _{ \lac}( s, s_{\lac}') =0.
\end{equation}
 
\begin{lemma} 
\label{lem existence billiard map no arclength}  
There exists $0<\varrho_1\leq \varrho_0$ small enough such that, for any fixed $\nu>0$,  if $\lna\in \Bl{\varrho}$, with $\varrho \leq \varrho_1$, then $f_{\dgalac}\in \Bf{b_\nu\varrho}{f_{\ga}}$ with $\Bf{b_\nu\varrho}{f_{\ga}}$ defined in~\eqref{ball fgamma} and $b_\nu$ a constant depending only on $\nu, \varrho_1$. In addition $\|f_{\dgalac}\|_{\difknu{k-3}{\nu}} \leq 2 \|f_{\ga}\|_{\difknu{k-1}{\nu}}$.
\end{lemma}

\begin{proof} 
Take now $\nu>0$ and consider $f_{\dgalac}$, that, for a given $ \lac$, $\dgalac$ is a deformation of $\ga$ as in~\eqref{deformations gamma ast}, namely
\[
\dgalac (s) = \ga (s) + \varepsilon \lac_0(s), \qquad  \lac_0(s) = \frac{\lac(s)}{\|\lac\|_{\mathcal{C}^k}}
\]
with $\varepsilon = \|\lac \|_{\mathcal{C}^k}$. By~\eqref{lambdaast lambda section 5}, $\|\lac\|_{\mathcal{C}^k} \leq \co \varrho$. 
Then, taking $\varrho$ small enough and applying Proposition~\ref{cor:expr:fepsilon},
\[
f_{\dgalac}(s,\varphi) = f_{\ga}(s,\varphi) + \varepsilon E(s,\varphi;\delta),
\]
where the function $E\in \mathcal{C}^{k-3}(\mathbb{A}_\nu, \mathbb{A})$, so that $\|E\|_{\difknu{k-3}{\nu}} \leq M_\nu \|\lac_0\|_{\mathcal{C}^{k}} = M_\nu$ for some constant $M_\nu>0$ that only depends on $\nu$ (recall that $\|\lac_0\|_{\mathcal{C}^k}=1$). 
As a consequence, 
\[
\|f_{\dgalac}-f_{\ga}\|_{\difknu{k-3}{\nu}} =\varepsilon \|E\|_{\difknu{k-3}{\nu}} \leq M_\nu\| \lac\|_{\mathcal{C}^k}  \leq M_\nu \co\varrho.
\]
Thus, taking $b_\nu= M_\nu \co $, we have that 
$f_{\dgalac} \in \Bf{\varrho}{f_{\ga}}$ and 
$\|f_{\dgalac}\|_{\difknu{k-3}{\nu}} \leq 2 \|f_{\ga}\|_{\mathcal{C}^{k-1}}$.  
\end{proof}
         
We recall that a functional $\mathcal{G}:U \to Y$ with $U\subset X$ an open subset and $X,Y$ are normed spaces  is Fr\'echet differentiable in $U$ if, for every $x\in U$, there  exists a linear map $d\mathcal{G}(x) : X \to Y$ such that 
\[
\lim_{\|\delta x\|_X \to 0} \frac{1}{\|\delta x\|_X} \|\mathcal{G}(x+\delta x) - \mathcal{G}(x) - d\mathcal{G}(x) \delta x\|_Y=0.
\]
This is equivalent to say that 
\[
\mathcal{G}(x+ \delta x) = \mathcal{G}(x) + d\mathcal{G}(x) \delta x + o(\delta x), \qquad \lim_{\|x\|_X \to 0} \frac{\|o(\delta x)\|_Y}{\|\delta x\|_X} =0.
\]
In addition if $d\mathcal{G}(x)$ {depends continuously on $x$, 
we say that $\mathcal{G}$ is $\mathcal{C}^1$ Fr\'echet differentiable. 

\begin{remark}\label{rmk Frechet} 
Note that the operator $\mathcal{D}: \mathcal{C}^\ell ( \mathbb{T};\mathbb{R}^m) \to \mathcal{C}^{\ell-1} ( \mathbb{T};\mathbb{R}^m) $, with $\ell \geq 2$, that sends a function to its derivative, $\mathcal{D}(g) = \dot{g} = \partial_s g$ is linear and as a consequence it is $\mathcal{C}^1$ Fr\'echet differentiable with the usual topology. 

In addition, the operator 
$\mathcal{G}: \mathcal{C}^\ell ( \mathbb{T};\mathbb{R}^m) \times \mathbb{T} \to \mathbb{R}^m $, $\ell \geq 2$, $m\in \mathbb{N}_{\geq 1}$ 
given by 
$\mathcal{G}(g,s)=g(s)$ is $\mathcal{C}^1$ Fr\'echet differentiable and 
$d\mathcal{G}(g,s) (\delta g, \delta s) = \der{g}(s) \delta s + \der{\delta g}(s)$.  
\end{remark}

In order to prove that the functional $\Gamma_\nu$ is Fr\'echet differentiable, we first consider the functional:
\begin{equation}\label{map:Gamma_1} 
\begin{split}  
\mathcal{L}:  \Bl{\varrho} \subset \mathcal{C}^k(\mathbb{T}) & \to \Bl{\varrho_1} \subset \mathcal{C}^k(\mathbb{T})  \\ 
\lna&\mapsto \frac{1}{|\partial \Omega(\gna)|}\lna\circ \sigma^{-1}
\end{split}
\end{equation}
with $\sigma^{-1}$ defined by~\eqref{arclength parameterization:section3} and $\varrho_1 = \co \varrho$ as in~\eqref{lambdaast lambda section 5}. It is clear that, since $\mathcal{L}$ is linear, it is $\mathcal{C}^1$ Fr\'echet differentiable.

Secondly we write $\Gamma_\nu = \Gamma_\nu^* \circ \mathcal{L}$ with 
\begin{equation}\label{map:Gamma:prep}  
\begin{split}  
\Gamma^*_\nu:  \Bl{\varrho_1} \subset \mathcal{C}^k(\mathbb{T}) & \to \Bf{\varrho_2}{f_{\ga}} \subset \mathcal{C}^{k-2}(\mathbb{A_\nu}, \mathbb{A})  \\ 
 \lac &\mapsto f_{\dgalac}
\end{split}
\end{equation}
for $\varrho_2= b_\nu \varrho>0$. 
Observe that $\Gamma_\nu^*( \lac)$ is the billiard map $f_{\dgalac}$ associated to the billiard table $\Omega (\dgalac)$, with $\dgalac= \ga +  \lac \cdot \narc$. Recall that $f_{\dgalac}$ is written in coordinates $(s,\varphi)$ where $s$ is the arclength parameterization of $\ga$. 

As a result, Proposition~\ref{prop Frechet Gamma nu} is a straightforward consequence of the following lemma.
\begin{lemma} The operator $\Gamma^*_\nu$ is $\mathcal{C}^1$ Fr\`echet differentiable.   
\end{lemma}
\begin{proof}
To obtain explicitly $(s_{ \lac}',\varphi_{ \lac}')= f_{\dgalac}(s,\varphi) $ we need to solve the equations \eqref{eq:implicitperturbedmap}. 
Moreover, by Corollary \ref{cor strictly convex perturbations}, the billiard map $f_{\dgalac}$ is well defined, and therefore, given $(s,\varphi)\in \mathbb{A}_\nu$ these equations have unique solutions $(s_{\lac}',\vp_{\lac}')=f_{\dgalac} (s,\varphi)$. We want to check that these solutions depend on a $\mathcal{C}^1$-Fr\'echet way on $ \lac$. 

Since the equation on the right  in \eqref{eq:implicitperturbedmap} gives explicitly $\vp'_\lac$, we focus on the equation on the left to solve for $s_\lac '$, and, changing variables to $s'_{\lac}= \ttt +s$,  we will apply the implicit function theorem to the equation $\mathcal{F} ( \lac,s,\vp,\ttt)=0$, where 
\[
\mathcal{F}:\Bl{\varrho}\times  [0,1]\times \left [-\frac \pi 2+\nu,\frac \pi 2 -\nu \right ] \times (\mu, 1-\mu) \to  \mathbb{R}
\]
is given by 
\[
\mathcal{F}( \lac,s,\vp,\ttt) = \|\derdgalac(s)\|  \sin \vp + \partial_1 \tau _{ \lac}(s,s +\ttt).
\]

First, we will check that $\mathcal{F}$ is $\mathcal{C}^1$ Fr\'echet differentiable. On the one hand, since $\ga$ is an arclength parameterization, we can apply Lemma \ref{lem:normtangent} (with $\varepsilon=1$) obtaining:
\[
\| \derdgalac (s)\|= \left [ \left ( 1-  \lac(s) \mathcal{K}(s)\right )^2 + (\derlac(s))^2 \right ]^{1/2} 
\]
with $\mathcal{K}$  the curvature of $\ga (\mathbb{T})$. By Remark~\ref{rmk Frechet} $\| \derdgalac (s)\|$ is $\mathcal{C}^1$ Fr\'echet differentiable. 
Indeed, we only need to emphasize that, taking $\varrho$ small enough 
\[ 
\frac{1}{2} \leq \|\derdgalac(s) \|   \leq 2 .
\]
We consider now (see~\eqref{generating function lambda}) 
$G: \Bl{\varrho} \times [0,1]\times [\mu, 1-\mu] \to \mathbb{R}$ defined as
\[
G( \lac, s, \ttt) =  \tau_{ \lac} (s, s+\ttt) =  \| \dgalac (s)  - \dgalac(s+\ttt)  \|
\]
and $H: \Bl{\varrho} \times [0,1]\times [\mu, 1-\mu] \to \mathbb{R}$
\[
H( \lac, s , \ttt) = \partial_1 \tau_{ \lac}(s,s+\ttt) = 
\frac{1}{ G( \lac,s, \ttt)} \langle \derdgalac(s), \dgalac (s) - \dgalac(s+\ttt)\rangle. 
\]
It is clear that $G, H$ are $\mathcal{C}^1$ Fr\'echet differentiable. Indeed, it is enough to take $\varrho$ small enough such that 
\[
G( \lac,s, \ttt) \geq  \left ( \|\ga(s) - \ga(s+\ttt)\| - 2 \| \lac \|_{\mathcal{C}^k}\right ) \geq (c-2\varrho) >0
\]
where $c$ is given in Lemma \ref{away from boundary 2}. 
As a result, $\mathcal{F}$ is $\mathcal{C}^1$ Fr\'echet. 
 
Take $(s_0, \vp_0,\ttt_0) \in [0,1]\times [-\frac \pi 2+\nu,\frac \pi 2 -\nu] \times (\mu, 1-\mu)$. 
By~\eqref{first gen fct:section3}, $\mathcal{F}(0,s_0, \vp_0, \ttt_0)=0$ with $\ttt_0= s_0'-s_0$, where $(s'_0, \vp'_0)= f_{\ga}(s_0,\vp_0)$ is the original billiard map at the point $(s_0,\vp_0)$. 
By~\eqref{der seconde gen}, and using that $\tau_0=\tau$, the generating function of $f_{\ga}$,
\begin{align*}
\partial_\ttt \mathcal{F} (0,s_0, \vp_0, \ttt_0) &= \partial_{12} \tau_0(s_0, s_0+\ttt_0)= \partial_{12} \tau (s_0, s'_0)= \dfrac{\cos \vp_0\cos\vp_0'}{\tau(s_0,s'_0)}  \geq \frac{\sin \nu \, \sin \nu'}{c}>0 
\end{align*}
and then by the implicit function theorem, there exist $\delta_{s_0,\vp_0}>0$, $\varrho_{s_0,\vp_0}>0$  and a $\mathcal{C}^1$ Fr\'echet map 
\[
\mathcal{S}:\Bl{\varrho_{s_0,\vp_0}} \times (s_0-\varrho_{\ss_0,\vp_0}, s_0+\varrho_{\ss_0,\vp_0}) \times  (\vp_0-\varrho_{s_0,\vp_0}, \vp_0+\varrho_{s_0,\vp_0}) \to (\ttt_0-\delta_{s_0,\vp_0}, \ttt_0+\delta_{s_0,\vp_0})
\]
such that 
\[
\mathcal{F} ( \lac, s, \vp, \ttt)=0 \qquad \text{if and only if } \qquad \ttt= \mathcal{S} ( \lac, s, \vp).
\]
By uniqueness and using that $\tau_{ \lac}(s+1,t+1)=\tau_{ \lac}(s,t)$, we have that $s_{\lac}'=\mathcal{S}( \lac, s,\vp)+s$. 

Now we consider the open cover of $A_\nu:=[0,1] \times \left [-\frac{\pi}{2} + \nu, \frac{\pi}{2}-\nu\right ]$ given by 
\[
A_\nu \subset \bigcup_{(s_0,\vp_0)\in \mathbb{A}_\nu} (s_0-\varrho_{s_0,\vp_0}, s_0+\varrho_{\ss_0,\vp_0}) \times (\vp_0-\varrho_{s_0,\vp_0}, \vp_0+\varrho_{s_0,\vp_0})
\]
Since $A_\nu$ is a compact set, there exist a finite subcover, namely, 
\[
A_\nu \subset  \bigcup_{k=1}^{k_0} (s^k-\varrho^k, s^k+\varrho^k) \times (\vp^k-\varrho^k, \vp^k+\varrho^k)
\]
with $\varrho^k = \varrho_{s^k,\vp^k}$. Let $\varrho\leq \min_{1\leq k\leq k_0} \varrho^k$ so that $b_\nu \varrho \leq \min\{1\leq k \leq k_0\} \delta_{\ss^k,\vp^k}$ where $b_\nu$ is the constant provided in Lemma~\ref{lem existence billiard map no arclength}. 
 
We can now define the map 
\[
\mathcal{S}:\Bl{\varrho} \times \mathbb{A}_\nu\to \mathbb{R}
\]
globally as follows. 
For any $(\lac,s,\vp) \in \Bl{\varrho} \times \mathbb{A}_\nu$, take $1\leq k \leq k_0$ such that the projection onto $\mathbb{A}_\nu$, 
$(s,\vp) \in (s^k - \varrho^k, s^k+ \varrho^k) \times (\vp ^k - \varrho^k, \vp^k+ \varrho^k)$. 
Then, $\lac \in \Bl{\varrho} \subset \Bl{\varrho^k}$ so that $\mathcal{S}(\lac,s,\vp)$ is well defined by the implicit function theorem. By uniqueness $\mathcal{S}$ is well defined and it is $\mathcal{C}^1$ Fr\'echet differentiable.

Since by Lemma~\ref{lem existence billiard map no arclength}, $\|\pi_s f_{\dgalac} - s'\|_{\difknu{k-3}{\nu}} \leq b_\nu \varrho$, by uniqueness, $ \pi_{s} f_{\dgalac}(s,\vp)= \mathcal{S} (\lac,s,\vp)+s$, with $\pi_{s}$ the projection onto the $s-$component. Then, by~\eqref{eq:implicitperturbedmap}, 
\[
\pi_{\vp} f_{\dgalac} (s,\vp) = \mathrm{arcsin} \left ( \frac{\partial_2 \tau_{\lac} (s, \pi_s f_{\dgalac}(s, \vp))}
{\|\derdgalac (\pi_s f _{\dgalac}
(s,\vp))\|}\right )
\]
defines de $\vp-$ component of $f_{\dgalac}$. 
One can reasoning in an analogous way as for $\pi_s$ to check that it is also a $\mathcal{C}^1$ Fr\'echet map, and this concludes the proof. 
\end{proof}

To end this section we compute $d \Gamma_\nu(0)$.
To compute $d \Gamma_\nu(0)\etana$, with $\etana\in B(1)$, we strongly rely in the computations done in Section~\ref{billiard perturbationTable}. Indeed, recall that 
$\Gamma_\nu = \Gamma_\nu^* \circ \mathcal{L}$ with $\mathcal{L}, \Gamma_\nu^*$ defined in~\eqref{map:Gamma_1} and~\eqref{map:Gamma:prep} respectively.
Therefore  
\[
d\Gamma_\nu (0)\etana= d\Gamma_\nu^* (\mathcal{L}(0))d\mathcal{L}(0)\etana = d\Gamma_\nu^* (0)d\mathcal{L}(0)\etana.
\]

We observe that,  for any $\lna$,
\begin{equation}\label{dif Lambda_1}
d\mathcal{L}(\lna) \etana = \frac{1}{|\partial \Omega(\gna)|} \etana \circ \sigma^{-1}, \qquad \etana\in B(1),
\end{equation}
and it only remains to compute $d\Gamma^*_\nu (0) \etarc $ with $\etarc= |\partial \Omega (\gna)|^{-1} \etana \circ \sigma^{-1} $. 

Recall that, since $\Gamma^*_\nu$ is $\mathcal{C}^1$ Fr\'echet differentiable $d\Gamma^*_\nu (0)\etarc$ coincides with the Gateaux differential of $\Gamma^*_\nu $ at $0$ in the $\etarc\in B(1)$ direction, namely
\begin{equation}\label{Gateaux_def}
d\Gamma^*_\nu(0)\etarc = \mathbf{d}_{\etarc} \Gamma^*_\nu (0) = \lim_{\varepsilon \to 0 } \frac{\Gamma^*_\nu( \varepsilon \etarc) - \Gamma^*_\nu(0) }{\varepsilon}.
\end{equation}
Since $\defga = \ga + \varepsilon \etarc \cdot \narc$ belongs to the set of deformations $\Omega_\varepsilon$ in~\eqref{deformations}, we can apply 
Proposition~\ref{cor:expr:fepsilon} and we have that 
\begin{equation}\label{def:Gamma}
\big [d\Gamma^*_\nu (0)\etarc \big ] (s,\varphi) = A(z,z') \left (\begin{array}{c} \etarc(s) \\ \deretarc(s) \end{array} \right ) + 
B(z' )\left (\begin{array}{c} \etarc(s') \\ \deretarc(s') \end{array} \right ).
\end{equation}
with $z=(s,\varphi)$, and $f(z)=(s',\vp')$ 
\begin{equation}\label{def:A:compactsup}
A(z,z')=-\frac{1}{\cos  \varphi'} \left ( \begin{array}{cc} \sin \varphi & \hat \tau(s,\varphi) \\ \sin \varphi \mathcal{K}(s') & \mathcal{K}(s') \hat \tau(s,\varphi) + \cos \varphi' \end{array}\right )
\end{equation}
and 
\begin{equation}\label{def:B:compactsup}
B(z')= \frac{1}{\cos \varphi'} \left ( \begin{array}{cc} \sin \varphi' & 0 \\ \sin \varphi' \mathcal{K}(s') & \cos \varphi' \end{array}\right ).
\end{equation}

Moreover, we also compute $\partial_s d\Gamma^*_\nu(0)$ and $\partial_\varphi d\Gamma^*_\nu(0)$:
\begin{equation}\label{def:dGamma}
\begin{aligned}
\partial_s \big [d\Gamma^*_\nu(0)\etarc \big ] (s,\varphi) = &\left [ \partial_s A(z,z') + \partial_s s'\partial_{s'} A(z,z')  + \partial_{s} \varphi' \partial_{\varphi'} A(z,z') \right ] \left (\begin{array}{c} \etarc(s) \\ \deretarc(s) \end{array} \right ) 
\\ & 
+ A(z,z') \left (\begin{array}{c} \deretarc(s) \\ \dertwoetarc(s) \end{array} \right ) \\
&+ \left [ \partial_s s' \partial_{s'} B(z')  + \partial_{s} \varphi'\partial_{\varphi'} B(z') \right ] \left (\begin{array}{c} \etarc(s' ) \\ \deretarc(s') \end{array} \right ) 
\\ &
+ \partial_s s' B(z')  
\left (\begin{array}{c} \deretarc(s') \\ \dertwoetarc(s') \end{array} \right ), \\ 
\partial_{\varphi} \big [d\Gamma^*_\nu(0)\etarc \big ] (s,\varphi) = &\left [ \partial_\varphi A(z,z') + \partial_\varphi s'\partial_{s'} A(z,z')  + \partial_{\varphi} \varphi'  \partial_{\varphi'} A(z,z') \right ] \left (\begin{array}{c} \etarc(s) \\ \deretarc(s) \end{array} \right ) 
\\
&+ \left [ \partial_\varphi s'\partial_{s'} B(z')  + \partial_{\varphi} \varphi' \partial_{\varphi'} B(z') \right ] \left (\begin{array}{c} \etarc(s') \\ \deretarc(s') \end{array} \right ) 
\\ &+ \partial_\varphi s' B(z')  
\left (\begin{array}{c} \deretarc(s') \\ \dertwoetarc(s') \end{array} \right ).
\end{aligned}
\end{equation}

\section{keeping a single hyperbolic Aubry-Mather periodic orbit of type $(p,q)$. Proof of Proposition~\ref{prop:hyperbolicfixedpoint:section3}}\label{sec:hyp}

The goal of this section is to give the proof of Proposition~\ref{prop:hyperbolicfixedpoint:section3} stated in Section~\ref{sec:hyp:section3}. 
In other words, 
given $r>0$, $\gna\in\mathcal{B}_r$,   $\frac p q\in \mathbb{Q}/\mathbb{Z}$, and $\epsilon>0$ arbitrarily small, we will find suitable analytic functions $\lna$, such that $\|\lna\|_r <\epsilon$, and for \begin{equation}\label{defi_gammalambd}
\dgnalna(\ss) = \gna(\ss) + \lna(\ss) \nna(\ss), 
\end{equation}
the billiard map $f_{\dgnalna}$ of the perturbed billiard table $\Omega(\dgnalna)$ has a unique periodic orbit in the Aubry-Mather set $\mathcal{M}_{\frac{p}{q}}(\Omega(\dgnalna))$ of rotation number $\frac{p}{q}$, which, in addition, is hyperbolic. 
\begin{proof}[Proof of Proposition~\ref{prop:hyperbolicfixedpoint:section3}]
Let $f_{\gna}$ be the billiard map associated to $\gna$, let $\Omega=\Omega(\gna)$, and let $\gentau \colon 
(\ss,\ss')\mapsto \|\gna(\ss)-\gna(\ss')\|$, see \eqref{generating function original:section3}. 
 
Consider the action functional
\begin{equation*}
L^0_{p,q}\colon
\left\{
\begin{array}{rcl}
\R^q &\to &\R_+,\\
(\ss_0,\cdots,\ss_{q-1}) &\mapsto& \sum_{k=0}^{q-1}\gentau(\ss_k,\ss_{k+1}),
\end{array}
\right.
\end{equation*}
with the convention that $\ss_q=\ss_0+p$.  

Periodic orbits in the Aubry-Mather  set $\mathcal{M}_{\frac{p}{q}}(\Omega)$  correspond to maximizers of the action functional $L^0_{p,q}$ (see Section~\ref{sec: preliminaries billiard perturbations}). Let us then fix a maximizer $\bar \ss$ of $L^0_{p,q}$. We choose a configuration $(\bar \ss_0,\cdots,\bar \ss_{q-1}) \in \mathbb{R}^{q}$     representing $\bar \ss$, and denote by $\bar\varphi_0,\cdots,\bar\varphi_{q-1}$ the respective angles at the collision points $\bar \ss_0,\cdots,\bar \ss_{q-1}$. We also denote by $\bar S$ the associated bi-infinite configuration 
\[
\bar S=(\bar S_j)_{j\in \mathbb{Z}},\quad \bar S_j:=\bar \ss_{j\mod q}+jp,\, \forall\, j \in \mathbb{Z},
\]
and by $\bar{\mathcal{S}}:=\{(\bar S_{j},\bar S_{j+1},\cdots, \bar S_{j+q-1}):j \in \mathbb{Z}\}$ the set of all words of length $q$ in $\bar S$. Let $\lna_0\colon \T\to \R$ be a function which meets the following requirements:
\begin{itemize}
\item $\lna_0\leq 0$ and $\lna_0<0$ on $\T\setminus \cup_{k=0}^{q-1} \{\bar \ss_k\}$;
\item $\lna_0(\bar \ss_k)=\derlna_0(\bar \ss_k)=0$ and $\dertwolna_0(\bar \ss_k) <0$, for any $k\in \{0,\cdots,q-1\}$.
\end{itemize}
We can take e.g. $\lna_0 \colon \ss \mapsto - \prod_{k=0}^{q-1} \sin^2 \left(\pi(\ss-\bar{\ss}_k)\right)$. 
 
We consider the one-parameter family $\{\gna_{\varepsilon \lna}=\gna+\varepsilon \lna n\}_{\varepsilon\in \mathbb{R}}$ of deformations of $\gna$ as in~\eqref{defi_gammalambd}, with $\lna= \|\lna_0\|_r^{-1} \lna_0$. For $\varepsilon\geq 0$ small, let $f_{\gna_{\varepsilon \lna}}$ be the associated billiard map. Let  $\gentau_{\varepsilon} \colon (\ss,\ss')\mapsto \|\gna_{\varepsilon \lna}(\ss)-\gna_{\varepsilon \lna}(\ss')\|$, and let us consider the action functional
\begin{equation*}
L^\varepsilon_{p,q}\colon
\left\{
\begin{array}{rcl}
\R^q &\to &\R_+,\\
(\ss_0,\cdots,\ss_{q-1}) &\mapsto& \sum_{k=0}^{q-1}\gentau_\varepsilon(\ss_k,\ss_{k+1}),
\end{array}
\right.
\end{equation*}
with the convention that $\ss_q=\ss_0+p$. 
    
As for $L^0_{p,q}$, periodic orbits in the Aubry-Mather set $\mathcal{M}_{\frac{p}{q}}(\Omega(\gna_{\varepsilon \lna}))$  correspond to maximizers of the action functional $L^\varepsilon_{p,q}$. Recall that $\bar{\mathcal{S}}$ is the set of all $q$-words in the bi-infinite sequence $\bar S$ associated to the maximizer $\bar \ss$ fixed above. 

Let $\ss=(\ss_0,\cdots,\ss_{q-1})\in \mathbb{R}^{q}$, and let $\ss_{-1}=\ss_{q-1}-p$, $\ss_q=\ss_0+p$. 
We claim that if $\ss\notin \bar{\mathcal{S}}$, then $L_{p,q}^\varepsilon(\ss)<L_{p,q}^\varepsilon(\bar \ss)$ for $\varepsilon>0$ sufficiently small.  
Indeed, for $k \in \{0,\cdots,q-1\}$, we let $\varphi_k^{-}$, resp. $\varphi_k^{+}$ denote the angles between the normal to $\partial\Omega$ at $\gna(\ss_k)$ and the line segment from $\gna(\ss_k)$ to $\gna(\ss_{k-1})$, resp. from $\gna(\ss_k)$ to $\gna(\ss_{k+1})$. 
By Lemma~\ref{lazutkin}, there exist $\varepsilon_0>0$ and $\nu>0$ such that for any $\varepsilon \in (-\varepsilon_0, \varepsilon_0)$, we have $\mathcal{M}_{\frac{p}{q}}(\Omega(\gna_{\varepsilon \lna}))\subset \mathbb{A}_\nu$ (recall~\eqref{def:Anu}). 
In particular, we are free to restrict ourselves to configurations $\ss$ for which $(\ss_k,\varphi_k^\pm)\in \mathbb{A}_\nu$ for each $k \in \{0,\cdots,q-1\}$, for otherwise, the configuration cannot be a maximizer. Then, by Lemmas~\ref{away from boundary 1} and~\ref{away from boundary 2}, there exists a constant $c>0$ such that, for any $\varepsilon \in (-\varepsilon_0, \varepsilon_0)$,     
\[
\gentau_\varepsilon(\ss_{k},\ss_{k+1})\geq c\quad \text{and}\quad \cos  (\varphi_k^\pm)  \geq \sin \nu,\qquad \forall\, k \in \{0,\cdots,q-1\}. 
\]
 
On the one hand, by Lemma~\ref{lemma pertr lenght func}, for every $k\in \{0,\cdots,q-1\}$, since $\lna_0(\bar{\ss}_k)=0$, we have that $\gentau_{\varepsilon}(\bar{\ss}_k,\bar{\ss}_{k+1})= \gentau (\bar{\ss}_k,\bar{\ss}_{k+1})$ and then $L^\varepsilon_{p,q}(\bar \ss)=L^0_{p,q}(\bar \ss)$. 

On the other hand, by Lemma~\ref{lemma pertr lenght func}, for every $k\in \{0,\cdots,q-1\}$, since $\lna_0\leq 0$, since $\lna_0(\ss_k)<0$, if $\ss\notin \bar{\mathcal{S}}$, 
\begin{align*}
\gentau_\varepsilon^2 (\ss_k,\ss_{k+1}) \leq &\gentau^2 (\ss_k,\ss_{k+1}) - \frac{2 \varepsilon c \sin \nu }{\|\lna_0\|_r}\big (|\lna_0 (\ss_k)| + |\lna_0(\ss_{k+1})|\big ) \\ & + \frac{\varepsilon^2 }{\|\lna_0\|_r^2}\big (|\lna_0(\ss_k)| + |\lna_0(\ss_{k+1})|)^2 \\ 
\leq &\gentau^2 (\ss_k,\ss_{k+1})  - \frac{\varepsilon}{\|\lna_0\|^2_r} (|\lna_0(\ss_k)| + |\lna_0(\ss_{k+1})|) \left (2 c \sin \nu -  2 \varepsilon\right ) \\  <& \gentau^2 (\ss_k,\ss_{k+1})
\end{align*}
if $\varepsilon$ is such that $ 0< \varepsilon  < c\sin \nu $. 
Then we have that, for $\ss\notin \bar{\mathcal{S}}$,  
we have 
\[
L_{p,q}^{\varepsilon} (\ss) < L_{p,q}^0(\ss) \leq L_{p,q}^{0}(\bar{\ss}) = L_{p,q}^\varepsilon(\bar{\ss}).
\]
This proves that $\bar{\ss}$ is the unique mazimizer of $L_{p,q}^\varepsilon$. 
Moreover, since by Lemma~\ref{lemma pertr lenght func}
\begin{equation*}
L_{p,q}^\varepsilon(\ss_0,\cdots,\ss_{q-1})=L_{p,q}^0(\ss_0,\cdots,\ss_{q-1})+ \frac{\varepsilon}{\|\lna_0\|_r} L_{p,q}^1(\ss_0,\cdots,\ss_{q-1}) +O(\varepsilon^2),	
\end{equation*}
where
\[
L_{p,q}^1\colon (\ss_0,\cdots,\ss_{q-1})\mapsto 
\sum_{k=0}^{q-1} \big [\lna_0(\ss_k)\cos \big (\varphi(\ss_k,s_{k+1}) \big )+\lna_0(\ss_{k+1})\cos \big (\varphi'(\ss_k,\ss_{k+1}) \big ) \big ]
\]
with $\varphi(\ss_k,\ss_{k+1})$ being the angle between $-\nna(\ss_k)$ and $\gna(\ss_{k+1})- \gna(\ss_k)$, we have that
\begin{equation*}
D^2L_{p,q}^\varepsilon(\bar \ss)=D^2L_{p,q}^0(\bar \ss)+2 \frac{\varepsilon}{\|\lna_0\|_r}\,  \mathrm{diag}(\dertwo{\lna}_0(\bar \ss_k)\cos \bar\varphi_k)_{k=0,\cdots,q-1}+O(\varepsilon^2)\, ,
\end{equation*}
using the fact that $\lna_0(\bar{\ss}_k)=\derlna_0(\bar{\ss}_k)=0$ for every $k$.
    
Therefore,
\begin{equation}\label{exp pert le fu}
\det D^2L_{p,q}^\varepsilon(\bar \ss)= \det D^2L_{p,q}^0(\bar \ss)+2 \frac{\varepsilon}{\|\lna_0\|_r}\sum_{k=0}^{q-1} \dertwolna_0(\bar \ss_k) \cos \bar\varphi_k\,\delta_{p,q}^k(\bar \ss)+O(\varepsilon^2),
\end{equation}
where $\delta_{p,q}^k(\bar \ss)$ is the determinant of the matrix $(D^2L_{p,q}^0(\bar \ss))_{i,j\neq k}$. 
In particular, as $\bar \ss$ is a local maximizer of the action functional $L_{p,q}^0$, the matrix $D^2L_{p,q}^0(\bar \ss)$ is negative semi-definite, and then, $(-1)^q  \det D^2L_{p,q}^0(\bar \ss) \geq 0$.
In fact, since $\bar \ss$ is in the Aubry-Mather set, by \cite{ArnaudSalto}, and since, once again, orbits in the Aubry-Mather set correspond to maximizers of $L^0_{p,q}$, for any $k\in \{0,\cdots,q-1\}$, the functional
\begin{equation*}
\hat{L}_{p,q}^{k}\colon
\left\{
\begin{array}{rcl}
\R^{q-1} &\to &\R_+,\\
(\ss_0,\cdots,\ss_{k-1},\ss_{k+1},\cdots,\ss_{q-1}) &\mapsto& 
\sum_{j=0}^{k-2}\gentau(\ss_j,\ss_{j+1})+\gentau(\ss_{k-1},\bar \ss_k)+\gentau(\bar \ss_{k},\ss_{k+1})\\
& &+\sum_{j=k+1}^{q-2}\gentau(\ss_j,\ss_{j+1})+\gentau(\ss_{q-1},\ss_0+p),
\end{array}
\right.
\end{equation*}
has a non-degenerate critical point at $(\bar \ss_0,\cdots,\bar \ss_{k-1},\bar \ss_{k+1},\cdots,\bar \ss_{q-1})$ (actually, a maximizer). Therefore, 
\begin{equation}\label{signdeltapq}
(-1)^{q-1}\delta_{p,q}^k(\bar \ss)> 0,\quad \forall\,k\in \{0,\cdots,q-1\},
\end{equation}
i.e., $\delta_{p,q}^k(\bar \ss)$ has the same sign as $(-1)^{q-1}$. By~\eqref{exp pert le fu}, there are two cases: 
\begin{itemize}
\item 
if $(-1)^q\det D^2L_{p,q}^0(\bar \ss)>0$, then for $\varepsilon>0$ sufficiently small, we have $(-1)^q\det D^2L_{p,q}^\varepsilon(\bar \ss)>0$;
\item 
if $\det D^2L_{p,q}^0(\bar \ss)=0$, then by~\eqref{signdeltapq} and since $\dertwolna_0(\bar \ss_k) <0$, for any $k\in \{0,\cdots,q-1\}$, for $\varepsilon>0$ sufficiently small, we also have $(-1)^q\det D^2L_{p,q}^\varepsilon(\bar \ss)>0$.
\end{itemize}
We have thus proved that for $\varepsilon>0$ sufficiently small, $\bar \ss$ is the unique periodic orbit in the Aubry-Mather set of rotation number $\frac p q$ for the perturbed billiard map, and that $D^2L_{p,q}^\varepsilon(\bar \ss)$ is non-degenerate. 
In particular, the periodic orbit $\bar \ss$ is hyperbolic (see e.g.~\cite{MacKayStark}).  
\end{proof}

\section{Existence of one fibered homoclinic points. Proof of Lemma \ref{lemma at most two fibered points:section3}, Proposition \ref{nec cond for not one fib:section3} and Proposition \ref{good property in the bas case:section3}}
\label{sec:existence one fibered}

This section is devoted to study if and when a homoclinic point belongs to a one-fibered orbit. As already emphasized in Section~\ref{sec:existence one fibered:sec3}, we will prove here that, in the case we are interested in, any homoclinic orbit is one-fibered if $q>2$, or it could have a two-symetrically fibered point when $q=2$.

Lemma~\ref{lemma at most two fibered points:section3} counts the maximum number of points belonging to the same orbit on the same fiber. We start by proving it.

\begin{proof}[Proof of Lemma~\ref{lemma at most two fibered points:section3}]
Let $P,P'\in \mathcal{P}$ be such that $Q\in W^\sta(P)\cap W^\uns(P')$. For $\ell\gg 1$ very large, $f_\gna^\ell(Q)\in W_{\mathrm{loc}}^\sta(f_\gna^\ell(P))$ and $f_\gna^{-\ell}(Q)\in W_{\mathrm{loc}}^\uns(f_\gna^{-\ell}(P'))$. Then $f_\gna^\ell(Q)$ (resp.  $f_\gna^{-\ell}(Q)$) is very close to $f_\gna^\ell(P)\in \mathcal{P}$ (resp.  $f_\gna^{-\ell}(P')$). On the one hand, for any  $i\in \{0,\cdots,q-1\}$, the stable space at $P_i$ is transverse to be vertical because $P_i$ is in the Aubry-Mather set, thus $W_{\mathrm{loc}}^\sta(P_i)$ is a partial graph over the coordinate $\ss$. Moreover, by Aubry-Mather theory (see Proposition~\ref{prop:injectivity}), all the points $(P_i)_{i=0,\cdots,q-1}$ in $\mathcal{P}$ lie on different vertical fibers. We conclude that for $\ell^\sta\gg 1$ sufficiently large, the points $\{f_\gna^\ell(Q):\ell \geq \ell^\sta\}$ all lie on different vertical fibers. Arguing similarly for negative iterates, for $\ell^\uns\gg 1$ sufficiently large, the points $\{f_\gna^{-\ell}(Q):\ell \geq \ell^\uns\}$ all lie on different vertical fibers. 
Up to choose a larger $\hat{\ell}^s\ge \ell^s$ and denoting by $p_1\colon \mathbb{A}\to\mathbb{T}$ the projection over the first coordinate, the set of points $\{f_\gna^{\ell}(Q):\ell \geq \hat{\ell}^\sta\}$ projects injectively on the first coordinate and its image by $p_1$ is disjoint from $p_1(\{f_\gna^\ell(Q):-\ell^\uns\leq\ell \leq \ell^\sta\})$. Similarly, up to choose a larger $\hat{\ell}^\uns\geq \ell^\uns$, the set of points $\{f_\gna^{-\ell}(Q):\ell \geq \hat{\ell}^\uns\}$ projects injectively on the first coordinate and its image by $p_1$ is disjoint from $p_1(\{f_\gna^\ell(Q):-\ell^\uns\leq\ell \leq \ell^\sta\})$.
We conclude that for $\ell\geq \max(\hat{\ell}^\sta,\hat{\ell}^\uns)$, there exists a most one other point $f_\gna^{-m}(Q)$, $m\geq \max(\ell^\sta,\ell^\uns)$, in the orbit of $Q$, which lies on the same vertical fiber as $f_\gna^\ell(Q)$. 
\end{proof}

We are going to see now that a general homoclinic point is localised one-fibered (see Definition~\ref{definition localised fibered:section3}), as stated in Proposition~\ref{nec cond for not one fib:section3}.

\begin{proof}[Proof of Proposition~\ref{nec cond for not one fib:section3}]
Let $\mathcal{P}=(P_i)_{i=0}^{q-1}$ be the unique periodic orbit in the Aubry-Mather set of rotation number $\frac p q$. Recall that $p_1\colon \mathbb{A}\to \mathbb{T}$ denotes the projection on the first coordinate. In particular, from Aubry-Mather theory (see Proposition~\ref{prop:injectivity}), we know that
\[
p_1(P_i)=\ss_{P_i}\neq p_1(P_j)=\ss_{P_j}\qquad \forall i\neq j\, .
\]
Since there is a finite number of points, there exists $\epsilon>0$ such that, considering, for every $i$, the intervals $(\ss_{P_i}-\epsilon, \ss_{P_i}+\epsilon)$, one has
\[
(\ss_{P_i}-\epsilon, \ss_{P_i}+\epsilon)\cap (\ss_{P_j}-\epsilon, \ss_{P_j}+\epsilon)=\emptyset\qquad \forall i\neq j\, .
\]
Observe that, again thanks to Aubry-Mather theory (since these points do not have conjugate points), at every $P_i$ the stable and unstable directions are transverse to the vertical one. Thus, the local stable and unstable manifolds of the point $P_i$ are partial graphs over $\mathbb{T}$. Up to shrink the local stable and unstable manifolds, we can assume that they are all graphs on domains contained in the union $\cup_{i=0}^{q-1}(\ss_{P_i}-\epsilon, \ss_{P_i}+\epsilon)$.

Let now $Q$ be a homoclinic point in $W^{\sta}(P)\cap W^{\uns}(P')$. There exists $N\in \mathbb{N}$ such that for every $n\geq N$:
\begin{itemize}
\item the point $f_\gna^n(Q)$ belongs to the union of the local stable manifolds $\cup_{i=0}^{q-1}W^{\sta}_{\mathrm{loc}}(P_i)$;
\item the point $f_\gna^{-n}(Q)$ belongs to the union of the local unstable manifolds $\cup_{i=0}^{q-1} W^{\uns}_{\mathrm{loc}}(P_i)$.
\end{itemize}
For $i=0,\dots, q-1$, denote by $ D^{\sta}_{P_i}$ (resp. $ D^{\uns}_{P_i}$) the domain of the local stable (resp. unstable) manifold at $P_i$. Since there is a finite number of points in the orbit of $Q$ which do not belong to such local stable and unstable manifolds, for every $i$, we can restrict further the domain $D^{\sta}_{P_i}$ (resp. $ D^{\uns}_{P_i}$) such that the set of points in the orbit of $Q$ whose projection on the first coordinate is in $D^\sta_{P_i}\setminus\{P_i\}$ (resp. $D^\uns_{P_i}\setminus\{P_i\}$) is made up only of the points of the orbit of $Q$ belonging to the local stable (resp. local unstable) manifold at $P_i$. 

The point $Q$ is in $W^{\sta}(P)\cap W^{\uns}(P')$. Then, there exists $k_1,k_2\in\mathbb{N}$ such that
\[
f_\gna^{k_1}(Q)\in p_1^{-1}( D^{\sta}_P\setminus\{\ss_P\})\qquad \text{and}\qquad f_\gna^{-k_2}(Q)\in p_1^{-1}( D^{\uns}_P\setminus\{\ss_P\})\, ,
\]
since $P,P'$ are points lying on the same periodic orbit. Assume that $f_\gna^{k_1}(Q)$ is the first point of the orbit of $Q$ belonging to the local stable manifold, whose first-coordinate projection is in $ D^{\sta}_P\cap  D^{\uns}_P$. Similarly, assume that $f_\gna^{-k_2}(Q)$ is the last point of the orbit of $Q$ belonging to the local unstable manifold, whose first-coordinate projection is in $ D^{\sta}_P\cap  D^{\uns}_P$. Thus, for every $n\ge 0$, one has
\[
f_\gna^{k_1+nq}(Q)\in p_1^{-1}( D^{\sta}_P\setminus\{\ss_P\})\qquad \text{and}\qquad f_\gna^{-k_2-nq}(Q)\in p_1^{-1}( D^{\uns}_P\setminus\{\ss_P\})\, ,
\]
while, for every $n\ge 1$ 
and every $i\not\equiv 0 \mod q$, we have
\[
f_\gna^{k_1+ni}(Q)\notin p_1^{-1}( D^{\sta}_P\setminus\{\ss_P\})\qquad \text{and}\qquad f_\gna^{-k_2-ni}(Q)\notin p_1^{-1}( D^{\uns}_P\setminus\{\ss_P\})\, .
\]
This is possible because of the choice of the domains $ D^{\sta}_{P_i}$ and $ D^{\uns}_{P_i}$ and by the properties of the local stable and unstable manifolds. 

We then look at the set of points
$$
\mathcal{S}:=\{ f_\gna^{k_1+nq}(Q) :\ n\ge 0 \}\cup \{ f_\gna^{-k_2-nq}(Q) :\ n\ge 0 \}\, ;
$$
there are two exclusive possibilities. 
\begin{itemize}
\item Either there exists a point in $\mathcal{S}$ that is one-fibered;
\item or, since the local stable and unstable manifolds are graphs over the domain $ D^{\sta}_P\cap  D^{\uns}_P$, we deduce that, for every $n\ge 0$, we have
\begin{equation}\label{condition not one fibered}
p_1(f_\gna^{k_1+nq}(Q))=p_1(f_\gna^{-k_2-nq}(Q))\, .
\end{equation}
\end{itemize}
The first possibility would contradict the hypothesis that the orbit of $Q$ is not localised one-fibered. Thus, consider the second possibility. For every $i\ge 0$, we denote
\[
(\ss_{k_1+i},\varphi^+_{k_1+i})=f_\gna^{k_1+i}(Q)\qquad\text{and}\qquad (\ss_{-k_2-i},\varphi^-_{-k_2-i})=f_\gna^{-k_2-i}(Q)\, .
\]
According to this notation, from \eqref{condition not one fibered}, we have, for every $i\ge 0$,
\[
\ss_{k_1+iq}=p_1(\ss_{k_1+iq},\varphi^+_{k_1+iq})=p_1(\ss_{-k_2-iq},\varphi^-_{-k_2-iq})=\ss_{-k_2-iq}\, .
\]
Moreover, for every $i\ge 0$, we have
\[
(\ss_{k_1+i+1},\varphi^+_{k_1+i+1})=f_\gna(\ss_{k_1+i},\varphi^+_{k_1+i})\quad\text{and}\quad (\ss_{-k_2-i},\varphi^-_{-k_2-i})=f_\gna(\ss_{-k_2-i-1},\varphi^-_{-k_2-i-1})\, .
\]
Up to change the order of the indices, we can assume that $P=P_0$. Up to further restrict the domains $D^{\sta}_{P_0}, D^{\uns}_{P_0}$ and up to replace $k_j$ with $k_j+q$ for $j=1,2$, we can assume that the first $q$ iterates in the future and in the past of $f_\gna^{k_1}(Q)$ and of $f_\gna^{-k_2}(Q)$ have a first coordinate projection that still belongs to the domains $ D^{\sta}_{P_j}, D^{\uns}_{P_j}$ (for the suitable $j$).

Denoting $P_q=P_0$, observe that, for $j=0,\dots, q-1$, we have
\[
f_\gna^{k_1+nq+j}(Q)\in W^{\sta}_{\mathrm{loc}}(P_j)\qquad \forall n\ge 0\, ,
\]
and
\[
f_\gna^{-k_2-nq-j}(Q)\in W^{\uns}_{\mathrm{loc}}(P_{q-j})\qquad \forall n\ge 0\, .
\]
Actually, their projections on the first coordinate are all contained in the domains $ D^{\sta}_{P_j}$, resp. $D^{\uns}_{P_{q-j}}$.

Again, there is a dichotomy. 
\begin{itemize}
\item Either there exists an index $j=1,\dots, q-1$ such that there exists a point in the set
\[
\mathcal{S}_j:=\{f_\gna^{k_1+nq +j}(Q) :\ n\ge 0\}\cup\{f_\gna^{-k_2-nq-(q-j)}(Q):\ n\ge 0\}
\]
which is one-fibered;
\item or for every point there is another point of the orbit on the same fiber. More precisely, for every $j=1,\dots,q-1$ and for every $n\ge 0$, there exists $m\in \mathbb{Z}$ such that
\begin{equation}\label{condition not one fibered everywhere}
p_1(f_\gna^{k_1+nq+j}(Q))=p_1(f_\gna^{-k_2-mq-(q-j)}(Q))\, .
\end{equation}
\end{itemize}
Again, the first possibility contradits the hypothesis on the orbit pof $Q$ not being localised one-fibered. We then consider this second case. See Figure~\ref{fig: two periodic case}.  
\begin{figure}
\begin{minipage}{{0.32\textwidth}}
\vspace{1.2cm}
\begin{overpic}[scale=0.08]{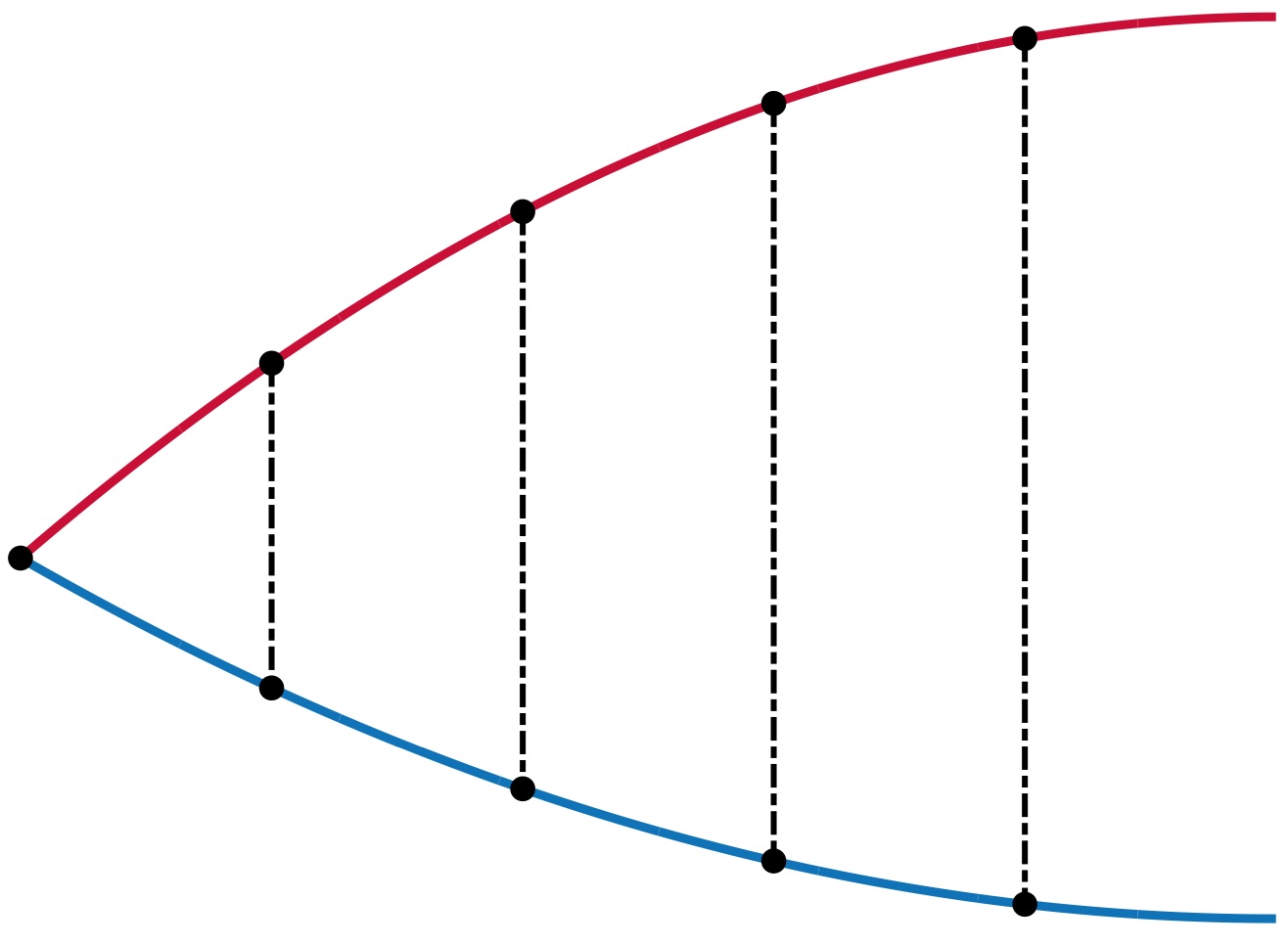}
	\put(-14,29){{\color{black} $P_0$ }}
	\put(15,70){{\color{red} $W^{\sta}_{\mathrm{loc}}(P_0)$ }}
	\put(15, -5){{\color{blue} $W^{\uns}_{\mathrm{loc}}(P_0)$ }}
	\put(65,77){{\color{black} $f_\gna^{k_1}(Q)$ }}	
	\put(65, -10){{\color{black} $f_\gna^{-k_2}(Q)$ }}	
\end{overpic}
\end{minipage}
\begin{minipage}{{0.32\textwidth}}
\begin{overpic}[scale=0.08]{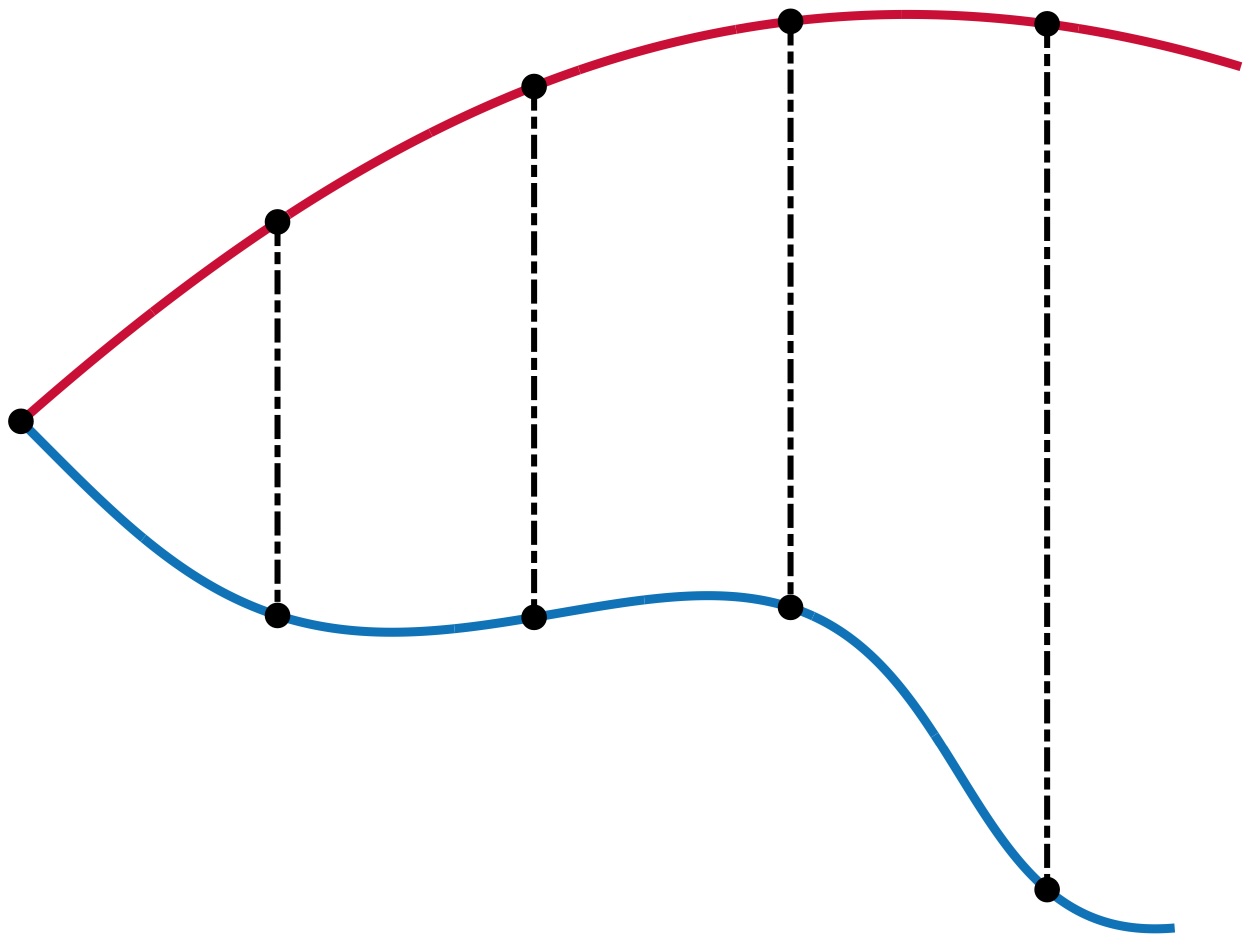}
	\put(-14,40){{\color{black} $P_1$ }}
	\put(15,80){{\color{red} $W^{\sta}_{\mathrm{loc}}(P_1)$ }}
	\put(15, 4){{\color{blue} $W^{\uns}_{\mathrm{loc}}(P_1)$ }}
	\put(65,80){{\color{black} $f_\gna^{k_1+1}(Q)$ }}	
	\put(60, -10){{\color{black} $f_\gna^{-k_2-q+1}(Q)$ }}	
\end{overpic}
\end{minipage}
\begin{minipage} {0.32\textwidth}
\vspace{1.7cm}
\begin{overpic}[scale=0.08]{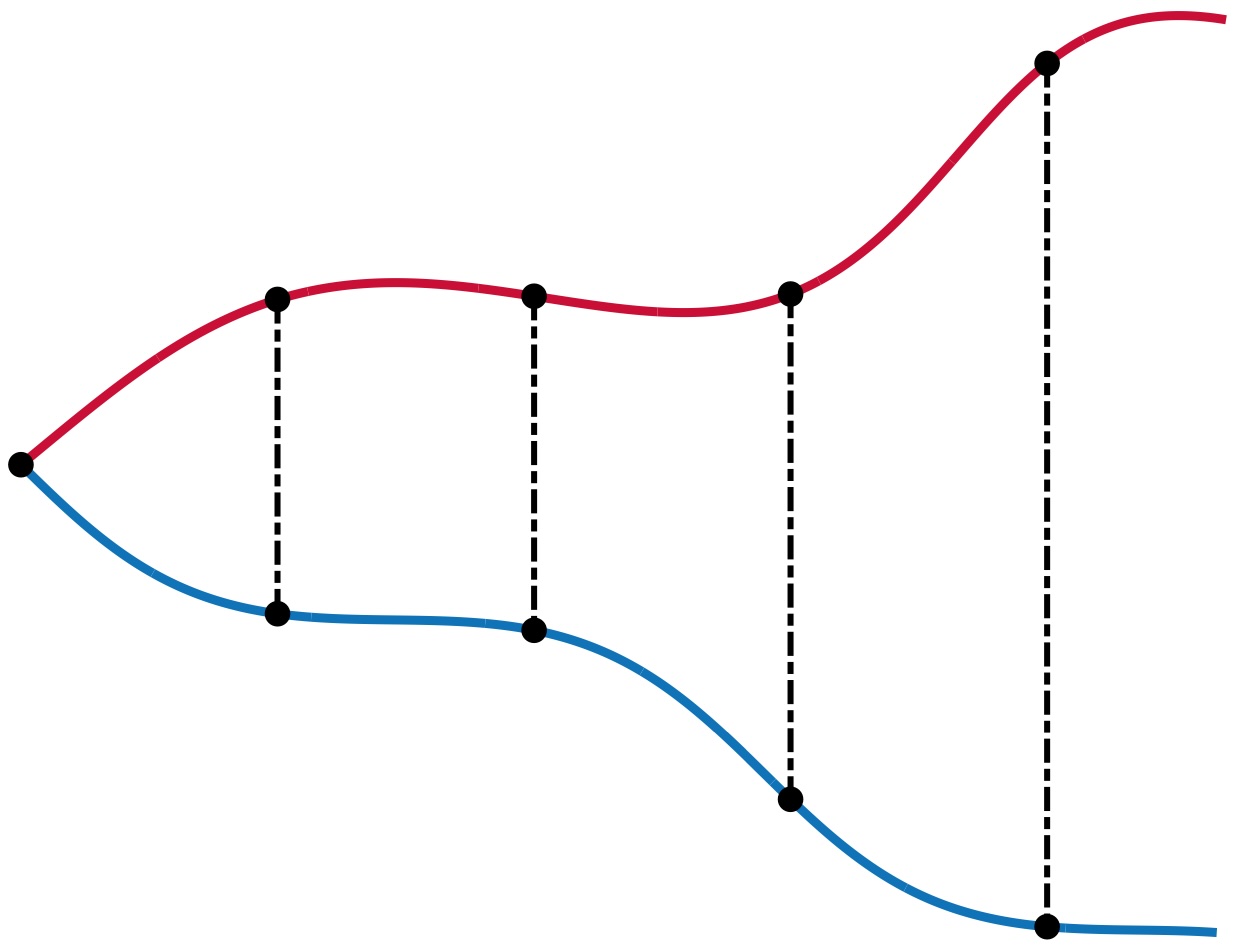}
	\put(-14,35){{\color{black} $P_2$ }}
	\put(13,68){{\color{red} $W^{\sta}_{\mathrm{loc}}(P_2)$ }}
	\put(15, 3){{\color{blue} $W^{\uns}_{\mathrm{loc}}(P_2)$ }}
	\put(65,82){{\color{black} $f_\gna^{k_1+2}(Q)$ }}	
	\put(65, -10){{\color{black} $f_\gna^{-k_2-2q+2}(Q)$ }}	
\end{overpic}
\end{minipage}
\vspace{0.3cm}
\caption{On the local stable and unstable manifolds of every $P_j$, there is no one-fibered points in the orbit of $Q$. 
}
\label{fig: two periodic case}
\end{figure}
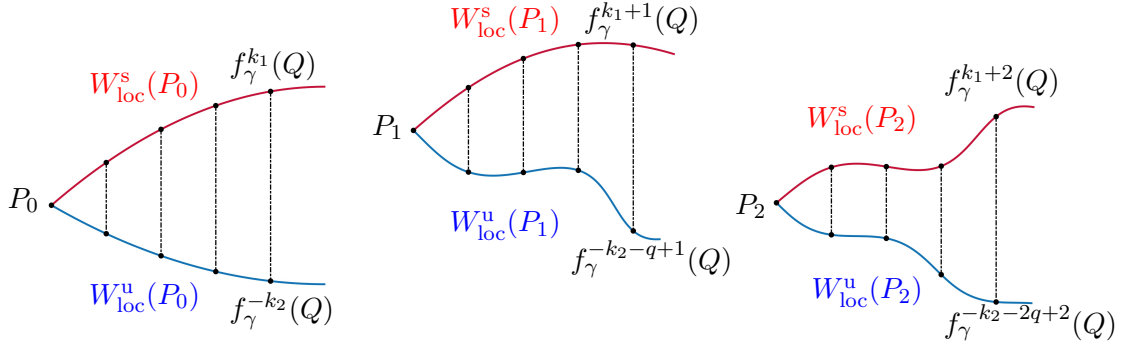

Observe that, since the orbit of $P_0$ is in the Aubry-Mather set, if the branch of the local stable manifold at $P_0$ lies above the branch of the unstable manifold at $P_0$, then the order is the same between the branches of the stable and unstable manifolds at $P_j$, for $j\neq 0$. This comes from the fact that the point $P_0$ does not have conjugate points.

Consider the point $f_\gna^{k_1}(Q)$. It belongs to the local stable manifold of $P_0$ and it lies on the same fiber as $f_\gna^{-k_2}(Q)$, from \eqref{condition not one fibered}.
\begin{enumerate}
\item Assume that $f_\gna^{k_1}(Q)$ lies above $f_\gna^{-k_2}(Q)$. That is, denoting by $p_2$ the projection over the second coordinate:
\[
p_2(f_\gna^{-k_2}(Q))<p_2(f_\gna^{k_1}(Q))\, .
\]

The image of $f_\gna^{k_1}(Q)$ through $f_\gna$, i.e., the point $f_\gna^{k_1+1}(Q)$, belongs to the local stable manifold of $P_1$. Moreover, from \eqref{condition not one fibered everywhere}, it exists $m_1\in \mathbb{Z}$ such that $f_\gna^{k_1+1}(Q)$ lies on the same fiber as $f_\gna^{-k_2-m_1q+1}(Q)$. Since the local unstable manifold of $P_1$ is a graph, because of the choice of $D^s_{P_j}$ and $D^u_{P_j}$, and because of the (positive) twist condition, we can deduce that $m_1\le -1$.

Repeat now the same argument, starting with $f_\gna^{k_1+1}(Q)$ and $f_\gna^{-k_2-m_1q+1}(Q)$. We then deduce that the point $f_\gna^{k_1+2}(Q)$ lies on the same fiber as $f_\gna^{-k_2-m_1q-m_2q+2}(Q)$, for some $m_2\le -1$. Repeating then the argument $q$ times, we get that the point $f_\gna^{k_1+q}(Q)$, which belongs to the local stable manifold of $P_0$, lies on the same fiber as $f_\gna^{-k_2-q\sum_{i=1}^qm_i+q}(Q)$, where all $m_i$ are smaller or equal $-1$.

On the other hand, from \eqref{condition not one fibered}, we know that $f_\gna^{k_1+q}(Q)$ lies on the same fiber as $f_\gna^{-k_2-q}(Q)$. We conclude that
\[
-q\sum_{i=1}^qm_i+q=-q\, ,
\]
that is $\sum_{i=1}^qm_i=2$, which is a contradiction since all $m_i$ are smaller or equal $-1$.

\item The remaining possibility is then that $f_\gna^{k_1}(Q)$ lies below $f_\gna^{-k_2}(Q)$. That is, denoting by $p_2$ the projection over the second coordinate:
\[
p_2(f_\gna^{-k_2}(Q))>p_2(f_\gna^{k_1}(Q))\, .
\]
Repeating the same argument as before, we have now that $f_\gna^{k_1+1}(Q)$ lies on the same fiber as $f_\gna^{-k_2-m_1q+1}(Q)$, for some $m_1\ge 1$. In the end, we obtain that $f_\gna^{k_1+q}(Q)$ lies on the same fiber as
\[
f_\gna^{-k_2-q\sum_{i=1}^qm_i+q}(Q)\qquad\text{with }m_i\ge 1\ \forall i\, .
\]
From \eqref{condition not one fibered}, we deduce that
\[
f_\gna^{-k_2-q\sum_{i=1}^qm_i+q}(Q)=f_\gna^{-k_2-q}(Q)\quad \Rightarrow \quad 2=\sum_{i=1}^q m_i\ge q\, .
\]
This is possible if and only if $q=2$, i.e., if and only if $P$ is a periodic point of period 2.
\end{enumerate}
\end{proof}

The only thing that we still need to prove is the special case of a homoclinic orbit to a $2$-periodic orbit, which is not localised one-fibered. The statement of Proposition~\ref{good property in the bas case:section3} deals with such a case: we are going to prove it right now.

\begin{proof}[Proof of Proposition~\ref{good property in the bas case:section3}]
Recall that, since the orbit of $Q$ is not localised one-fibered, there exists $k_1,k_2\in\mathbb{N}$ such that for every $n\ge 0$ one has
\[
p_1(f^{k_1+2n}(Q))=p_1(f^{-k_2-2n}(Q))\, ,
\]
and, for any fixed $n$, they are the only points on the same fiber. Because of the recalled symmetry of the stable and unstable manifolds, we deduce that, for every $n\ge 0$,
\begin{equation}\label{good prop 1}
\mathcal{I}(f_\gna^{k_1+2n}(Q))=f_\gna^{-k_2-2n}(Q)\, .
\end{equation}
From \eqref{involution:section3} and \eqref{good prop 1}, we deduce that
\[
f_\gna\circ \mathcal{I}(f_\gna^{k_1}(Q))=\mathcal{I}\circ f_\gna^{-1}(f_\gna^{k_1}(Q))=\mathcal{I}(f_\gna^{k_1-1}(Q))
\] 
but also
\[
f_\gna\circ \mathcal{I}(f_\gna^{k_1}(Q))=f_\gna(f_\gna^{-k_2}(Q))=f_\gna^{-k_2+1}(Q)\, ,
\]
that is
\begin{equation}
\mathcal{I}(f_\gna^{k_1-1}(Q))=f_\gna^{-k_2+1}(Q)\, .
\end{equation}
Applying the same argument by recurrence, we obtain for every $i\ge 0$
\begin{equation}\label{involution at the orbit}
\mathcal{I}(f_\gna^{k_1-i}(Q))=f_\gna^{-k_2+i}(Q)\, .
\end{equation}
Observe that $k_1>-k_2$, thanks to the fact that they correspond to points respectively on the local stable and unstable manifolds, which are different local graphs. Let us consider the following two cases.
\begin{enumerate}
\item If $k_1+k_2$ is odd, then by choosing $i=\frac{k_1+k_2-1}{2}$ and by \eqref{involution at the orbit}, we have that $\mathcal{I}(f_\gna^{k_1- i}(Q))=f_\gna^{-k_2+ i}(Q)$, i.e.,
\[
\mathcal{I}(f_\gna\circ f_\gna^{\ell}(Q))=f_\gna^{\ell}(Q)\, ,
\]
with $\ell= \frac{k_1-k_2-1}{2}$. This is actually not possible. Indeed, we would have a point $f_\gna^{\ell}(Q)$ whose image by the billiard map lies on the same vertical: since the billiard map is the identity on the boundaries $\mathbb{T}\times \{\pm \frac{\pi}{2}\}$, and since $f_\gna^{\ell}(Q)$ is in the interior of the annulus, this would contradict the twist condition.
\item If $k_1+k_2$ is even, then by choosing $i=\frac{k_1+k_2}{2}$ and by \eqref{involution at the orbit}, we obtain that $\mathcal{I}(f_\gna^{k_1- i}(Q))=f_\gna^{-k_2+ i}(Q)$, i.e.,
\[
\mathcal{I}(f_\gna^{\ell}(Q))=f_\gna^{\ell}(Q)\, ,
\]
with $\ell=\frac{k_1-k_2}{2}$. This implies that $f_\gna^{\ell}(Q)$ belongs to the zero section, as required.
\end{enumerate}

We have then proved the existence of a point $f_\gna^{\ell}(Q)$ on the orbit of $Q$ lying on the zero section. 
\end{proof}

\begin{remark}\label{rmk tangent symmetric} Under the conditions of Proposition~\ref{good property in the bas case:section3}, at any point of the orbit of a non transverse homoclinic point $Q$, belonging to $ \mathbb{T}\times \{0\}$,  the stable (and so the unstable) manifold is tangent either to the vertical or to the horizontal direction. Indeed, denoting by $E^{\sta}_x$, resp. $E^{\uns}_x$, the stable, resp. unstable, direction at the point $x=(\ss,\varphi)$, since $\mathcal{I}$ is linear and since the stable and unstable manifolds of $2$-periodic orbits are symmetric with respect to $\mathcal{I}$, we have
\begin{equation}\label{involution and eigenspaces}
\mathcal{I}(E^{\sta}_{f_\gna^{\ell}(Q)})=E^{\uns}_{\mathcal{I}(f_\gna^{\ell}(Q))}=E^{\uns}_{f_\gna^{\ell}(Q)}=E^{\sta}_{f_\gna^{\ell}(Q)}\, ,
\end{equation}
where the last equality comes from the fact that we are assuming that $Q$ is not a transverse homoclinic point. Therefore, the stable direction at $f_\gna^{\ell}(Q)$ is preserved by the involution $\mathcal{I}$: the only possibilities are that $E^{\sta}_{f_\gna^{\ell}(Q)}$ (so also $E^{\uns}_{f_\gna^{\ell}(Q)}$) is either the vertical direction or the horizontal one, as we claimed.
\end{remark}

\section{Compactly supported perturbations and surjectivity of $d\Psi(0)$: proof of Lemmas~\ref{lem: computing the differential} and~\ref{lem: computing the differential ii} and \ref{exhaustive} } 
\label{sec:prooflemmas}

This section is devoted to prove Lemmas~\ref{lem: computing the differential} and~\ref{lem: computing the differential ii} and \ref{exhaustive}.

\subsection{Proof of Lemma \ref{lem: computing the differential}}
Recall that $\Gamma_\nu=\Gamma^*_\nu \circ \mathcal{L}$ and that, therefore,
\[
d \Gamma_\nu(0)\lna= d \Gamma^*_\nu(\mathcal{L}(0))  d \mathcal{L}(0) \lna= d \Gamma^*_\nu(0)  \lac
\]
where $\lac$ is given in~\eqref{lambdaast lambda}.

Since  $Q_0=Q =(\ss_0,\varphi_0)$ is a one fibered point for the billiard map $f_\gna$, so is $\tilde Q=(s_0,\vp_0)$, where $s_0=\sigma (\ss_0)$ for $f_{\ga}$, where we recall that $\ga$ is the arc-lengh parameterization of $\gna$. 
Moreover, if $\lna$ is a function compactly supported in $U_Q$, then $\lac$ will be supported in $\tilde U_{\tilde Q}$, the image of the neighborhood $U_Q$, which will be also one-fibered for the map $f_{\ga}$.

From now on, to avoid a cumbersome notation we denote $f=f_{\ga}$. 
As $\tilde U_{\tilde Q}$ is a one-fibered neighborhood of $\tilde Q =\tilde {Q}_0$, for every $i\neq 0$, the point $\tilde {Q} _i= f^i(\tilde Q) \notin \tilde U_{\tilde Q}\times \left[-\frac{\pi}{2} , \frac{\pi}{2} \right]$. 
Consequently, as $\lna (\ss_i)=0$ for $i\neq 0$, we also have  $\lac(s_i)=0$  for $i\neq 0$, and therefore, by~\eqref{def:Gamma} and~\eqref{def:dGamma}  
\[
\big [d \Gamma^*_\nu(0)\lac \big ](\tilde Q_i) =   0, \;\;\;  D\big [d \Gamma^*_\nu(0)\lac \big ](\tilde Q_i)=0 \qquad i\neq 0,-1
\]
where we have denoted by $D$ the differentiation with respect to $(s,\varphi)$. 
Therefore, by Proposition~\ref{prop:Genecand}, using that $\Gamma^*_\nu(0)=f$ and the chain rule, we obtain that
\begin{align}
d (\Phi_1 \circ \Gamma_\nu)(0) \lna= &  (Df(\tilde Q_0) \mathbf{t}_0 ) \wedge \big [d \Gamma^*_\nu(0) \lac \big ] (\tilde Q_0) +
\mathbf{t}_{0} \wedge \big [d \Gamma^*_\nu(0) \lac \big ] (\tilde Q_{-1}) \notag\\
d (\Phi_2 \circ \Gamma^*_\nu)(0) \lac =   & (Df(\tilde  Q_0) \mathbf{t}_0) \wedge \big [
D [d \Gamma^*_\nu(0) \lac](\tilde Q_0) \mathbf{t}_0 
\big ]  \notag \\ 
& + \mathbf{t}_0 \wedge \big [
D [d \Gamma^*_\nu(0) \lac]( \tilde Q_{-1}) Df^{-1} ( \tilde Q_0) \mathbf{t}_{0} \big ] + \mathbf{w}
\label{formuladphi12Gamma}
\end{align}
with $\mathbf{t}_0$ being the (common) vector tangent of $W^{\uns,\sta}$ at $\tilde Q$ and $\mathbf{w}$ depending on $d\Gamma^*_\nu(0)\lac$, but being independent on $D [d\Gamma^*_\nu(0) \lac]$. 

Using~\eqref{def:Gamma} to compute $\big [d\Gamma^*_\nu(0) \lac](\tilde Q_0)$ and 
$\big [d\Gamma^*_\nu(0) \lac](\tilde Q_{-1})$, 
using that
\begin{equation}\label{propertyM}
M \mathbf{u}\wedge \mathbf{v} = \mathrm{det} (M\mathbf{u},\mathbf{v})= \mathrm{det} (M) \cdot \mathrm{det} (\mathbf{u},M^{-1} \mathbf{v}), \qquad \mathbf{u},\mathbf{v} \in \mathbb{R}^2 , \, M \in \mathcal{M}_{2\times 2},
\end{equation}
and formula~\eqref{formula:Dfsphi} for $Df(s,\varphi)$ which gives $\mathrm{det} Df(\tilde  Q_0)=(\cos \varphi_0) (\cos \varphi_1)^{-1}$,
we obtain that 
\begin{align*}
d (\Phi_1 \circ \Gamma_\nu)(0) \lna= &
(Df( \tilde Q_0) \mathbf{t}_0 ) \wedge \left [A( \tilde Q_0, f( \tilde Q_0)) \left (\begin{array}{c} \lac(s_0) \\ \derlac(s_0) \end{array} \right ) \right ] 
+ \mathbf{t}_0 \wedge \left [B( \tilde Q_0) \left (\begin{array}{c} \lac(s_0) \\ \derlac(s_0) \end{array} \right ) \right ] \\ =&
\mathbf{t}_0 \wedge \left [  \left (\frac{\cos \varphi_0}{\cos \varphi_1} (Df( \tilde Q_0))^{-1} A( \tilde Q_0,f(Q_0)) + B( \tilde Q_0)  \right )\left (\begin{array}{c} \lac(s_0) \\ \derlac(s_0) \end{array} \right )\right ].
\end{align*}
Tedious but straightforward computations lead us to 
\begin{equation}\label{formula for dPsi1}
d (\Phi_1 \circ \Gamma_\nu)(0) \lna= 
\frac{1}{\cos \varphi_0 \cos\varphi_1} \mathbf{t}_0 \wedge  
\mathcal{T}(s_0,\varphi_0,\varphi_1) 
\left (\begin{array}{c} 
\lac(s_0) \\ \derlac(s_0) 
\end{array}\right )
\end{equation}
with $\mathcal{T}$ the matrix defined in~\eqref{def mathcalT}.

Now we compute $d(\Phi_2 \circ \Gamma_\nu)(0)(\lna)$, assuming also that $\lna(\ss_0)=\derlna(\ss_0)=0$, and therefore $\lac(s_0)=\derlac(s_0)=0$. 
From~\eqref{def:dGamma} and letting $f_1 = \pi_s f$, the first component of $f$, we have that 
\begin{equation}\label{first:formulaD}
\begin{aligned}
D[d \Gamma^*_\nu(0) \lac ](\tilde Q_0)  &= -\frac{\dertwolac(s_0)}{\cos \varphi_1} \left ( \begin{array}{cc} \hat{\tau} & 0 \\ \mathcal{K}(s_1) \hat \tau + \cos \varphi_1  & 0 \end{array} \right ), \\   
D [d \Gamma^*_\nu(0) \lna] ( \tilde Q_{-1})  &=  \dertwolac(s_0) \left (\begin{array}{cc} 0 & 0 \\ \partial_s f_1(s_{-1}, \varphi_{-1}) & \partial_\varphi  f_1(s_{-1},\varphi_{-1}) \end{array} \right ).
\end{aligned}
\end{equation} 

We compute now each term in~\eqref{formuladphi12Gamma}. Using again~\eqref{propertyM}
\[
(Df(\tilde Q_0) \mathbf{t}_0) \wedge \big [ D [d \Gamma^*_\nu(0) \lac]( \tilde Q_0) \mathbf{t}_0 \big ] 
= \frac{\cos\varphi_0}{\cos \varphi_1} \mathbf{t}_0 \wedge Df(\tilde Q_0)^{-1} D[d \Gamma^*_\nu(0) \lac ](\tilde Q_0) \mathbf{t}_0.
\]
Notice that, from formula~\eqref{formula:Dfsphi}, we have that 
\begin{align*}
Df(\tilde Q_0)^{-1} D[d \Gamma^*_\nu(0) \lac ](\tilde Q_0) \mathbf{t}_0 & = (\pi_1 \mathbf{t}_0) \dertwolac(s_0) Df(\tilde Q_0)^{-1} \left (\begin{array}{c} -\frac{\hat \tau}{\cos \varphi_1} \\ -\frac{1}{\cos \varphi_1} \left (\mathcal{K}(s_1) \hat\tau + \cos \varphi_1 \right )\end{array}\right ) \\ & = (\pi_1 \mathbf{t}_0) \dertwolac(s_0) \left (\begin{array}{c} 0 \\ 1 \end{array}\right )
\end{align*}
Therefore 
\begin{equation}\label{first:surjective}
(Df(\tilde Q_0) \mathbf{t}_0) \wedge \big [
D [d \Gamma^*_\nu(0) \lac](\tilde Q_0) \mathbf{t}_0 
\big ] =(\pi_1 \mathbf{t}_0)^2  \dertwolac(s_0) \frac{\cos\varphi_0}{\cos \varphi_1}.
\end{equation}
Now we deal with the second term in~\eqref{formuladphi12Gamma}. 
We first note that using that $Df^{-1}(\tilde Q_0) = (Df(\tilde Q_{-1}) )^{-1}$ we have that 
\[
\left (\begin{array}{cc} 0 & 0 \\ \partial_s f_1(s_{-1}, \varphi_{-1}) & \partial_\varphi f_1(s_{-1},\varphi_{-1}) \end{array} \right ) Df^{-1}(\tilde Q_0) = \left (\begin{array}{cc} 0 & 0\\ 1 & 0 \end{array} \right ).
\]
and from~\eqref{first:formulaD}  
\[
\mathbf{t}_0 \wedge \big [
D [d \Gamma^*_\nu(0) \lac](\tilde Q_{-1}) Df^{-1} (\tilde Q_0) \mathbf{t}_{0} \big ] =  \big (\dertwolac(s_0)  \big )\mathbf{t}_0 \wedge \left (\begin{array}{cc} 0 & 0\\ 1 & 0 \end{array} \right ) \mathbf{t}_0 =  \dertwolac(s_0)  (\pi_1 \mathbf{t}_0)^2.
\]
Finally, by~\eqref{first:surjective} and from~\eqref{formuladphi12Gamma} we obtain 
\begin{equation}\label{formula dPsi2}
d(\Phi_2 \circ \Gamma_\nu)(0)\lna  = \dertwolac(s_0)  (\pi_1 \mathbf{t}_0)^2 \left ( \frac{\cos \varphi_0}{\cos \varphi_1} + 1\right )+ \mathbf{w}.
\end{equation}

\subsection{Proof of Lemma \ref{lem: computing the differential ii}}
The proof of this result is analogous to the one of Lemma~\ref{lem: computing the differential}. Indeed, by Definition~\ref{def two fibered}, if $\lna\in \mathcal{C}^\infty_{\mathrm{supp}} (U_Q )$, then, as in Lemma~\ref{lem: computing the differential}, we have $\lac\in \mathcal{C}^\infty_{\mathrm{supp}} (\tilde U_{\tilde Q} )$, were $\tilde Q=(s_0,0)$, with $s_0=\sigma(\ss_0)$. Therefore $\lac(s_i)=0$ for $i\neq \ell, -\ell$ and  using~\eqref{def:Gamma} and~\eqref{def:dGamma}  
\[
\big [d \Gamma^*_\nu(0)\lac \big ](\tilde Q_i)=0, \;\;\;  D\big [d \Gamma^*_\nu(0)\lac \big ](\tilde Q_i)=0 \qquad i\neq \ell ,\ell-1, -\ell, -\ell -1
\]
where $\tilde Q_i = f^{i}(\tilde Q)=(s_i,\varphi_i)$. Recall that $D$ means the differentiation with respect to $(s,\varphi)$. We use now Proposition~\ref{prop:Genecand} together with the chain rule to obtain  
\begin{equation}\label{formuladphi12Gammaell}
\begin{aligned}
d (\Phi_1 \circ \Gamma_\nu)(0) \lna=    &  (Df(\tilde Q_\ell) \mathbf{t}_\ell ) \wedge \big [d \Gamma^*_\nu(0) \lac \big ] (\tilde Q_\ell) +
\mathbf{t}_{\ell} \wedge \big [d \Gamma^*_\nu(0) \lac \big ] (\tilde Q_{\ell-1})  \\ 
&+  (Df(\tilde Q_{-\ell}) \mathbf{t}_{-\ell} ) \wedge \big [d \Gamma^*_\nu(0) \lac \big ] (\tilde Q_{-\ell}) +
\mathbf{t}_{-\ell} \wedge \big [d \Gamma^*_\nu(0) \lac \big ] (\tilde Q_{-\ell-1}) \\
d (\Phi_2 \circ \Gamma_\nu)(0) \lna=   & (Df(\tilde Q_\ell) \mathbf{t}_\ell) \wedge \big [
D [d \Gamma^*_\nu(0) \lac](\tilde Q_\ell) \mathbf{t}_\ell 
\big ]   \\  
& + \mathbf{t}_\ell \wedge \big [
D [d \Gamma^*_\nu(0) \lac](\tilde Q_{\ell-1}) Df^{-1} (\tilde Q_\ell) \mathbf{t}_{\ell} \big ] 
\\ &+ (Df(\tilde Q_{-\ell}) \mathbf{t}_{-\ell}) \wedge \big [
D [d \Gamma^*_\nu(0) \lac](\tilde Q_{-\ell}) \mathbf{t}_{-\ell} 
\big ] 
\\ & + \mathbf{t}_{-\ell} \wedge \big [
D [d \Gamma^*_\nu(0) \lac](\tilde Q_{-\ell-1}) Df^{-1} (\tilde Q_{-\ell}) \mathbf{t}_{-\ell} \big ] + \mathbf{w} 
\end{aligned}
\end{equation}
with $\mathbf{t}_{\ell}, \mathbf{t}_{-\ell}$ being the (common) vector tangent of $W^{\uns,\sta}(\mathcal{P})$ at $f^{\ell}(\tilde Q)$ and $f^{-\ell}(\tilde Q)$, respectively, and $\mathbf{w}$ depending on $d\Gamma^*_\nu(0)\lac$, but being independent on $D [d\Gamma^*_\nu(0) \lac]$. 

Performing the same computations as in the proof of Lemma~\ref{lem: computing the differential}, we obtain the analogous to~\eqref{formula for dPsi1}
\begin{equation}\label{formula for dPsi1 ii}
\begin{aligned}
d (\Phi_1 \circ \Gamma_\nu)(0) \lna= &\frac{1}{\cos \varphi_\ell \cos\varphi_{\ell + 1}} \mathbf{t}_\ell \wedge  
\mathcal{T}(s_\ell,\varphi_\ell,\varphi_{\ell+1}) \left (\begin{array}{c} \lac(s_\ell) \\ \derlac(s_\ell) \end{array}\right ) \\ 
&+\frac{1}{\cos \varphi_{-\ell} \cos\varphi_{-\ell+1}} \mathbf{t}_{-\ell} \wedge  
\mathcal{T}(s_{-\ell},\varphi_{-\ell},\varphi_{-\ell+1}) \left (\begin{array}{c} \lac(s_{-\ell}) \\ \derlac(s_{-\ell}) \end{array}\right )
\end{aligned}
\end{equation}
and when $\lna(\ss_\ell)= \derlna(\ss_{\ell})=0$, we obtain the analogous formula to~\eqref{formula dPsi2}:
\begin{equation}\label{formula dPsi2 ii}   
\begin{aligned}
d(\Phi_2 \circ \Gamma_\nu)(0)\lna  = & \dertwolac(s_\ell)  (\pi_1 \mathbf{t}_\ell)^2 \left ( \frac{\cos \varphi_\ell}{\cos \varphi_{\ell+1}} + 1\right )  \\ 
&+\dertwolac(s_{-\ell})  (\pi_1 \mathbf{t}_{-\ell})^2 \left ( \frac{\cos \varphi_{-\ell}}{\cos \varphi_{-\ell+1}} + 1\right )+ \mathbf{w}.
\end{aligned}
\end{equation}

Now we use the special symmetry properties of this case. Indeed, recall that $\mathcal{I}\circ f = f^{-1} \circ \mathcal{I}$ with $\mathcal{I}$ the involution defined as $\mathcal{I}(s,\varphi)=(s,-\varphi)$ (see~\eqref{eq:involutionI} and Definition~\ref{def two fibered}) and that $\tilde Q_0=\mathcal{I}(\tilde Q_0)$, {where $\tilde Q_0=\tilde Q$}. Then $\tilde Q_i=f^{i}(\tilde Q)$ and $\tilde Q_{-i}=f^{-i}(\tilde Q)$ satisfy $\tilde Q_{-i}=(s_{i},-\varphi_{i}) = \mathcal{I}(\tilde Q_i)$, for $i\in \mathbb{Z}$. In addition, since $\mathbf{t}_{i+1}=Df(\tilde Q_{i})\mathbf{t}_i$, for $i\geq 0$ and $\mathbf{t}_{i+1}=Df^{-1}(\tilde Q_{i}) \mathbf{t}_i$, $i<0$ (see~\eqref{defQti}) we conclude that 
$\mathbf{t}_i = \mathcal{I}(\mathbf{t}_{-i})$ (see also Remark~\ref{rmk tangent symmetric}). 
In particular
\[
(s_{-\ell}, \varphi_{-\ell})=(s_\ell , -\varphi_{\ell}), \qquad \pi_1 \mathbf{t_{-\ell}} = \pi_1 \mathbf{t}_\ell
\]
Using these  facts we deduce the result form~\eqref{formula for dPsi1 ii} and~\eqref{formula dPsi2 ii}.

\subsection{Proof of Lemma \ref{exhaustive}}
For any $k\in \mathbb{Z}$, we introduce 
\[
C_{k}= \frac{1}{\cos \varphi_{k+1}}  \big ( \cos \varphi_{k}+ \cos \varphi_{k+1}\big ) =1 + \frac{\cos \varphi_k}{\cos \varphi_{k+1}}>0
\]
because we recall that, for $k\in \mathbb{Z}$, the point $f_\gna^{k}(Q) \in \mathcal{M}_{\frac{p}{q}} (\Omega) \subset \mathbb{A}_\nu$. Now we choose suitable $\lna_1, \lna_2$. 

We first deal with the case $Q$ one-fibered satisfying $\pi_1 \mathbf{t}_0\neq 0$. We take any $\lna_1 \in \mathcal{C}^\infty_{\mathrm{supp}} (U_Q)$ satisfying $\lna_1(\ss_0)=0$ and $\derlna_1(\ss_0) = 1$. 
Then, by Lemma~\ref{lem: computing the differential}
\[
d(\Phi \circ \Gamma_\nu) (0) \lna_1 =  \big (\pi_1 \mathbf{t}_0 \, C_0 , \mathbf{a} \big )
\]
for some constant $\mathbf{a}$. Secondly, we notice that, if  $\lna_2\in  \mathcal{C}^\infty_{\mathrm{supp}} (U_Q)$ satisfying that $\lna_2 (\ss_0) = \derlna_2 (\ss_0) =0$, again using Lemma~\ref{lem: computing the differential}  
\[
d(\Phi \circ \Gamma_\nu)(0)\lna_2   = \left( 0,\dertwolac_2(s_0)  \big (\pi_1 \mathbf{t}_0 \big )^2 C_0+ \mathbf{w}\right).
\]
We then choose $\lna_2 \in \mathcal{C}^\infty_{\mathrm{supp}} (U_Q)$ satisfying that $\lna_2(\ss_0) = \derlna_2(\ss_0) =0$, and $\dertwolna_2(\ss_0) = |\partial \Omega(\dgnalna)|(\der{\sigma}(\ss_0))^2(1- \mathbf{w} C_0^{-1}\big (\pi_1 \mathbf{t}_0 \big )^{-2}) $ and therefore
\[
d(\Phi \circ \Gamma_\nu) (0) \lna_2 = (0, \big (\pi_1 \mathbf{t}_0 \big )^2 C_0). 
\]
Notice that, since $\pi_1 \mathbf{t}_0 , C_0\neq 0$,                    
\[
\left |\begin{array}{cc} \pi_1 \mathbf{t}_0 \, C_0& 0 \\
\mathbf{a} & \big (\pi_1 \mathbf{t}_0 \big )^2 C_0 \end{array} \right |  =  \big (\pi_1 \mathbf{t}_0 \big )^3 \frac{1}{(\cos \varphi_1)^2} ( \cos \varphi_0 + \cos \varphi_1)^2 \neq 0
\]
and therefore $\{d(\Phi \circ \Gamma_\nu)(0) \lna_1, d(\Phi \circ \Gamma_\nu)(0) \lna_2\} \subset \mathbb{R}^2$ is a basis of $\mathbb{R}^2$.

The case $Q=(\ss_0,0)$ two-fibered follows analogously. Indeed, if $\lna_1 \in \mathcal{C}^\infty_{\mathrm{supp}}(U_Q)$ is such that $\lna_1(\ss_\ell)=0$ and $\derlna_1(\ss_\ell)=1$, by Lemma~\ref{lem: computing the differential ii} and using that $\pi_1 \mathcal{I}(\mathbf{t}_\ell)= \pi_1 \mathbf{t}_\ell$, we obtain
\[
d(\Phi\circ \Gamma_\nu)(0)\lna_1 = (\pi_1 \mathbf{t}_\ell C_\ell + \pi_1 \mathcal{I}(\mathbf{t}_\ell) C_{-\ell}, \mathbf{b} ) = \big (
\pi_1 \mathbf{t}_\ell [C_\ell + C_{-\ell} ] , \mathbf{b} \big )
\]
for some constant $\mathbf{b}$. 

As before, we choose now $\lna_2 \in \mathcal{C}^\infty_{\mathrm{supp}}(U_Q)$ such that $\lna_2(\ss_\ell) = \derlna_2(\ss_\ell)=0$. Then, by Lemma~\ref{lem: computing the differential ii}
\[
d(\Phi\circ \Gamma_\nu)(0)\lna_2 = \big (0, \dertwolac_2(s_\ell)\big (\pi_1 \mathbf{t}_\ell\big )^2 (C_\ell + C_{-\ell}) + \mathbf{w}\big ).
\]
Finally, by choosing $\dertwo{\lna}_2(\ss_\ell) = |\partial \Omega(\dgnalna)|(\sigma'(\ss_\ell))^2(1- \mathbf{w} (C_\ell + C_{-\ell})^{-1} (\pi_1 \mathbf{t}_\ell)^{-2})$, we obtain 
\[
d(\Phi\circ \Gamma_\nu)(0)\lna_2 = (0, \big (\pi_1 \mathbf{t}_\ell \big )^2 (C_\ell + C_{-\ell}))
\]
and the result is proven, as done for the one-fibered case.

\bibliographystyle{alpha}
\bibliography{biblio.bib}  
\end{document}